\documentclass[10pt]{article}
\usepackage[english]{babel}
\usepackage{amsmath}
\usepackage{amsfonts}
\usepackage{amssymb}
\usepackage{amsthm}
\usepackage{amscd}
\usepackage{fullpage}
\usepackage{changepage}
\usepackage{mathtools} 

\usepackage[final]{showkeys}
\usepackage[matrix,arrow,ps]{xy}

\usepackage{hyperref}
\usepackage[sort=def]{glossaries}

\newcommand{\Z}{\mathbb{Z}}
\newcommand{\R}{\mathbb{R}}

\newcommand{\C}{\mathbb{C}}

\newcommand{\U}{\mathbb{U}}
\newcommand{\E}{\mathbb{E}}

\newcommand{\mc}{\mathcal}
\newcommand{\mb}{\mathbb}
\newcommand{\mf}{\mathfrak}
\newcommand{\eps}{\varepsilon}
\newcommand{\ind}{{\bf 1}}

\newcommand{\wt}{\widetilde}

\newcommand{\ra}{\rangle}
\newcommand{\la}{\langle}

\newcommand{\VV}{{\mathcal V}}

\newcommand{\vertiii}[1]{{\left\vert\kern-0.25ex\left\vert\kern-0.25ex\left\vert #1 
    \right\vert\kern-0.25ex\right\vert\kern-0.25ex\right\vert}}

\DeclareMathOperator{\Cov}{Cov}
\DeclareMathOperator{\Var}{Var}
\DeclareMathOperator{\Div}{div}

\DeclareMathOperator{\sgn}{sgn}

\DeclareMathOperator{\osc}{osc}

\newcommand{\m}{\mathbf{m}}

\newtheorem*{thmr}{Theorem}
\newtheorem{thm}{Theorem}

\newtheorem{Def}[thm]{Definition}
\newtheorem*{Defr}{Definition}

\newtheorem{Prop}[thm]{Proposition}
\newtheorem*{Propr}{Proposition}

\newtheorem{Lem}[thm]{Lemma}
\newtheorem{Cor}[thm]{Corollary}

\def\Xint#1{\mathchoice
{\XXint\displaystyle\textstyle{#1}}%
{\XXint\textstyle\scriptstyle{#1}}%
{\XXint\scriptstyle\scriptscriptstyle{#1}}%
{\XXint\scriptscriptstyle\scriptscriptstyle{#1}}%
\!\int}
\def\XXint#1#2#3{{\setbox0=\hbox{$#1{#2#3}{\int}$ }
\vcenter{\hbox{$#2#3$ }}\kern-.6\wd0}}

\def\dashint{\Xint-}

\usepackage{xcolor}

\numberwithin{equation}{section}
\title{Metric growth dynamics in Liouville quantum gravity}

\author{Julien Dub\'edat  \thanks{Department of Mathematics, Columbia University, 2990 Broadway, New York, NY 10027, USA.}  \and Hugo Falconet \thanks{Courant Institute, New York University, 251 Mercer street, New York, NY 10012, USA.}}

\date{\today}

\begin{document}

\maketitle

\begin{abstract}
We consider the metric growth in Liouville quantum gravity (LQG) for $\gamma \in (0,2)$. We show that a process associated with the trace of the free field on the boundary of a filled LQG ball is stationary, for every $\gamma \in (0,2)$.  The infinitesimal version of this stationarity combined with an explicit expression of the generator of the evolution of the trace field $(h_t)$ provides a formal invariance equation that a measure on trace fields must satisfy.  When considering a modified process corresponding to an evolution of LQG surfaces, we prove that the invariance equation is satisfied by an explicit $\sigma$-finite measure on trace fields. This explicit measure on trace fields only corresponds to the pure gravity case. On the way to prove this invariance, we retrieve the specificity of both $\gamma = \sqrt{8/3}$ and of the LQG dimension $d_{\gamma} = 4$. In this case, we derive an explicit expression of the (nonsymmetric) Dirichlet form associated with the process $(h_t)$ and construct dynamics associated with its symmetric part.
\end{abstract}

\tableofcontents

\section{Introduction}

\subsection{Liouville quantum gravity}

Gaussian free fields (GFF) appear as the universal scaling limit of a large class of random discrete surfaces \cite{Kenyon-GFF, Naddaf-Spencer} and play a fundamental role in mathematical physics. Formally,  they are measures on fields $h$ defined on a domain $D$ such that
\begin{equation}
\label{eq:path-integral-GFF}
\rho(dh) \propto \exp \left( - \sigma^{-2} \int_D | \nabla h |^2 d\lambda \right) \mc{D}h
\end{equation}
where $\mc{D}h$ is the formal Lebesgue measure on fields, $\sigma >0$ and $\lambda$ is the Lebesgue measure on $D$. In two dimensions,  they belong to the class of $\log$-correlated Gaussian fields for which the covariance kernel is given (up to a multiplicative constant) by $\E(h(x) h(y)) = - \log | x-y| + O(1)$ and are conformally invariant measures. Furthermore, the field has an important domain Markov property, which plays a role here.

Now, suppose given a metric tensor $ds^2$ on a two dimensional Riemannian manifold $X$. Then, under mild assumptions, \textit{locally}, it can be represented using isothermal coordinates by $ds^2 = \rho (du^2 + dv^2)$ for some smooth $\rho >0$ and the associated conformal factor $\phi$ is given by $\rho = e^{\phi}$. Using the complex coordinate $z= u + i v$, the volume form and distance function are locally given by
$$
e^{\phi(z)} d^2 z \quad  \quad \text{and} \quad \inf_{\pi : x \to y} \int_{\pi} e^{\frac{\phi}{2}} ds,
$$
where the infimum is taken over all piecewise continuously differentiable paths $\pi$ with endpoints $x$ and $y$. In what follows, we are interested in the case where the conformal factor $\phi$ is given by $\gamma h$, where $h$ is a Gaussian free field and $\gamma \in (0,2)$. Since $h$ is a random Schwartz distribution with negative regularity, the exponential $e^{\gamma h}$ only makes sense formally.  The volume form and distance function are then given by
\begin{equation}
\label{eq:Forme-LQG}
 ``e^{\gamma h(z)} d^2 z" \quad\quad \text{and} \quad  ``\inf_{\pi : x \to y}  \int_{\pi} e^{\frac{ \gamma h}{d_{\gamma}}} ds",
\end{equation}
where $d_{\gamma} > 2$ is the almost sure Hausdorff/Minkowski dimension of the associated metric measure space. Liouville quantum gravity (LQG) is the random geometry associated with this metric measure space. It was originally introduced in the physics literature in 1981 by Polyakov  \cite{P81}. In \cite{DS11}, Duplantier and Sheffield gave a rigorous meaning to the volume form $``e^{\gamma h}  d^2z"$ by taking the limit
\begin{equation}
\label{eq:LQG-measure}
\mu_h(dz) =  \lim_{\eps \to 0} \eps^{\gamma^2/2} e^{\gamma h_\eps(z)} \ dz
\end{equation}
where $h_{\eps}(z)$ is the $\eps$-circle average approximation of the field. This is a special case of Gaussian multiplicative chaos \cite{Kahane85, Shamov16, Berestycki17} (a generalization to $\log$-correlated Gaussian fields in any dimension). Moreover, they proved the following coordinate change formula: if $f: D \to D'$ is a conformal map then, almost surely, the push-forward of the measure $\mu_h$ by $f$ is given by
\begin{equation}
\label{eq:measure-cov}
f_{*} \mu_h = \mu_{ h\circ f^{-1} + Q \log |(f^{-1})'|}, 
\end{equation}
where
\begin{equation}
\label{def:Q}
Q = \frac{\gamma}{2} + \frac{2}{\gamma}.
\end{equation}

More recently, for every $\gamma \in (0,2)$, the distance function associated with LQG was constructed in \cite{DDDF19, GM19uniqueness}. It is proved to be the scaling limit of a similar approximation scheme as \eqref{eq:LQG-measure}, called Liouville first passage percolation (LFPP), but with a specific mollification procedure (the heat kernel), which we denote by $h_{\eps}^*$ and a different parameter in the exponential,
\begin{equation}
\label{def:xi}
\xi = \frac{\gamma}{d_{\gamma}},
\end{equation}
where $d_{\gamma}$ was shown to exist before the construction of the distance in  \cite{DZZ18, DG18}. Then, for appropriate normalizing constants $\lambda_{\eps}$ satisfying $\lambda_{\eps} = \eps^{1-\xi Q + o(1)}$, the $\gamma$-LQG metric $D_h$ is given by the following limit
\begin{equation}
\label{eq:Renormalization-Metric}
D_h(x,y) = \lim_{\eps \to 0} \lambda_{\eps}^{-1} \inf_{\pi : x \to y} \int_{\pi} e^{\xi h_\eps^*} ds.
\end{equation}
It is almost surely bi-H\"older with respect to the Euclidean distance and therefore induces the Euclidean topology. However, it is almost surely not a Riemannian metric. It also satisfies, for every conformal map $f : D \to D'$,
\begin{equation}
\label{eq:metric-cov}
f_{*} D_h = D_{ h\circ f^{-1} + Q \log |(f^{-1})'|}. 
\end{equation}
Given \eqref{eq:measure-cov} and \eqref{eq:metric-cov} it is natural to consider two pairs $(D,h)$ and $(D' , h')$ related by a conformal map as
\begin{equation}
\label{eq:conformal-covariance}
h' = h\circ f^{-1} + Q \log |(f^{-1})'|
\end{equation}
as different parametrizations of the same \textit{LQG surface}. This coordinate change formula, sometimes referred to as ``LQG coordinate change" says that this metric measure space depends only on the quantum surface, not on the particular choice of parametrization.

Fine properties of the geodesics and metric balls of this metric measure space have been the focus of an intense direction of research recently. Regarding geodesics, Gwynne and Miller proved a ``confluence of geodesics" phenomenon in \cite{GM19confluence} and geodesic networks were studied in \cite{Gwynne-Geodesics}. Concerning metric balls, the LQG volume of LQG balls of radius $r$ in a compact set was proved to be of order $r^{d_\gamma +o(1)}$ in \cite{AFS20}, a contrast with the LQG volume of Euclidean balls where the exponent depends on location of the center of the ball.  An argument of Miller and Sheffield  \cite[Proposition 2.1]{TBM-charac} shows that the boundary of a filled LQG ball is a Jordan curve.  Gwynne \cite{Gwynne-Boundary} established a formula for the Hausdorff dimension of the boundary of LQG balls in terms of $\gamma$ and $d_{\gamma}$, conditionally on a zero-one law type result which was subsequently proved in \cite{GPS20}. Stronger statements of many of these results have been obtained for the specific value $\gamma = \sqrt{8/3}$  \cite{Angel-geodesics, MQ-strong, LG-measure, LG-stars}, for which a connection with the Brownian map exists (and is discussed below). Here, we are interested in the growth process associated with LQG metric balls for every $\gamma \in (0,2)$.

\subsection{QLE$(8/3,0)$ and the $\sqrt{8/3}$-LQG metric}

\paragraph{Quantum Loewner Evolutions.}  Schramm-Loewner evolutions (SLE) are a one-parameter family of random non-self-crossing and conformally invariant curves in the plane, usually indexed by the parameter $\kappa$, describing the roughness of the curves. They were introduced by Schramm \cite{Schramm} as a combination of stochastic calculus and of Loewner's  theory of the evolution of planar slit domains. This family describes the scaling limit of interfaces of some discrete statistical physics models at criticality, such that percolation (with SLE$_6$) and the Ising model (with SLE$_3$). Chordal SLEs are defined in simply connected domains of the complex plane, with prescribed starting point and endpoint on the boundary. Radial SLE curves have one fixed boundary point and one fixed interior point.  A relation between level lines of the GFF and SLE$_4$ was established in \cite{SS-dgff} and further relations between SLE and the GFF have been studied in \cite{Dub-SLE-GFF, Hadamard-Izyurov, MS-imaginery}.

In \cite{MS16}, Miller and Sheffield constructed a family of random growth models, called Quantum Loewner Evolutions (QLE),  describing an evolution of triple $(K_t, \nu_t, \mathfrak{h}_t)$ where  $(K_t)$ are compact sets of the unit disk growing inward ($K_0$ being the unit circle), $(\nu_t)$ are probability measures on the circle and $(\mathfrak{h}_t)$ are  harmonic functions  on the disk. A natural growth model starting from the origin grows outward toward infinity, but by applying a conformal inversion the growth target becomes the origin. Roughly speaking, the construction of QLE relies on using SLE$_\kappa$ as an exploration method of a quantum surface ($\Phi, \mb{D}$) where $\Phi$ is a specific GFF with an $\alpha$-singularity at the origin, namely $\alpha \log | \cdot |^{-1}$ where $\alpha = \frac{\kappa+6}{2\sqrt{\kappa}}$ for $\kappa > 1$, and $\mb{D}$ is the unit disk.

More precisely, it is constructed as a subsequential limit of a continuous approximation. An SLE$_{\kappa}$ is grown starting from a boundary point of the unit circle $\mb{U}$ sampled according to a specific boundary LQG measure ($e^{- \sqrt{\kappa}^{-1} \mf{h}_0^{\delta}}$) and grows inward, targeting the origin for $\delta$ units of capacity time. $(K_t^{\delta})$ is the associated growing hull, for $t < \delta$. It is associated with a family of conformal maps $g_t^\delta : \mb{D} \backslash K_t^\delta \mapsto \mb{D}$ such that $g_t^\delta(0)=0$ and $(g_t^\delta)'(0) = e^t$. $\mf{h}_t^\delta$ represents the harmonic extension of the (formal) values of the field $\Phi$ on the boundary of the component of $\mb{D} \backslash K_t^\delta$ containing $0$, when mapped back to $\mb{D}$ with $g_t^\delta$ and using an LQG change of coordinates.  Namely, $\mf{h}_t^\delta$ is the harmonic part of $\Phi \circ (g_t^\delta)^{-1} + Q \log |((g_t^{\delta})^{-1})'|$. Here,  $Q = 2/\gamma + \gamma/2$ with $\gamma = \min (\sqrt{\kappa}, \sqrt{16/\kappa})$. At ``time" $\delta$, one uses the LQG coordinate change formula and this process, namely sampling a boundary point (now with $e^{- \sqrt{\kappa}^{-1} \mf{h}_\delta^\delta}$), then an SLE$_{\kappa}$ and uniformizing after $\delta$ units of capacity time, is then iterated. This procedure is described in Figure \ref{fig:approx}. In this approximation, the image $\xi_t^\delta = g_t^\delta(\gamma_t^\delta)$ of the the tip $\gamma_t^\delta$ of the SLE$_\kappa$  is a Brownian motion on the unit circle. This QLE approximation is generated by a sequence of independent Brownian motions $(\xi_{kt}^\delta)_{t \in [0,\delta)}$. 

\begin{figure}[ht]
\centering
\includegraphics[scale=1]{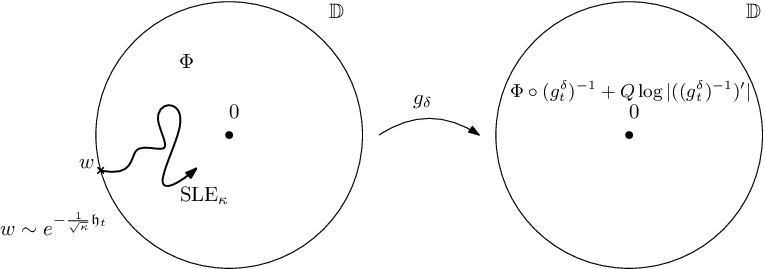}
\caption{QLE approximation}
\label{fig:approx}
\end{figure}

The key observation is that this process is stationary in (capacity) time, namely
$$
\Phi \circ (g_t^\delta)^{-1} + Q \log |((g_t^{\delta})^{-1})'| \overset{(d)}{=} \Phi
$$
seen as field modulo constant. In particular, the harmonic extension $\mf{h}_t^{\delta}$ of the boundary values of the field  is stationary in $t$. Furthermore,  for every $\delta \in (0,1)$,  for fixed $t \geq 0$, the distribution of $\mf{h}_t^{\delta}$ is explicit.

This gives in fact an approximation $(\zeta_t^{\delta}, g_t^{\delta}, \mf{h}_t^{\delta})$ where the boundary probability measure $\zeta_t^{\delta}$ only consists of a Dirac mass at a point sampled from $e^{- \sqrt{\kappa}^{-1} \mf{h}_t^{\delta}}$ at times $t=0, \delta, 2\delta$, ... and a Dirac mass at the location of the running Brownian motion at intermediate times. Miller and Sheffield proved in \cite{MS16} that as $\delta \to 0$, there exist subsequential limits $(\zeta_t, g_t, \mf{h}_t)$ that satisfy the triangle of maps given in Figure \ref{fig:triangle}, with $\zeta_t = \nu_t$. For a precise statement of this tightness result, we refer the reader to \cite[Section 6]{MS16}.

\begin{figure}[ht]
\centering
\includegraphics[scale=1]{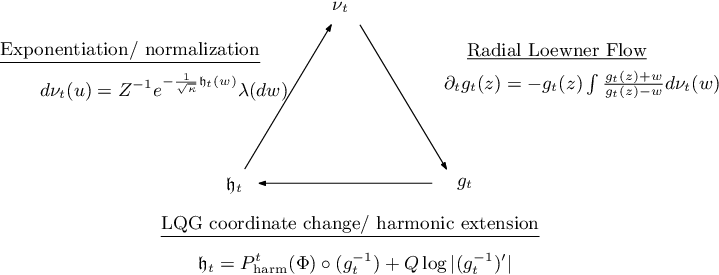}
\caption{QLE dynamics. The map from $\mf{h}_t$ to $\nu_t$ is defined for Lebesgue typical times.}
\label{fig:triangle}
\end{figure}
When $\kappa = 6$, this process was then used to construct the $\sqrt{8/3}$-LQG metric in \cite{MS15b, MS16a, MS16b} without using the renormalization procedure \eqref{eq:Renormalization-Metric}.  We discuss this in further details below.

\paragraph{$\sqrt{8/3}$-LQG metric.} We introduce here additional background and references to the literature associated with the special case $\gamma = \sqrt{8/3}$ but none of this is used anywhere else below.

The metric associated with the specific value $\gamma = \sqrt{8/3}$ was constructed earlier and in a completely different manner. Underlying its construction is the relation between SLE and quantum surfaces pioneered by Sheffield in \cite{Sheffield2016}. Subsequent works with Duplantier and especially with Miller culminated in the construction of this metric and its equivalence with the Brownian map (the scaling limit of uniform random planar maps on the sphere \cite{LeGall13, Miermont13}). To achieve this, they developed a theory of quantum surfaces (variant of the free field $h$ with an associated volume form $e^{\gamma h}$ but without a metric) including some surfaces without boundary: the quantum cone $\mb{C}$ and the quantum sphere $S^2$, and some surfaces with a boundary: the quantum wedge $\mb{H}$ and the quantum disk $\mb{D}$. The later ones are additionally equipped  with a boundary length measure $e^{\frac{\gamma}{2} h}$. These surfaces have nice geometric interpretations (e.g., the cone is the local limit of an LQG surface near a $\gamma$-typical point and the sphere can be obtained by ``pinching off a bubble" from a quantum cone). Furthermore, they obtained an explicit description of the associated fields.

Quantum surfaces have their equivalent versions in the Brownian geometry describing the scaling limits of uniform random planar maps, seen as metric measure spaces (graph distance, counting measure) where the convergence is w.r.t. the Gromov-Hausdorff-Prokhorov topology. The analogue objects are the Brownian plane, sphere, half-plane and disk \cite{LeGalltakagi}. They also have another and equivalent interpretation in terms of Liouville Quantum Field Theory (LQFT)  \cite{DKRV16, AHS17}. This latter perspective generalizes to quantum surfaces on complex tori \cite{HRV18} and on higher genus surfaces \cite{GRV16b}. This different approach was recently used to prove integrability results conjectured from Conformal Field Theory \cite{DOZZ, Bootstrap}.

The relations between SLE and LQG can be motivated by their discrete counterparts \cite{ang03, Angel-Schramm, Angel-Curien, Curien-Glimpse}. When ``gluing" two independent uniform infinite half plane quadrangulations, one gets a new uniform infinite random quadrangulation together with a self avoiding path. The associated scaling limit corresponds to $\sqrt{8/3}$-LQG with an independent SLE$_{8/3}$ \cite{SAW-LQG}. One can also consider site percolation on a uniform infinite half plane triangulation and the associated cluster separating path. This gives rise to a non-self-crossing curve which cuts some ``holes". The distribution of the surfaces on these holes is explicit and given by independent uniform triangulations of the disk with prescribed boundary lengths. The associated scaling limit corresponds to $\sqrt{8/3}$-LQG with an independent SLE$_{6}$ and the surfaces cut out form a Poisson point processes of quantum disks associated with an explicit intensity measure on boundary lengths. We emphasize  that the field on the boundary of  the surfaces cut-out is understood well enough to make sense of the boundary length $e^{\frac{\gamma}{2}h}$. In \cite{sphere}, Miller and Sheffield describe the exploration of a whole-plane SLE$_6$ from one $\sqrt{8/3}$-typical point to another one on the $\sqrt{8/3}$-LQG sphere.  They rigorously proved that the holes cut out by an SLE$_6$ are given by a Poissonian collection of quantum disks, that the law of the region which contains the target point of an SLE$_6$ is equal to that of a quantum disk weighted by its quantum area and that the law of the tip of an SLE$_6$ is distributed according to the quantum length measure of the boundary of the unexplored region.

This last point implies that a certain ``reshuffling operation" is the analogue of the Eden model on a $\sqrt{8/3}$-LQG sphere. Roughly speaking, first passage percolation on uniform triangulation induces at large scale a distance proportional to the graph distance \cite{Curien-LeGall}. As such, one expects that macroscopic behaviour of first passage percolation on $\sqrt{8/3}$-LQG  induces the $\sqrt{8/3}$-LQG metric. Let us focus on the Eden model on a uniform infinite planar triangulation.  Start the Eden exploration from a root triangle. Pick an edge uniformly on the boundary and explore the adjacent triangle. Iterate this process. The holes cut out are conditionally independent triangulations of the disk given their boundary lengths. The boundary length of the infinite component is a Markov chain. This is similar to the percolation model described above (in term of the distribution of the surfaces cut out and of the boundary length process) but with ``tip-forgetting" (in the Eden case, the growth occurs all along the boundary, in the percolation case the exploration follows the interface separating percolation clusters). The analogue construction in the continuum is a reshuffling of SLE$_6$ curves. With the tip of the SLE$_6$ distributed according to the quantum length measure of the boundary of the unexplored region, it is natural to ``reshuffle" the SLE$_6$ exploration, namely to sample a point according to the boundary length measure and an SLE$_6$ from this point for a certain amount of time $\delta$ and to repeat this process. Letting $\delta \to 0$ gives a growth process where the growth naturally occurs all along the boundary: this is the QLE$(8/3)$ process. The time parametrization used to define the QLE process here is different than the capacity time mentioned above and is intrinsic to the LQG surface, they call it the ``quantum natural time parametrization". The distance from an LQG typical point $x$ to another LQG typical point $y$ is then defined to be the amount of time it takes QLE to grow from $x$ to $y$. The proof that this is almost surely equal to the amount of time it takes QLE to grow from $y$ to $x$ is the bulk of \cite{MS15b}. Finally, they use another parametrization: with $X_t$ the LQG length of the complementary component of the QLE, they introduce the ``quantum distance time"  such that if a point is reached at ``quantum natural time" $t$, its ``quantum distance time" is $\int_0^t \frac{du}{X_u}$.  The quantum natural parametrization is the  analogue of parameterizing a percolation growth by the number of edges traversed. This time change is the analogue of adding edges at a rate proportional to boundary length.  Once this metric is constructed, they proved a characterization of the Brownian map \cite{TBM-charac} and used it to conclude on the equivalence between $\sqrt{8/3}$-LQG and the Brownian map. 

We refer the reader to the surveys \cite{Miller-ICM, MOT-Survey, Sheffield-survey} for additional details and an extensive list of references.

\subsection{Infinitesimal generator, invariance equation and pure gravity Dirichlet form}

We present here informal statements of our main results and outline the main proof ideas. The precise statements require the introduction of additional notations and are postponed to the relevant sections of the paper. We explain the nature and meaning of these statements in this introduction.

\paragraph{Stationary growth process, generator of the LQG metric growth, and motivations.}  We consider a  full plane GFF $\phi_{\C}$, seen as $\sigma$-finite measure ($\phi_{\mb{C}} =\phi_{\mb{C}}^0 + c$ , where $\phi_{\mb{C}}^0$ is a whole-plane GFF modulo constant and c is a “sample” from the Lebesgue measure on $\mb{R}$), $B_t$ the geodesic ball centered at $0$ of radius $e^t$ for a $(\xi,Q)$ metric and denote by $B_t^{\bullet}$ the filling of $B_t$ (the union of $\overline{B_t}$ and the set of points in the complex plane $\mb{C}$ which are disconnected from $\infty$ by $\overline{B_t}$). Furthermore, we denote by $\hat g_t$  the map uniformizing the complement of $B_t^{\bullet}$ with $\hat g_t(z)=az^{-1}+o(1)$ for $z\gg 1$, $a>0$. Using the LQG change of coordinates and the Markov property of the free field,  we decompose 
\begin{equation}
\label{eq:beta=0}
\phi_\C\circ \hat g_t^{-1}+Q\log|(\hat g_t^{-1})'| = \phi_t+\bar{h}_t,
\end{equation}
where $\phi_t$ is a Dirichlet GFF on $\mb{D}$ and $\bar{h}_t$ is an independent harmonic function on $\mb{D}\backslash \{ 0\}$. We are particularly interested in the process $\bar{h}_t$. We also consider the following decomposition
\begin{equation}
\label{eq:beta/=0}
\phi_\C\circ \hat g_t^{-1}+Q\log|(\hat g_t^{-1})'| - t/\xi = \phi_t+\tilde{h}_t.
\end{equation}
The processes $(\bar{h}_t)$ and $(\tilde{h}_t)$ encode the boundary values of the restriction of the GFF on the boundary of metric balls after uniformization, or their harmonic extension. In this paper, we are interested in their dynamics. The growth process is naturally measurable with respect to $\phi_{\mb{C}}$ (as the metric is) and the field $\phi_t$ represents the randomness of the domain to be explored after time $e^{t}$, it therefore encodes the noise driving these processes.

In Section \ref{sec:stationarity-martingales}, in particular in Proposition \ref{prop:stat}, we will see that $(\phi_t,\tilde{h}_t)$ in \eqref{eq:beta/=0} is stationary (as a $\sigma$-finite measure). We use a uniform notation with general parameters $(\alpha, \chi, \beta)$ corresponding to the decomposition $\phi_\C\circ \hat g_t^{-1}-\chi \log|(\hat g_t^{-1})'| - \beta t = \phi_t+h_t + \alpha \log | \cdot |$ where $h_t$ is a harmonic function on $\mb{D}$. Above, the $\alpha$-singularity is given by $\alpha = -2 Q$. When considering the LQG metric associated with the field $\phi_\C + \omega \log | \cdot|$, then $\alpha = -2Q - \omega$. The relationship between $h$ and $\tilde{h}$ is given by $\tilde{h}_t = h_t  + \alpha \log |\cdot|$ with $\alpha = -2Q -\omega$.


In order to describe the generator of the process $(h_t)$, we  consider martingales associated to this growth process (with the natural filtration of the growth).  The expression is obtained by using the vanishing of the drift of local martingales and the calculation involves the variation of the Green function of perturbed domains (Hadamard variation, Lemma \ref{lem:hadamard}) and variation of conformal maps (Loewner equation).  The expression  involves various integrals over the disk as we integrate against test functions in the bulk. It also includes a boundary measure $\mu$ (i.e., a Borel measure on the unit circle) which is the driving measure of the Loewner-Kufarev equation. Here, since we do not use the parametrization by capacity but rather a parametrization using LQG distances, this is a priori not a probability measure.

We emphasize that we will see the trace field $h = h_0$ as a harmonic function. In particular, natural test functionals are of the form $\int_{\mb{D}} h f d\lambda$ for smooth functions $f$ with compact support in the bulk. These can be seen as testing against the trace field $h$ on $\mb{U}$ via \eqref{eq:adjoint-int} below. An associated natural but rather weak topology is that of local uniform convergence  in $\mb{D}$. The following definition makes sense for a general  Borel measure $\mu$ on $\mb{U}$.

\begin{Defr}[Definition \ref{def:generator}] We consider functionals of harmonic functions $h$ that take the form
$F(h)=\psi(\int_{\mb D}f_1 h d\lambda,\dots, \int_{\mb D}f_n h d\lambda)$, where $\psi \in C_c^{\infty}(\mb{R}^n)$, $f_i \in  C_c^{\infty}(\mb{D} )$. For any such $f_i$, we set $p_i := f_i^* := H^* f_i$ where $H^*$ is the adjoint of the Poisson kernel (see Section \ref{sec:notation}). We define an operator ${\mc L}_{\alpha, \chi, \beta}$ on the set of test functionals of the form 
\begin{equation}
{\mc L}_{\alpha, \chi, \beta} F := \sum_i b(p_i)\psi_i+\frac 12\sum_{ij}\sigma(p_i,p_j)\psi_{i,j} 
\end{equation}
where $b$ and $\sigma$ are given by
$$
\sigma(p,q)=  4\pi^2  \int_{\mb{U}} pqd\mu
$$
and
$$
b(f^*) =  \int_{\mb D} h(D_\mu f)d\lambda-2\pi\alpha\int_{\mb U}f^*d\mu+\chi\int_{\mb D}f\Re(L'_\mu) d\lambda -\beta \int_{\mb U} f^*d\lambda,
$$
 and where
$$
 D_\mu f(z)   = -2\int_{\mb U}\Re\left(\partial_z\left(z\frac{z+w}{z-w}f(z)\right)\right)\mu(dw), \qquad L_\mu(z) :=-\int_{\mb U}z\frac{z+w}{z-w}\mu(dw).
$$
\end{Defr}

In principle, the knowledge of stationary properties and of an exact expression of the generator of the dynamics help to understand the underlying invariant distribution: in our case the distribution of the free field on the boundary of a metric ball. This can be motivated by the following finite dimensional analogy. Suppose known that the Ornstein-Uhlenbeck dynamics $dX_t = - X_t dt + \sqrt{2} dB_t$  leaves invariant the standard Gaussian (and suppose we do not know the p.d.f. of the standard Gaussian). Given that the generator is $\mc{L} f(x) = f''(x) - x f'(x)$, one obtains with $g = f'$ and $X$ distributed as a standard Gaussian variable $\mc{N}(0,1)$, the condition $\E(g'(X)) = \E(X g(X))$. Now, this relation encodes all information about   $\mc{N}(0,1)$ and characterizes it.  In infinite dimensions, when considering instead the Euclidean growth, it is possible to carry this scheme, this leads to the dynamics \eqref{eq:conc-circles} and the invariant measure is of the form \eqref{eq:field-circle}. However, in the present setting finding by this method an explicit distribution seems to be a difficult problem.

Of course, in the case of $\mc{N}(0,1)$, once one knows the explicit form of the invariant measure, a simple proof of $\E(g'(X)) = \E(X g(X))$ is by integration by parts. For that reason, we look for a  ``good candidate" $\rho(dh)$ satisfying a condition $\int \mc{L} F(h) \rho(dh) = 0$. Formally, if $h \sim \int e^{-V(h)} \mc{D}h$, one would like to use an infinite dimensional integration by parts to retrieve the relation $\int \mc{L} F(h) \rho(dh) = 0$, or, assuming that this relation holds, an equation that the potential $V$ must satisfy. Here and below, $\mc{D}h$ formally stands for an infinite dimensional analogue of the Lebesgue measure, which is translation invariant. Although it doesn't make sense, for a quadratic potential $V$, it usually falls in the classical framework of Gaussian measure in Banach or Hilbert spaces.

\paragraph{A solution of the invariance equation.}
We consider in Section \ref{sec:Trace-fields-and-IBP} the $\sigma$-finite measure $\rho_c$ corresponding to the canonical  log-correlated Gaussian field on $\mb{U}$ (defined in \eqref{eq:field-circle}) with an infinite measure on zero modes and the GMC measure $\mu$ which is  well-defined  for $\rho_c$ a.e. $h$,
\begin{equation}
\label{eq:potential-v}
 \quad \rho_c(dh) = e^{- \frac{1}{4\pi} \int_{\mb{D}} | \nabla H h |^2 d\lambda - c \int_{\mb{U}} h d\lambda} \mc{D}h \quad \text{and} \quad  \mu = e^{-\xi h}.
\end{equation}
The first term is a path integral representation of $\rho_c$ and $\mc{D}h$ stands formally for the infinite dimensional Lebesgue measure on boundary fields on the unit circle $\mb{U}$.  Furthermore, $H h$ is the harmonic extension of $h$ on the disk. Note that $\int_{\mb{D}} | \nabla H h |^2 d\lambda$ is exactly the $H^{1/2}(\mb{U})$ norm.

The measure $\mu = e^{-\xi h}$ is the Gaussian multiplicative chaos measure (GMC) associated with $h$, with parameter $\xi$ (for more on GMC, see Section \ref{subsec:gff}). We discuss why this is a natural measure to consider in this setup in Section \ref{sec:discussion-boundary-measure}. We emphasize here that it is not known that the Loewner measure of the metric growth is given by such a multiplicative chaos measure. Additional discussions are postponed to that section.

Given the explicit form of $\mc{L}_{\alpha, \chi, \beta}$ in $\eqref{def:L}$ with this specific boundary measure $\mu$, which is well-defined on the class of cylindrical bulk test functions for $\rho_c$ a.e. $h$, we look for explicit measures on trace fields that satisfy $\int \mc{L}_{\alpha, \chi, \beta} F(h) \rho(dh) = 0$. In one of our main theorem (Theorem \ref{thm-invariance}), we prove the following.
\begin{thmr}[Theorem \ref{thm-invariance}] Consider the generator $\mc{L}_{\alpha, \chi, \beta}$ as defined above together with $\rho_c(dh)$ and $\mu$ defined in \eqref{eq:potential-v}.  If $F$ is a cylindrical bulk test functional, the following invariance condition holds
\begin{equation}
\label{eq:invariance-condition-intro}
\int \mc{L}_{\alpha, \chi, \beta} F(h)  \rho_c(dh) = 0,
\end{equation}
as soon as the following relations are satisfied:
\begin{equation}
\label{eq:sufficient-condition}
2\xi + (2\xi)^{-1} = -\chi, \quad - 2\pi c = \chi - \alpha, \quad \xi^2 = (2\pi c)^2, \quad \beta = 0.
\end{equation}
\end{thmr}
In the natural LQG setup associated with \eqref{eq:beta=0} where the growth is considered from a $\gamma$-typical point (a $\omega = - \gamma$ singularity), we have $(\alpha,\chi,\beta) = (-2Q+\gamma, -Q,0)$. In this case, the first condition in \eqref{eq:sufficient-condition} gives $\frac{\gamma}{2}+\frac{2}{\gamma} = Q = -\chi  = 2\xi +(2\xi)^{-1}$ so $2\xi = \frac{\gamma}{2}$ or $2\xi = \frac{2}{\gamma}$. With $\xi = \gamma/d_{\gamma}$, the former case gives $2\xi = \gamma/2$, so $d_{\gamma} = 4$. The later one gives $d_{\gamma} = \gamma^2$ which can be excluded by Ang's bound \cite{Ang} ($d_{\gamma} \geq 2 + \gamma^2/2$ so $\gamma^2 \geq 4$ which is impossible when $\gamma \in (0,2)$).  The second and third conditions give $Q - \gamma = \chi - \alpha = -2\pi c = \pm \xi$. Since $\xi = \frac{\gamma}{4}$, this implies $Q = \frac{5}{4} \gamma$ or $Q = \frac{3}{4} \gamma$. With $Q > 2$ and $\gamma \in (0,2)$, we exclude the second possibility. So, $Q = \frac{5}{4} \gamma$ which implies $\gamma^2 = 8/3$ and, together with $d_{\gamma} = 4$, this is consistent with the known formula $d_{\sqrt{8/3}} = 4$. The Hausdorff dimension of the Brownian map was obtained in \cite{LeGall07} and transferred to $\sqrt{8/3}$-LQG by the equivalence \cite{MS15b}.  We emphasize that this calculation does not provide another proof of the formula $d_{\sqrt{8/3}} = 4$.  That would require additional steps, among which the uniqueness of measure satisfying this invariance, and the fact that $\mu = e^{- \xi h}$ in the context of LQG metric growth. For more on this and another approach, see Section \ref{sec:discussion-boundary-measure}.

It is natural to consider other terms in the expression of $\rho_c(dh)$ in \eqref{eq:potential-v}, such as the total mass of a GMC measure, but we have not managed to find a measure on fields $h$ satisfying an invariance equation when $\mu = e^{-\xi h}$ with $\xi > 0$ outside of the pure gravity case.

The proof of the invariance \eqref{eq:invariance-condition-intro} naturally relies on (infinite dimensional) Gaussian integration by parts since they characterize the boundary measure $\rho_c(dh)$ in \eqref{eq:potential-v}. However, given the definition of the generator and the form of the boundary measure $\rho_c(dh)$, this is not trivial nor immediate and requires in particular boundary localization of the terms appearing in the definition of $\mc{L}$. This is the content of Section \ref{sec:boundary-localization-kernelV} whose main result is the following proposition. 
\begin{Propr}[Proposition \ref{Prop:gen-loc}]
For a bulk cylindrical test function $F(h)=\psi(\int_{\mb D}f_1 h d\lambda,\dots, \int_{\mb D}f_n h d\lambda)$,  and with $p_i = f_i^*$, we have
\begin{equation}
{\mc L}_{\alpha, \chi, \beta} F = \sum_i b(p_i)\psi_i+\frac 12\sum_{ij}\sigma(p_i,p_j)\psi_{i,j} 
\end{equation}
where $\sigma(p,q)=  4\pi^2  \int_{\mb{U}} pqd\mu$ and, with the kernels $V_{f^*}$ from Section \ref{sec:boundary-localization-kernelV},
$$
b(f^*) =  \int_{\U^2}V_{f^*}(w,w') \partial_n Hh(w')\mu(dw)d\lambda(w')-2\pi \chi \int_{\mb{U}} \partial_n H p_i d\mu+2\pi(\chi-\alpha)\int_{\mb U}f^*d\mu-\beta \int_{\mb U} f^*d\lambda.
$$
\end{Propr}

Furthermore, the invariance is also based on a non-trivial cancellation of several terms. These are observed in Section \ref{sec:invariance-condition}, which contains both Theorem \ref{thm-invariance} and its proof.

\paragraph{An explicit expression of a Dirichlet form in the pure gravity case.}
From Section \ref{sec:dirichlet-forms} to the end of the paper, we suppose here that we are in the pure gravity case, i.e. $(\gamma,\xi,Q)= (4/\sqrt{6},1/\sqrt{6},5/\sqrt{6})$, for which the  equations in \eqref{eq:sufficient-condition} are satisfied. We use then the notation $\mc{L}$ instead of $\mc{L}_{\alpha, \chi,0}$. The main theorem of this section is Theorem \ref{thm:dirichlet-form}, in which we derive  an  expression of the Dirichlet  form $\mc{E}$ associated with the infinitesimal generator $\mc{L}$. The theorem  states the following.
\begin{thmr}[Theorem \ref{thm:dirichlet-form}]
For test functions $F$ and $G$ of the form  $\varphi(\int_{\mb D}f_1h d\lambda,\dots, \int_{\mb D}f_nh d\lambda)$
where $f_i \in C^{\infty}_c(\mb D \setminus \{ 0 \})$ with at least one $f_i$ with non-vanishing mean and $\varphi$ compactly supported, we have 
\begin{align}
\label{eq:Dirichlet-Form}
\mc{E}(F,G):= \int F (-\mc{L} G) d\rho = & 2\pi^2 \int \langle DF, DG \rangle_{L^2(\mu)} d\rho + 2\pi^2  \iint_{\mb{U}} \left( \widetilde{\widetilde{DF} DG} - \widetilde{DF \widetilde{DG}} \right) d\mu d\rho  \\
& + 2\pi^2 \xi  \int \left( \int_{\mb{U}} DF d\lambda \cdot G - \int_{\mb{U}} DG d\lambda \cdot F \right) | \mu | d\rho.  \nonumber
\end{align}
\end{thmr}
In the above theorem, $| \mu |$ is the total mass of the boundary GMC measure $e^{-\xi h}$, $\wt{u}$ is the harmonic conjugate of $u$ (see Section \ref{sec:harmonic-conjugate}), and $DF$ denotes the $L^2(\lambda)$ gradient of $F$,  characterized by $\langle DF, p \rangle = D_p F$ where $D_p F$ is the Fr\'echet derivative of $F$ in the direction $p$, where $p$ is in the Cameron-Martin space of $h$. 

We can decompose the Dirichlet form with its symmetric and anti-symmetric parts as follows
\begin{align*}
\label{eq:symmetric-part}
\mc{E}(F,G) := \int F (- \mc{L} G) d\rho =  \tilde{\mc{E}}(F,G) + \check{\mc{E}}(F,G)
\end{align*}
where the symmetric part is given by
\begin{equation}
 \tilde{\mc{E}}(F,G) = 2\pi^2 \int \langle DF, DG \rangle_{L^2(\mu)} d\rho
\end{equation}
and the antisymmetric part is given by
\begin{align}
\label{eq:antisymmetric-part}
\check{\mc{E}}(F,G) & = 2\pi^2  \iint_{\mb{U}} \left( \widetilde{\widetilde{DF} DG} - \widetilde{DF \widetilde{DG}} \right) d\mu d\rho + 2\pi^2\xi  \int \left( \int_{\mb{U}} DF d\lambda \cdot G - \int_{\mb{U}} DG d\lambda \cdot F \right) | \mu | d\rho.
\end{align}

This result generalizes the invariance equation we obtained. We provide two proofs: the first one in Section \ref{sec:dirichlet-expression} relies on the machinery developed in Section 4. The second one, which is not rigorous (but can be made with some additional efforts) and done in Section \ref{sec:dirichlet-other-proof}, is based on a non-trivial observation that some terms in the generator can be written under divergence form. We include it as this way of calculating might be easier to study or guess what happens away from $\gamma = \sqrt{8/3}$.

\paragraph{Metric growth from an SPDE perspective.}  

The Dirichlet form formalism is a useful theory to prove the existence of weak solution of stochastic partial differential equations (SPDE). It is particularly useful when the form is symmetric \cite{FOT11} and some extensions to non-symmetric forms are considered as well in \cite{Ma-Rockner}.  We discuss in Section \ref{sec:symmetric-part} (Proposition \ref{prop:weak-solution}) how to construct a weak solution associated with the symmetric part of the Dirichlet form, using a formal change of variables. 

For the stochastic heat equation (SHE), the stochastic quantization of $\phi^4_2$ \cite{Albeverio-Rockner, DaPrato-Tubaro} or $\exp(\Phi)_2$ \cite{Hoshino-L2, Hoshino-L1, G18} and for the stochastic Ricci flow \cite{DS19}, the invariant measure of the dynamics is constructed using the Gaussian free field as a building block and it is possible to construct weak solutions using the formalism of symmetric Dirichlet forms. We gather in the table below these dynamics together with their invariance measure and their Dirichlet form, by order of complexity. The $\xi = 0$ case of our Dirichlet form is similar to the SHE in the sense that the Laplacian is replaced by another operator (the Dirichlet-to-Neumann operator, introduced in Section \ref{sec:Dir-to-Neu}). Furthermore, our symmetric Dirichlet form has a similar structure as the one of the Stochastic Ricci flow. We note that making sense of $\Phi^4_2$ and $\exp (\Phi)_2$ takes some work. See, e.g., \cite[Section 9]{Hairer-LN} in the case of $\Phi^4_2$. The $\exp (\Phi)_2$ model is a specific case of Gaussian Multiplicative Chaos and we refer the reader to \cite{Berestycki17}. In all these cases, the integration by parts to obtain the symmetric Dirichlet form from the expression of the generator is straightforward and is essentially a one line computation. We note that in some of these cases, strong solutions are in fact known to exist and the proof generally follows a Da Prato-Debussche argument, which goes back to \cite{DaPrato}.

\noindent
\begin{table}[!ht]
\begin{adjustwidth}{-0.5cm}{}
\begin{tabular}{| l | l | l | c | }
\hline
 & Dynamics & Invariant measure & Dirichlet form \\ \hline 
 &  & & \\ 
S.H.E.   & $ \partial_t h = \Delta h + \sigma \xi$ & $  \rho_0(dh) = e^{ - \sigma^{-2}  \int h (-\Delta h) d\lambda } \mc{D}h$ & $\frac{\sigma^2}{2} \int \langle DF, DG \rangle_{L^2(\lambda)} d\rho_0$  \\ 
 & & & \\ \hline
& & & \\ 
S.Q. of $\Phi_2^4$  & $\partial_t h = \Delta h - :h^3: +  \sigma \xi $ &  $ \nu(dh) = e^{ - \sigma^{-2} \int \frac{1}{2} :h^4: d\lambda } \rho_0(dh) $ & $\frac{\sigma^2}{2} \int \langle DF, DG \rangle_{L^2(\lambda)} d\nu$ \\ 
 & & & \\ \hline
& & & \\ 
S.Q. of $\exp(\Phi)_2$  & $\partial_t h = \Delta h  - \alpha  e^{\alpha h } + \sigma \xi$ & $ \nu(dh) = e^{ - 2 \sigma^{-2} \int e^{\alpha h} d\lambda }\rho_0(dh)$ &  $ \frac{\sigma^2}{2} \int \langle DF, DG \rangle_{L^2(\lambda)} d\nu$ \\ 
 &   & & \\ \hline
& $\partial_t \phi = e^{-2 \phi} \Delta \phi - \alpha + \sigma e^{-\phi} \xi_0$ &  & \\ 
S. Ricci flow & or, with $A = e^{2 \phi} A_0$,  & $\nu(d\phi) = e^{ - \sigma^{-2} \alpha |A| }  $  & $\frac{\sigma^2}{2} \int \langle DF, DG \rangle_{L^2(A)}  d\nu$ \\
& $ \partial_t A = 2 \Delta \phi A_0 -2 \alpha A + 2 \sigma e^{\phi} \xi_0 A_0$  & $ \qquad \qquad e^{ - \sigma^{-2}  \int_{\mb{T}^2} \phi (- \Delta \phi) d A_0 } \mc{D} \phi$  &  \\ 
\hline
\end{tabular}
\caption{Some dynamics and associated symmetric Dirichlet forms. The Lebesgue measure is denoted by $\lambda$.}
\label{tab:background-dynamics}
\end{adjustwidth}
\end{table}

{\it An SPDE with a non-trivial anti-symmetric part: stochastic Burgers.} 

 One would naturally like to use the theory of non-symmetric Dirichlet forms \cite{Ma-Rockner} to construct a process $(h_t)$ whose Dirichlet form is given by \eqref{eq:Dirichlet-Form}. In \cite{Ma-Rockner}, this is possible when some \textit{weak sector condition} holds but here, it is not clear whether this is the case. Very roughly speaking, the condition holds when the symmetric part of Dirichlet form can be used to ``control" the anti-symmetric part.  An interesting feature in our framework is that bulk cylindrical test functionals belong to the domain of the generator, but do not seem to belong to the domain of the symmetric part of the generator (this is discussed in Section \ref{sec:symmetric-part}). 

The stochastic Burgers equation is an example of an SPDE for which the anti-symmetric part is non-trivial as well. It can be obtained by starting from the Kardar-Parisi-Zhang (KPZ) equation describing a large class of surface growth models and is given by $\partial_t h = \Delta h + (\partial_x h)^2 + \xi$, $h: \mb{R}^+ \times \mb{T} \to \mb{R}$ (see, e.g., the survey \cite{KPZ-corwin}). To get an invariant probability measure, one moves from KPZ to stochastic Burgers by taking $u=\partial_x h$. Then $\partial_t u = \Delta u + \partial_x u^2 + \partial_x \xi$ and $u$ has white-noise as a stationary solution. An interesting feature of this equation is that it is possible to make sense of the term $``\int_0^t \mc{L} F(u_s) ds"$ appearing in the martingale problem for cylindrical test functions but $\mc{L} F$ alone does not exist for these test functionals \cite{Goncalves-Jara, Gubinelli-Jara, Gubinelli-Perkowski-Generator-Burgers}. 
The stochastic Burgers generator can be decomposed as $\mc{L} = \mc{L}_0 + \mc{G}$ where $\mc{L}_0$ is the Ornstein-Uhlenbeck generator and $\mc{G}$ is the Burgers generator. The symmetric part of $\mc{L}$ is $\mc{L}_0$.
The analysis of the dynamics is done by thinking of $\mc{L}$ as a perturbation of $\mc{G}$: when directly solving the resolvent equation, the starting point is to rewrite $(\lambda-\mc{L})u = f$ as $(\lambda-\mc{L}_0)u = \mc{G}u + f$. 

In our framework, the symmetric part of the generator is just within reach, we can make sense of a process only when studying formally $(e^{\xi h_t})$ and it is not clear whether a process associated with the Dirichlet form can be constructed by starting from the process associated with the symmetric part of the Dirichlet form.

Associated dynamics are given (formally), for a space-time white noise $W$, by
\begin{equation}
\label{eq:weak-sols}
\frac{d}{dt} e^{\xi h_t}  =  \pi \xi (\partial_n H h_t  + \xi ) +2\pi \xi e^{\frac{1}{2} \xi h_t} W(dw,dt)
\end{equation}
The reason why these dynamics describe the symmetric part of the Dirichlet form only at a formal level is explained at the beginning of Section  \ref{sec:symmetric-part} and  in Section \ref{sec-formal-gen}. However, the proposition below, which provides a construction of a weak solutions of \eqref{eq:weak-sols}, is rigorous.

In the following statement, $\mc{M}(\mb{U})$ denotes the space of measure on the unit circle with finite total mass. 
\begin{Propr}[Informal version of Proposition \ref{prop:weak-solution}]
For every $\xi \in (0,1)$, there exists an $\m$-symmetric diffusion  $( (\mu_{t}^{\xi})_{t\geq 0}, (P_{\mu_0})_{\mu_0 \in \mc{M}(\mb{U})})$ on $\mc{M}(\mb{U})$ such that for any smooth function $p$ and $\m$-every $\mu_0 \in \mc{M}(\mb{U})$, under $P_{\mu_0}$, $\mu_0^{\xi} = \mu_0$ and
\begin{equation}
d \int_{\mb{U}} p(w) \mu_t^{\xi}(dw) = \pi \xi \int_{\mb{U}} p (\partial_n H h_t  + \xi ) d\lambda dt +2\pi \xi \left( \int_{\mb{U}} p(w)^2 \mu_t^{\xi}(dw) \right)^{1/2} d\beta_t^p
\end{equation}
where $(\beta_t^p)_{t \geq 0}$ is a standard Brownian motion. 
\end{Propr}
The definition of $\m$ in the statement and the Dirichlet form corresponding to these dynamics are given in the paragraph surrounding \eqref{def:symmetric-dirichlet-form}. More details on how for each $t>0$, $h_t$ is defined as measurable function of $\mu_t^{\xi}$ are given in that section as well.

\paragraph{Acknowledgments.} We would like to thank the referee for their careful reading as well as for their helpful suggestions for improvements.

\section{Preliminaries}
\label{sec:prel}

\subsection{Notation} 

\label{sec:notation}

We will denote by $c$ and $C$ constants whether they should be thought
as small or large. They may vary from line to line and depend on other parameters that are fixed. $\mb{C}$ stands for the complex plane, $\mb{U}$ for the unit circle and  $\mb{D}$ for the unit disc in the plane. If $z\in \mb{C}$, we denote by $\Re(z)$ and $\Im (z)$  the real and imaginery parts of $z$ and by $\bar{z}$ the complex conjugate of $z$.

We denote by $\partial_n$ the inward pointing normal derivative and by $\partial_\theta$ the tangential derivative in the counterclockwise direction. We use $\Delta = \frac{\partial^2}{\partial x^2}+\frac{\partial^2}{\partial y^2}$ and $\Delta_{\mb{U}} = \partial_\theta^2$. We  denote by $\partial_t$ the derivative in time and also write sometimes $\frac{d}{dt}$.  Furthermore, if $\psi \in C^{\infty}(\mb{R}^n)$, we denote by $\psi_i$ its partial derivative with respect to its $i$-th coordinate and by $\psi_{i,j}$ its second order partial derivative with respect to the $i$-th and $j$-th coordinates.

Let $\lambda$ denote the Lebesgue measure (on the plane, disk or circle). We set $\dashint_E f d\lambda = \lambda(E)^{-1} \int_E f d\lambda$.  We denote by $\mc{M}(\mb{U})$ the space of Borel measures with finite total mass on $\mb{U}$ and for  $\mu \in \mc{M}(\mb{U})$, we denote by $|\mu|$ its total mass. 

$L^2(\mb{U})$ and $L^2(\mb{D})$ are the space of $L^2$ integrable functions on $\mb{U}$ and $\mb{D}$ w.r.t. the Lebesgue measure. When we consider another measure $\mu$ instead of $\lambda$, we denote the associated $L^2$ space by $L^2(\mu)$. We use the following notation for the Fourier basis of $L^2(\mb{U})$,
\begin{equation}
\label{def:Fourier}
e_0=\frac 1{\sqrt{2\pi}}, \quad 
e_{2m-1}=\frac{\cos(m\cdot)}{\sqrt\pi}, \quad
e_{2m}=\frac{\sin(m\cdot)}{\sqrt\pi}.
\end{equation}
We also set $H^s(\mb{U}) := \{ p  = \sum_{k} a_k e_k : \sum  |k|^{2s} |a_k|^{2} < \infty  \}$, the Sobolev space with index $s$ on $\mb{U}$.

We denote by $H$ the Poisson kernel given by
\begin{equation}
\label{def:Poisson-Kernel}
H(z,w) = \frac{1}{2\pi} \Re( \frac{w+z}{w-z}) 
\end{equation}
The harmonic extension of a function $h_\partial$ on $\mb{U}$ to ${\mb D}$ is then given by $ (Hh_{\partial})(z)=\int_{\mb U} H(z,w)h_{\partial}(w)d\lambda(w)$. We denote by $H^*$ the adjoint of $H$ w.r.t. the standard $L^2({\mb D})$, $L^2({\mb U})$ norms:
$$
(H^*f)(w)=\int_{\mb D}H(z,w)f(z)d\lambda(z)
$$
so that 
\begin{equation}
\label{eq:adjoint-int}
\langle p,H^*f\rangle_{L^2({\mb U})}=\langle Hp,f\rangle_{L^2({\mb D})}
\end{equation}
We sometimes abuse the notation and write $f^*$ instead of $H^* f$.

\subsection{Harmonic conjugation} 
\label{sec:harmonic-conjugate}

If $f = u + iv$ is a holomorphic function on $\mb{D}$, we say that $v$ is a harmonic conjugate of $u$. The harmonic conjugate of $u$ is unique up to a constant. We denote by $\widetilde{u}$ the harmonic conjugate of $u$ whose value at the origin is zero, if it exists. For an introduction to harmonic conjugation, see Chapter 3 in \cite{bounded-analytic-functions}. 

Let $p\in C^\infty({\mb U})$. Its harmonic extension $Hf$ has a harmonic conjugate on $\mb{D}$ given by, for $z \in \mb{D}$,
$$(\widetilde{Hp})(z)=\frac 1{2\pi}\int_{\mb U}p(w)\Im\left(\frac{w+z}{w-z}\right)d\lambda(w)$$
with $(\widetilde{Hp})(0)=0$. It extends continuously to a smooth function on the boundary:
\begin{equation}
\label{eq:harm-conj-pv}
\tilde p(w)=\frac 1{2\pi}p.v.\int_{\mb U}p(w')\Im\left(\frac{w'+w}{w'-w}\right)d\lambda(w')
\end{equation}
By construction $\tilde{\tilde p}=-p+\dashint_{\mb U}p$. This conjugation has a simple effect on the Fourier basis of $L^2({\mb U})$:
$$
1\mapsto 0, \Re(w^n)\mapsto\Im(w^n), \Im(w^n)\mapsto -\Re(w^n)
$$
and we have the antisymmetry
$$
\langle \tilde p,q\rangle_{L^2({\mb U})}=-\langle p,\tilde q\rangle_{L^2({\mb U})}
$$
Clearly, it also commutes with rotations, so that $\partial_\theta \tilde p=\widetilde{\partial_\theta p}$. Furthermore, writing $P Q = (p + i \tilde{p})(q + i \tilde{q}) = p q - \tilde{p} \tilde{q} + i (\tilde{p} q + p \tilde{q})$ gives the following identity:
\begin{align*}
\widetilde{(\tilde p q+p\tilde q)}&=-(pq-\tilde p\tilde q)+\dashint (pq-\tilde p\tilde q)\\
&=-(pq-\tilde p\tilde q)+\dashint p\dashint q
\end{align*}
In the left-hand side, the conjugation is applied to the sum $\tilde{p}q + p \tilde{q}$.

\subsection{Dirichlet-to-Neumann operator}

\label{sec:Dir-to-Neu}

The Dirichlet-to-Neumann operator, a pseudo-differential operator, is given by the composition $\partial_n \circ H$ which we shorten as $\partial_n H$. Writing the Fourier basis using $\Re(z^n)$ and $\Im (z^n)$, we see that
\begin{equation}
\partial_n H e_i = - \lambda_i e_i, \quad \partial_\theta e_i = -\lambda_i \tilde{e}_i \quad \text{where} \quad \lambda_i: =  \lceil i/2 \rceil
\end{equation}
So the Fourier functions form an eigenvector basis for $\partial_n H$. In particular $(\partial_n H)^2 = - \partial_\theta^2 = -\Delta_{\mb{U}}$. So, ``$\partial_n H = -(-\Delta_{\mb{U}})^{1/2}$" is an instance of the fractional Laplacian. In general, we have the following relations
\begin{align*}
\partial_nHp&=-\partial_\theta\tilde p=-\widetilde{\partial_\theta p}\\
\partial_nH\tilde p&=\partial_\theta p
\end{align*}
Furthermore, if  $p \in C^{\infty}(\mb{U})$, then 
$$
(\partial_nHp)(w) = -\frac 1{2\pi}p.v.\int_{{\mb U}}\partial_\theta p(w')\Im\left(\frac{w'+w}{w'-w}\right)d\lambda(w').
$$

\subsection{Gaussian free fields}

\label{subsec:gff}

Gaussian free fields (GFF) are random distributions (in the sense of Schwartz) that are Gaussian and whose covariance kernel is given by a Green function associated with the Laplacian. Below, we are interested in the whole-plane GFF, its restriction to the unit circle and in GFFs on the disc. In this latter case, we need to prescribe boundary conditions to make sense of $\Delta^{-1}$. We will consider the Dirichlet GFF with zero boundary values and the Neumann  GFF (sometimes called free GFF). In this latter case, the GFF is canonically defined in the space of distributions modulo constants, or seen as a $\sigma$-finite measure with the Lebesgue distribution on the zero modes. We refer the reader to Section 4 in \cite{Dub-SLE-GFF} and to \cite{Scott-GFF, Powell-Werner} for more background on the GFF.  Below, we introduce some notation and basic facts that we need.

\paragraph{GFF on the disk.}  The Dirichlet GFF with zero boundary values on $\mb{D}$ is a Gaussian field whose covariance kernel is given by 
$$
G_{D}(z_1, z_2) = - \log \left| \frac{z_1-z_2}{1-\bar{z}_1 z_2} \right|,
$$
which solves $\Delta G_D(z,\cdot) = -2\pi  \delta_z(\cdot)$ and $G_D(z,\cdot) = 0$ on $\mb{U}$ for $z \in \mb{D}$.

The Neumann GFF on $\mb{D}$ has covariance kernel given by
$$
G_N(z_1, z_2) = -\log | (z_1-z_2) (1-\bar{z}_1 z_2) |, 
$$
a solution of $\Delta G_N(z,\cdot) = -2\pi  \delta_z(\cdot)$ and $\partial_n G_D(z,\cdot) = 1$ on $\mb{U}$. We can decompose the Neumann GFF  $h_N$ as $h_N = h^{\circ} + h^{\partial}$ where $h^{\circ}$ is a Dirichlet GFF with zero boundary value and $h^{\partial}$ has covariance kernel given on $\mb{D}$ by
$$
G_{\partial}(x,y) = \E(h^{\partial}(x),h^{\partial}(y))= - 2\log | 1 - \bar{x} y|.
$$
Note that $h^{\partial}$ can be realized by taking the harmonic extension of a Gaussian field on $\mb{U}$ whose covariance is given for $w,w' \in \mb{U}$ by
$$
\E( h^{\partial}(w) h^{\partial}(w')) = - 2\log | w-w'|.
$$
We also introduce the Green kernel $G = \Delta^{-1}$ on the unit disk
\begin{equation}
\label{eq:def-Green}
G(z_1, z_2) = \frac{1}{2\pi}\log\left|\frac{z_1-z_2}{1-z_1\bar z_2}\right|,
\end{equation}
and we note that $G_D(z_1, z_2) = -2\pi G(z_1,z_2)$.

\paragraph{Whole-plane GFF.}

The ``law" of the whole-plane GFF $\phi_{\mb{C}}$ is given by 
$$
\int F(m+\phi_{\mb{C}}^0) dm \otimes d \phi_{\mb{C}}^0
$$
where $\phi_{\mb{C}}^0$ is a whole-plane GFF modulo constant, normalized so that for $f,g \in H^1(\mb{C})$ with $\dashint f = 0$, $\dashint g = 0$,
$$
\Cov (\langle \phi_{\mb{C}}^0, f \rangle_{L^2} \langle  \phi_{\mb{C}}^0, g \rangle_{L^2}) = \int f(x) (-\log |x-y|) f(y) dx dy.
$$
Here, we see the whole-plane GFF as a $\sigma$-finite measure. 

\paragraph{Boundary field on the circle.} The Neumann GFF $h_N$ with covariance $-\log| w- z|$ in the bulk has covariance kernel given by $-2\log | w-z|$ on the boundary and its restriction to the boundary can be written as
\begin{equation}
\label{eq:field-circle}
h = h_N |_{\mb{U}} =  \sqrt{2\pi} \sum_{m \geq 1}  \frac{e_m}{\sqrt{\lambda_m}} X_m
\end{equation}
where the  $X_m$'s are i.i.d. standard Gaussian. Indeed, for $w = e^{i \theta}$, $z = e^{i \theta'}$, 
\begin{align*}
\E(h(w) h(z)) & = 2\pi \sum_{i=1}^{\infty} \frac{1}{\lambda_i} e_i(w) e_i(z) = -2 \Re\left( - \sum_{m=1}^{\infty}\frac{1}{m} e^{i(\theta-\theta')m} \right)  \\
& = -2 \Re \log (1-e^{i(\theta-\theta')}) = -2 \log |1- \bar{z} w| = -2 \log |w-z| 
\end{align*}
From $\eqref{eq:field-circle}$, one can check that $h$ is almost surely in the Sobolev spaces $H^{-s}(\mb{U})$ for every $s > 0$. This Gaussian field can be represented by using the $L^2(\lambda)$ white noise $W = \sum_{m \geq 1} X_m e_m$ for i.i.d. standard Gaussian $(X_m)$ as follows
$$
h = \sqrt{2\pi} \sum_{m \geq 1} \frac{e_m}{\sqrt{\lambda_m}} X_m = \sqrt{2\pi} (-\partial_n H )^{-1/2} W.
$$

Recall that $H$ denotes the Poisson operator (harmonic extension of $h_\partial$ to ${\mb D}$), and $\partial_nH$ the Dirichlet-to-Neumann operator (inward pointing normal derivative). For $h$ harmonic in ${\mb D}$,
$$\int_{\mb D}|\nabla h|^2 d\lambda = \int_{\mb{D}} | \nabla H h |^2 d\lambda =-\int_{\mb U}h\partial_n Hh d\lambda_{\partial}$$ 
Let $\rho_N$ be the measure on trace fields induced by a Neumann (free) GFF, summed over zero modes, with 2-point function $\sim -\log|x-y|$ in the bulk.  Formally, it is given by
\begin{equation}
\label{def:rho0}
\rho_N(dh) = \rho_0(dh_0) \otimes dm \propto \exp \left( -\frac{1}{4\pi} \int_{\mb{U}} h(-\partial_n H h) d\lambda \right) \mc{D}h
\end{equation}
$h$ is the sum $h = m + h_0$ where $m = \dashint_{\mb{U}} h d\lambda$ is the mean of $h$ and $\rho_0(dh_0)$ is the probability measure associated with \eqref{eq:field-circle}. Finally,  we note that the restriction of the kernel $- \frac{1}{2\pi} G_N(z_1,z_2)$ to ${\mb U}$ inverts $\partial_nH$, namely $\partial_n H G_{\partial}(\cdot,x) = - 2\pi \delta_x(\cdot)$.  

\paragraph{Gaussian integration by parts for $\rho_N$.}

Here we discuss the integration by parts formula for a Gaussian measure. We specify it for $\rho_N(dh) = \rho_0(dh_0) \otimes dm$, the measure on trace fields induced by a Neumann GFF summed over zero modes, formally given in \eqref{def:rho0}.

The Cameron-Martin space of the restriction of the Neumann GFF to $\mb{U}$ is 
$$
\left \lbrace p : \int_{\mb{D}} | \nabla H p |^2 d\lambda < \infty \right \rbrace  = \left \lbrace p : - \int_{\mb{U}} p \partial_n H p d\lambda < \infty \right \rbrace = H^{1/2}(\mb{U})
$$
Indeed,  using the Fourier decomposition of $p$, namely $p=\sum_{k \geq 1} p_k e_k$, we have $- \int_{\mb{U}} p \partial_n H p d\lambda = \sum_{k \geq 0} \lambda_k p_k^2$.  Set 
$$
\langle p,q \rangle_{H^{1/2}(\mb{U})} := \frac{1}{2\pi} \int_{\mb{D}} \nabla H p \cdot \nabla H q d\lambda   = -  \frac{1}{2\pi} \int_{\mb{U}} p \partial_n H q d\lambda
$$
By the Cameron-Martin formula, for $p \in H^{1/2}(\mb{U})$, $t \in \mb{R}$,
$$
\int G(h + t p) \rho_N(dh) = \int G(h) \exp \left( t  \langle h, p \rangle_{H^{1/2}(\mb{U})} - \frac{t^2}{2} \| p \|^2_{H^{1/2}(\mb{U})} \right) \rho_N(dh).
$$
The integration by parts 
\begin{equation}
\label{prel:ibp}
\int D_p G(h) \rho_N(dh) = \int G(h) \langle h, p \rangle_{H^{1/2}(\mb{U})} \rho_N(dh)
\end{equation}
follows by taking the derivative w.r.t. $t$ in the Cameron-Martin formula and by evaluating it at $t=0$. In what follows, we often write  $- \frac{1}{2\pi} \langle p , \partial_n H h \rangle_{L^2(\lambda)}$ instead of $\langle p, h \rangle_{H^{1/2}(\mb{U})}$.

\paragraph{Gaussian multiplicative chaos.} We are interested here in measures on $\mb{U}$ of the form 
$$
M_{\alpha}(h) := e^{\alpha h(w) - \alpha^2\E(h(w)^2)/2} \lambda(dw)
$$ 
for $\alpha \in (-1,1)$, where $\lambda$ is the Lebesgue measure on $\mb{U}$ and where $h$ is a log-correlated field with $-2 \log$ singularity on the diagonal (in particular $h$ given by \eqref{eq:field-circle} or by \eqref{def:rho0}). This falls in the general theory of Gaussian Multiplicative Chaos (GMC) \cite{Kahane85, Shamov16, Berestycki17}. In particular, this object is constructed as the limit of an approximation scheme similar to \eqref{eq:LQG-measure} namely using a regularization (so that the approximating measures are well-defined) and a renormalization (in order to obtain a non trivial limit). GMC measures $e^{\gamma \phi}$ are defined in any dimension $d \geq 1$ with $\log$-correlated Gaussian fields $\phi$ s.t. $\E(\phi(x) \phi(y)) = - \log |x-y| + O(1)$ when $\gamma^2 < 2d$, which is the reason of $\alpha \in (-1,1)$ above. We note that by an application of the Cameron-Martin formula,  if $F$ is some bounded continuous function, then 
\begin{equation}
\label{eq:weighting-mass}
\E   \left ( \int_{\mb{U}}  f(w) M_\alpha(h)(dw)    F(h)  \right )   = \int_{\mb{U}} f(w)\E[    F( (h(w')   +E[h(w') h(w)])_{w' \in \mb{U}}   )   ]  \lambda(dw).
\end{equation}
Furthermore, for $p$ smooth, $D_p \int_{\mb{U}} f(w) M_{\alpha}(h)(dw) = \alpha \int_{\mb{U}} f(w) p(w) M_{\alpha}(h)(dw) $.

\subsection{Radial Loewner chain}

A growth process $(K_t)$ on $\overline{\mb{D}}$ growing inwards from the boundary, namely with $K_0 = \mb{U}$, is associated with simply connected decreasing domains $D_t = \mb{D} \backslash K_t$, $D_0 = \mb{D}$, $0 \in D_t \subset D_s$ for $s<t$, that can be encoded with a family of conformal maps via the Riemann mapping theorem. Under mild conditions, these conformal maps satisfy the Loewner-Kufarev equation, driven by a measure $\nu_t(dw)dt$ on $[0,T] \times \mb{U}$. When considering the LQG metric growth, we use the notation $\mu_t$ to denote a disintegration of this measure at ``time" $t$.  In fact, the Loewner dynamics give a way to reconstruct the growth process from the measure $\nu_t(dw)dt$ by solving for each $z \in \mb{D}$ an ordinary differential equation (ODE) and considering at each time $t$ the set of points whose lifetime is greater than $t$.

The radial Loewner equation is given, for $z \in \mb{D}$, by $g_0(z) = z$ and
$$
\partial_t g_t(z) = - g_t(z) \frac{g_t(z)+\xi_t}{g_t(z)-\xi_t}
$$
where $\xi_t$ is a continuous process on $\mb{U}$. When $\xi_t$ is a Brownian motion on $\mb{U}$ with variance $\kappa$, this corresponds to radial SLE$_\kappa$ growing from a boundary point towards the origin.  The radial Loewner-Kufarev equation is given, for $z \in \mb{D}$, by $g_0(z) = z$ and
\begin{equation}
\label{eq:loewner}
\partial_t g_t(z) = - g_t(z) \int \frac{g_t(z)+w}{g_t(z)-w} \nu_t(dw).
\end{equation}
Here, the Loewner chain is driven by a measure $\nu_t(dw) dt$ on $[0,T] \times \mb{U}$.  This  ODE is well defined up to a random time $T_z$ (the lifetime of the solution $g_t(z)$). Then, set $D_t := \{ z \in \mb{D} ~ : ~ T_z > t \}$, a simply connected domain containing the origin and $K_t: = \mb{D} \backslash D_t$. In the case of SLE$_\kappa$, there exists a continuous process $(\gamma_t)$ with values in $\overline{\mb{D}}$ such that $\mb{D} \backslash K_t$ is almost surely the connected component of $\mb{D} \backslash \gamma_{[0,t]}$ containing the origin. The growth occurs from the ``tip" of this SLE$_{\kappa}$ path $(\gamma)$. In general, the growth is no longer concentrated at a single point and we are mainly interested in this case. In the simple case where $\nu_t(dw) = s(t) \lambda(dw)$, we have $D_t = e^{- 2\pi \int_0^t s(u) du} \mb{D}$. The equation \eqref{eq:loewner} gives in particular $\partial_t \log  g_t'(0) = |\nu_t |$ so the conformal radius of $D_t$ from $0$ is given by $g_t'(0)^{-1} = e^{-\int_0^t |\nu_s | ds}$.  We also see that a change of parametrization of a growth process $(\bar{g}_t)$ into $(g_t) = (\bar{g}_{u(t)})$ reads $\nu_t(dw) = u'(t) \bar{\nu}_{u(t)}(dw)$ for the driving measures. A natural assumption on the measure $\nu$ is  to ask for $t \mapsto |\nu_t|$ to be locally integrable and in this case, one can solve the Loewner-Kufarev equation. In the other direction, for the equation \eqref{eq:loewner} to hold (in integrated form), it is sufficient to assume that $t \mapsto g_t'(0)$ is absolutely continuous.

\section{Infinitesimal generator of the LQG metric growth}

\label{sec:section-gen}

In \ref{sec:stationarity-martingales}, we study how invariance properties of the whole-plane GFF translate in stationarity at the level of metric growth. Then, we consider the natural filtration given by the metric growth and associated martingales. In \ref{sec:generator}, we consider a shift on (bulk field, boundary field, boundary measure) which generalizes the previous framework and whose properties are motivated by these martingales. Finally, we formally compute the infinitesimal generator of the boundary field process (which can also be seen as a process on harmonic functions) in this generalized framework.

\subsection{Stationarity and martingales of a $(\xi,Q)$-metric}
\label{sec:stationarity-martingales}

Let $\phi_{\C}$ be a full plane GFF (summing over zero modes), $B_t$ the geodesic ball centered at $0$ of radius $e^t$ for a $(\xi,Q)$ metric and $B_t^{\bullet}$ the filling of $B_t$  (the union of $\overline{B_t}$ and the set of points in the complex plane $\mb{C}$ which are disconnected from $\infty$ by $\overline{B_t}$).

Let $\hat g_t:U_t\subset\hat\C\setminus B_t\rightarrow{\mb D}$ the map uniformizing the complement of $B_t^{\bullet}$ (i.e., the unbounded component $U_t$ of the complement of $B_t$ in $\hat\C$), with $\hat g_t(z)=az^{-1}+o(1)$ for $z\gg 1$, $a>0$. Let ${\mc F}_t=\sigma(B_t,(\phi_{\C})_{|B_t})$ (or its right continuous completion). The $\sigma$-finite measure $\mu_\C$ induced by $\phi_\C$ has an image measure $\mu_t$ by $\phi_{\C}\mapsto (B_t,(\phi_\C)_{|B_t})$ (also $\sigma$-finite); we then have a disintegration $\mu_\C=\mu_tK$, where $K$ is a Markov transition kernel (sampling a GFF in $\C\setminus B_t$ with boundary condition $(\phi_\C)_{\partial B_t}$). The family $(\mu_t)_{t\in\R}$ can be thought of as entrance laws.

Concretely, decompose $\phi_\C=\phi^\eps+c_\eps$ where $c_\eps=\dashint_{C(0,\eps)} \phi_\C d\lambda$, $C(0,\eps)$ the Euclidean circle of radius $\eps$, so that $\phi^\eps$ induces a probability measure and $c_\eps$ induces the Lebesgue measure. Set $\mu_t^\eps$ to be the image measure of $\phi_\C$ on the set $\{\phi_\C: C(0,\eps)\subset B_t\}$. For $t$ fixed, for almost every $\phi_{\mb{C}}$, when $\eps$ is small enough, $C(0,\eps) \subset B_t$. So, take $\eps\searrow 0$  to obtain $\mu_t$ satisfying the desired disintegration.

\begin{figure}[ht]
\centering
\includegraphics[scale=1]{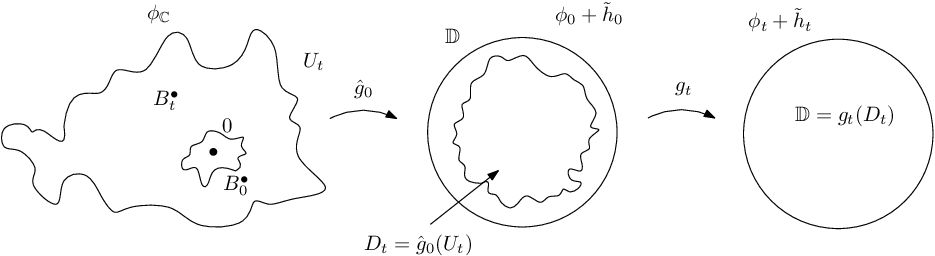}
\caption{Uniformization of the complement of a filled metric ball}
\label{fig:def-notation}
\end{figure}

\paragraph{Stationarity.} We describe the stationarity of (Dirichlet GFF, harmonic function) with respect to the metric growth parametrization. Using the change of coordinates and the Markov property, we decompose
\begin{equation}
\label{def:fields-decompo}
\phi_t+\tilde h_t=\phi_\C\circ \hat g_t^{-1}+Q\log|(\hat g_t^{-1})'|-t/\xi
\end{equation}
where $\phi_t$ is a Dirichlet GFF on $\mb D$ and $\tilde{h}_t$ is harmonic on $\mb D \setminus \{ 0\}$ (so that moving by $\delta r$ in the disk for $e^{\xi (\phi_t+\tilde{h}_t)}ds$ amounts to moving by $e^t\delta r$ in the plane for $e^{\xi \phi_{\mb{C}}}ds$). 

In the following proposition, we will consider the decomposition \eqref{def:fields-decompo} for a field $\underline{\phi_\C}$ instead of $\phi_{\mb{C}}$.
\begin{Prop}
\label{prop:stat}
The relation
$\underline\phi_{t+t_0}+\underline{\tilde h}_{t+t_0}=\phi_t+\tilde h_t$  holds for 
\begin{enumerate}
\item
$\underline{\phi_\C}:=\phi_\C+t_0/\xi$ and
\item
 $\underline{\phi_\C}(z):=\phi_\C(\lambda^{-1} z)$ with $\lambda>0$, $t_0=\xi Q\log\lambda$.
\end{enumerate}
In these two cases, as $\underline{\phi_\C}$ has the same distribution as $\phi_\C$, $(\phi_{t+t_0},{\tilde h}_{t+t_0})$ is distributed as $(\phi_{t},{\tilde h}_{t})$. \\ Furthermore, if $f:[0,\infty)\rightarrow\R$ is smooth and compactly supported, then with
$\underline{\phi_\C}(z) :=\phi_\C(z)+f(D_{\phi_\C}(0,z))$ we have  
\begin{enumerate}
\item
$D_{\underline{\phi_\C}}(0,z)=F(D_{\phi_\C}(0,z)$
where $F(r)=\int_0^re^{\xi f(s)} ds$ and
\item
with $e^{t'}=F(e^t)$, $\underline\phi_{t'}+\underline{\tilde h}_{t'} =\phi_t+\tilde h_t+f(D_{\phi_\C}(0,\hat g_t^{-1}))+(t-t')/\xi.$
\end{enumerate}
\end{Prop}

\begin{proof}
Set $D_t = \hat g_0(U_t)$ and $\hat g_t=g_t\circ \hat g_0$,  where $g_t: D_t \subset{\mb D}\setminus \hat g_0(B_t)\rightarrow {\mb D}$ uniformizes the component of 0 of points at distance $>e^t-1$ of the circle in the $\phi_0+\tilde h_0$ metric. We have, on $D_t$,
\begin{align*}
(\phi_t+\tilde h_t)\circ g_t&=\phi_\C\circ\hat g_0^{-1}+Q\log|(\hat g_t^{-1})'|\circ g_t-t/\xi\\
&=\phi_0+\tilde h_0-Q\log|g'_t|-t/\xi
\end{align*}

Let $\underline{\phi_\C}=\phi_\C+t_0/\xi$. The ball of radius $e^{t+t_0}$ for $\underline{\phi_\C}$ is the ball of radius $e^t$ for $\phi_\C$; then
$$
\underline\phi_{t+t_0}+\underline{\tilde h}_{t+t_0}=\phi_t+\tilde h_t
$$ 
i.e. offsetting the field corresponds to a time shift. In particular, the $\sigma$-finite measure induced on $(\phi_t,\tilde h_t)_{t\in\R}$ is stationary by the definition \eqref{def:fields-decompo} as the decomposition in unique.

Similarly, let $\underline{\phi_\C}(z)=\phi_\C(\lambda^{-1} z)$, $\lambda>0$, $t_0=\xi Q\log\lambda$. The ball of radius $\lambda^{\xi Q}e^t=e^{t+t_0}$ for $\underline{\phi_\C}$ is $\lambda B_t$; the uniformizing map for its complement is $\underline {\hat g}_{t+t_0}=\hat g_t(\lambda^{-1}\cdot)$. Then 
\begin{align*}
(\underline\phi_{t+t_0}+\underline{\tilde h}_{t+t_0})(w)&=\underline{\phi_\C}\circ\underline{\hat g}_{t+t_0}^{-1}(w)+Q\log|(\underline{\hat g}_{t+t_0}^{-1})'|(w)-(t+t_0)/\xi\\
&=\phi_\C\circ \hat g_t^{-1}(w)+Q\log|(\hat g_t^{-1})'|(w)+Q\log\lambda-(t+t_0)/\xi\\
&=\phi_t+\tilde h_t
\end{align*}
i.e. scaling space also results in a time shift.

More generally, let $f:[0,\infty)\rightarrow\R$ be, say, smooth and compactly supported. Let
$$\underline{\phi_\C}(z)=\phi_\C(z)+f(D_{\phi_\C}(0,z)),$$
a Girsanov shift. First, we observe that
$$
D_{\underline{\phi_\C}}(0,z)=F(D_{\phi_\C}(0,z))
$$
where $F(r) :=\int_0^re^{\xi f(s)} ds$. Indeed, $\leq$ holds by taking a geodesic for $D_{\phi_{\mb{C}}}$ and applying Weyl scaling. To bound it from below, namely the $\geq$ assertion, a geodesic from $0$ to $z$ for $D_{\underline{\phi_{\mb{C}}}}$ must cross at least once each boundary of  $D_{\phi_{\mb{C}}}$ ball centered at $0$ and of radius $r \leq D_{\phi_{\mb{C}}}(0,z)$. Using this, we note  that $\underline B_{t'}=B_t$ where $e^{t'}=F(e^t)$ and $\frac{dt'}{dt}=e^{t-t'+\xi f(e^t)}$.
\begin{align*}
\underline\phi_{t'}+\underline{\tilde h}_{t'}&=\underline\phi_\C\circ\underline{\hat g}_{t'}^{-1}+Q\log|(\underline{\hat g}_{t'}^{-1})'|-t'/\xi\\
&=\phi_t+\tilde h_t+f(D_{\phi_\C}(0,\hat g_t^{-1}))+(t-t')/\xi
\end{align*}
so that adding $c$ to $\tilde h$ speeds up growth time by a factor $dt'/dt=e^{\xi c}$.\end{proof}

Remark on singularities. One can also start from $\phi_\C+\omega\log|\cdot|$, $\omega\in\R$ fixed in some range (to ensure finiteness of distances). Then
$$\phi_t+\tilde h_t=\phi_\C\circ \hat g_t^{-1}+\omega\log|\hat g_t^{-1}|+Q\log|(\hat g_t^{-1})'|-t/\xi.$$
Offsetting the field and scaling space still result in a time shift and we still have
$$(\phi_t+\tilde h_t)\circ g_t=\phi_0+\tilde h_0-Q\log|g'_t|-t/\xi.$$
Note that for $z$ near  $0$,
$$\omega\log|\hat g_t^{-1}|(z)=-\omega\log|z|+(reg),  \quad 
Q\log|(\hat g_t^{-1})'|(z)=-2Q\log|z|+(reg),$$
so that, by the removable singularity theorem for harmonic function, 
\begin{equation}
h_t = \tilde{h}_t - \alpha \log |\cdot|
\end{equation} 
is harmonic in $\mb D$ for $\alpha = -2Q - \omega$.

\paragraph{First and second moment martingales.}  Let ${\mc F}_t=\sigma((\phi_\C)_{|B_t})$. Conditionally on ${\mc F}_t$, $\phi_t$ is a Dirichlet GFF in ${\mb D}$. The conditional expectations below refer to a Dirichlet GFF with $0$-boundary values. We have not integrated over the remaining ``randomness" which makes the measure only $\sigma$-finite. We will consider local martingales of the form $\E( \int_{\mb{C}} \phi_{\mb{C}} f d\lambda | \mc{F}_t )$ and $\E( (\int_{\mb{C}} \phi_{\mb{C}} f d\lambda)^2| \mc{F}_t )$  where $f$ is a smooth function with compact support.

{\em First moment martingale.} We have on $U_t$ (see Figure \ref{fig:def-notation} for a reminder of $U_t$),
$$(\tilde h_t-Q\log|(\hat g_t^{-1})'|+t/\xi)\circ\hat g_t=\tilde h_t\circ\hat g_t+Q\log|\hat g'_t|+t/\xi-\omega\log|\cdot|=\E(\phi_\C|{\mc F}_t)$$
and for $t\geq 0$, on $D_t$,
\begin{equation}
\tilde h_t\circ g_t+Q\log|g'_t|+t/\xi=\E(\phi_\C\circ\hat g_0^{-1}|{\mc F}_t)+Q\log|(\hat g_0^{-1})'|+\omega\log|\hat g_0^{-1}|
\end{equation}

{\em Second moment martingale.} We denote by $G_{U_t}$ the Green kernel in $U_t$, with the same normalization as $G$ in Section \ref{subsec:gff} .
For $f$ a test function supported outside of $B_0$, for $t$ small, using the Markov property and averaging over the Dirichlet GFF on $U_t$,
$$
\E((\int_\C \phi_\C f d\lambda )^2 |{\mc F}_t )=(\int_\C \E(\phi_\C|{\mc F}_t)fd\lambda)^2-2\pi\int fG_{U_t}f
$$
Write $f_t=f\circ \hat g_t^{-1}=f_0\circ g_t^{-1}$, $\tilde f_t=f_t|(\hat g_t^{-1})'|^2=\tilde f_0\circ g_t^{-1}|(g_t^{-1})'|^2$; then
\begin{align}\
\E \left( \left( \int \phi_\C f d\lambda \right) ^2|{\mc F}_t \right) &=\left(\int_{\mb D} \left(\tilde h_t-Q\log|(\hat g_t^{-1})'|+t/\xi \right) f_t|(\hat g_t^{-1})'|^2d\lambda\right)^2-2\pi\int_{{\mb D}^2}f_t|(\hat g_t^{-1})'|^2G_{\mb D}f_t|(\hat g_t^{-1})'|^2 \nonumber \\
&=\left(\int_{\mb D} \left( \tilde h_t-Q\log|(\hat g_t^{-1})'|+t/\xi \right) \tilde f_td\lambda\right)^2-2\pi\int_{D_t^2}\tilde f_0G_{D_t}\tilde f_0
\end{align}
with $D_t=g_t^{-1}({\mb D})$ and $G_{D_t}$ is the Green kernel in $D_t$, with the same normalization as $G$ in Section \ref{subsec:gff}.

Altogether, this gives the following lemma.
\begin{Lem}[Local martingales] For $f \in C_c^{\infty}(\mb{C})$ supported outside $B_0$, we have, on $\{B_t \cap \mathrm{Supp}(f) = \emptyset \}$,
\begin{align*}
\E \left( \int \phi_{\mb{C}} f d\lambda ~ | ~ \mc{F}_t \right) & = \int_{D_t} (\tilde h_t\circ g_t+Q\log|g'_t|+t/\xi)\tilde{f}_0 d\lambda, \\
\E \left(  (\int \phi_{\mb{C}} f d\lambda )^2 ~ | ~ \mc{F}_t \right) & = \left(\int_{\mb D} \left( \tilde h_t-Q\log|(\hat g_t^{-1})'|+t/\xi \right) \tilde f_td\lambda\right)^2-2\pi\int_{D_t^2}\tilde f_0G_{D_t}\tilde f_0.
\end{align*}
\end{Lem}

\subsection{Shift and generator on bulk cylindrical test functions}
\label{sec:generator}

In this section we consider a framework similar to the one above but with general parameters instead of the ones coming specifically from the metric growth. The setup is almost the same and briefly described in the next paragraph  (a growth with Loewner measure $\mu_t$ where the boundary values of the GFF are encoded by $h_t$, which satisfy a relation at each time with an independent Dirichlet GFF). We want to define a generator for this growth, for the Markov process $(h_t)$, by choosing a specific set of test functionals and deriving an explicit expression when it acts on this set.

\paragraph{Shift.} We want to define a left shift $\theta_t$,
\begin{equation}
\label{eq:shift-equation}
\theta_t(\phi,h,\mu)=(\phi_t,h_t,\mu_t),
\end{equation}
satisfying the following properties: $\phi_t$ is a Dirichlet GFF on $\mb D$ independent of $h_t$ (with $\sim -\log|x-y|$ covariance in the bulk), $\mu_t$ corresponds to a Loewner flow $(g_t)$ and on $\mb D$,  $\phi_t + \tilde{h}_t = (\phi + \tilde{h}) \circ g_t^{-1} - \chi \log | (g_t^{-1})'| - \beta t$, where $\tilde{h}_t = h_t + \alpha \log |\cdot|$. Equivalently, we want to have on  $D_t = g_t^{-1}({\mb D})$,
\begin{equation}
\label{eq:Shift-equation}
(\phi_t+\tilde h_t)\circ g_t=\phi+\tilde h+\chi\log|g'_t|-\beta t.
\end{equation}
This generalizes the LQG growth cases where $g_t$ is associated with metric balls, \eqref{eq:beta=0} where $\beta = 0$ and \eqref{eq:beta/=0} where $\beta =\xi^{-1}$.

Let $\nu$ denote the law of a  Dirichlet GFF on ${\mb D}$ and $\rho$ a $\sigma$-finite measure on harmonic functions (thought of as harmonic extensions of boundary fields). The shift invariance condition reads
$$\int (F\circ\theta_t)\nu(d\phi)\rho(dh)=\int F\nu(d\phi)\rho(dh)$$
for $F$ a fixed test functional and $t$ small (given $F$). The infinitesimal version formally reads
\begin{equation}
\label{eq:framework-invariance-eq}
\int {\frac{d}{dt}}_{|t=0}(F\circ\theta_t) \nu(d\phi)\rho(dh)=0.
\end{equation}
This is formal in the sense that we have switched the integral and the derivative.  We want to find $\alpha, \chi, \beta$, $\rho(dh)$ and $h\mapsto\mu(h)$ so that this holds. This will be the primary focus of Section \ref{sec-invariance}.

\paragraph{Hadamard variation.} 

Here, we discuss the relation between the driving measure associated to a growth process in Loewner's equation and a speed measure $s(w) \lambda(dw)$. We first mention the Hadamard variation formula \eqref{eq:Hadamard} below. Then, we introduce some notations and preliminaries regarding the Loewner vector field and equation that are used to prove Lemma \ref{lem:hadamard} and in the subsequent sections. An application of the identification is given in Section \ref{sec:discussion-boundary-measure}, regarding the boundary measure appearing in the Loewner equation in the context of the LQG growth.

Start from the unit disk ${\mb D}$ and flow the boundary normally inward locally at speed $s(w)$, $w\in{\mb U}$ so each point $w \in \mb{U}$ moves locally normally inward via an application $S_t$ in a way that $\frac{d}{dt  }_{| t = 0} S_t(w) = - s(w) w$. Let $G_t(z,w)$ be the Green kernel in the new domain $D_t$. We have 
$$
G(z_1,z_2)=G_0(z_1,z_2)=\frac{1}{2\pi}\log\left|\frac{z_1-z_2}{1-z_1\bar z_2}\right|
$$
and at $t=0$, the Hadamard variation \cite{Hadamard} gives
\begin{equation}
\label{eq:Hadamard}
{\frac{d}{dt}}_{|t=0}G_t(z_1,z_2)=\int_{\mb U}\partial_nG(w,z_1)\partial_nG(w,z_2)s(w)d\lambda(w)
\end{equation}
with 
\begin{align}
\partial_nG(w,z)&=\frac{1}{2\pi}\Re\left(-w\partial_w\log\left(\frac{w-z}{1-w\bar z}\right)\right)=-\frac{1}{2\pi}\Re\left(\frac{w}{w-z}+\frac{w\bar z}{1-w\bar z}\right) \nonumber \\
&=-\frac{1}{2\pi}\Re\left(\frac{w}{w-z}+\frac{w^{-1} z}{1-w^{-1}z}\right)=-\frac{1}{2\pi}\Re\left(\frac{w+z}{w-z}\right) = - H(w,z). 
\label{eq:dnG-H}
\end{align}
The Hadamard variation formula is sometimes written with a minus sign in front of the integral in \eqref{eq:Hadamard}. This depends on the convention of the Green kernel and here $G$ corresponds to $\Delta^{-1}$ rather than $(-\Delta)^{-1}$. In this later case, when the domains are decreasing,  Brownian motion spends less time inside and the negative sign naturally arises.

Now, we move forward to the perturbation of the Green function associated with evolving domains in term of the Loewner measure. Associated to a measure $\mu$ on ${\mb U}$, we have a Loewner vector field:
\begin{equation}
\label{eq:Loewner-vector-field}
L_\mu(z) :=-\int_{\mb U}z\frac{z+w}{z-w}\mu(dw)
\end{equation}
and for a family $(\mu_t)$, we have the Loewner equation
$$\partial_tg_t(z)=L_{\mu_t}(g_t(z))$$
so that $g_t$ maps $D_t\subset {\mb D}$ to ${\mb D}$. $L_\mu$ operates on test functions via
\begin{equation}
\label{eq:operator-flow}
(L_\mu f)(z) :=\frac{d}{dt}_{|t=0}f\circ g_t(z)=L_\mu(z)\partial_zf+\overline {L_\mu(z)}\partial_{\bar z}f=2\Re(L_\mu(z)\partial_zf(z))
\end{equation}
if $f$ is real. Furthermore, $g'_t(z)=1+tL_\mu'(z)+o(t)$, $|g'_t(z)|=1+t\Re(L'_\mu(z))+o(t)$.

The following lemma expresses the variation of the Green function in evolving domains in the disk in terms of the Loewner measure.
\begin{Lem}
\label{lem:hadamard}
For every $z_1, z_2 \in \mb{D}$, we have
\begin{equation}
\label{eq:Hadamard-Variation}
{\frac{d}{dt}}_{|t=0}  G_t(z_1,z_2) =  2\pi \int_{\mb{U}} \partial_n G(w,z_1) \partial_n G (w,z_2) \mu(dw).
\end{equation}
\end{Lem}

\begin{proof}
By conformal invariance of the Green's function, $G_t(z_1,z_2)=G(g_t(z_1),g_t(z_2))$. Hence, recalling \eqref{eq:def-Green}, at $t=0$
$$
 {\frac{d}{dt}}_{|t=0} 2\pi G_t(z_1,z_2) =\Re\left(\frac{1}{z_1-z_2}(L_\mu(z_1)-L_{\mu}(z_2))+\frac{\bar z_2}{1-z_1\bar z_2}L_\mu(z_1)+\frac{\bar z_1}{1-z_2\bar z_1}L_\mu(z_2)\right).
$$
This is linear in $\mu$ and for $\mu=\delta_w$, this gives ($z^*=\bar z^{-1}$)
\begin{align*}
 {\frac{d}{dt}}_{|t=0} 2\pi G_t(z_1,z_2)&=-\Re\left(\frac{1}{z_1-z_2}(z_1\frac{z_1+w}{z_1-w}-z_2\frac{z_2+w}{z_2-w})+\frac{\bar z_2}{1-z_1\bar z_2}z_1\frac{z_1+w}{z_1-w}+\frac{\bar z_1}{1-z_2\bar z_1}z_2\frac{z_2+w}{z_2-w}\right)\\
&=-\Re\left(\frac{1}{z_1-z_2}(z_1-z_2+\frac{2w^2}{z_1-w}-\frac{2w^2}{z_2-w})+\frac{z_1}{z_2^*-z_1}\frac{z_1+w}{z_1-w}+\frac{ z_1}{z_2^*-z_1}\frac{w+z_2^*}{w-z_2^*}\right)\\
&=-\Re\left(1-\frac{2w^2}{(z_1-w)(z_2-w)}+\frac{2wz_1}{(z_1-w)(z_2^*-w)}\right)\\
&=-\Re\left(1-\frac{2w^2}{(z_1-w)(z_2-w)}+\frac{2z_1\bar z_2}{(z_1-w)(\bar w-\bar z_2)}\right)\\
&=2\Re\left(\frac 12+\frac{z_1}{w-z_1}+\frac{z_2}{w-z_2}+\frac{z_1z_2}{(w-z_1)(w-z_2)}+\frac{z_1\bar z_2}{(w-z_1)(\bar w-\bar z_2)}\right)\\
&=\Re\left(1+\frac{2z_1}{w-z_1}\right)\Re\left(1+\frac{2z_2}{w-z_2}\right)= \Re(\frac{w+z_1}{w-z_1})\Re(\frac{w+z_2}{w-z_2}) = (2\pi)^2 H(z_1,w) H(z_2,w).
\end{align*}
The conclusion follows from the relation $\partial_n G(w,z) = - H(w,z)$ from \eqref{eq:dnG-H}.
\end{proof}

This identifies  $\mu\leftrightarrow s(w) \frac{\lambda}{2\pi}$. Indeed, by taking $z_1 = z, z_2 = 0$ and integrating against a test function $f$ with a support in the bulk of the unit disk, we get $\int f^* d\mu = \int f^* s \frac{d\lambda}{2\pi}$ which implies $\mu(dw) = s(w) \frac{\lambda(dw)}{2\pi} $.

We briefly mention here some references which considered the Hadamard variation. Some relation between the Loewner and Hadamard variations was examined in \cite{Loewner-Hadamard}.  A Hadamard type formula was proved and  used in the study of couplings of SLE curves with the GFF in the work \cite{Hadamard-Izyurov}. The proof there follows different lines than the computations above. \cite{GFF-Hadamard} used the formula in the sole context of the GFF and in \cite{Viklund-Wang} it was proved and used in the context of  Loewner-Kufarev energy and foliations by chord-arc Jordan curves of the twice punctured Riemann sphere.

\paragraph{Random Loewner chains.}

We consider a process $(h_t,\mu_t)$ taking values in (fields, measures) on ${\mb U}$ and the radial Loewner chain $(g_t)$ associated to $(\mu_t)$. 
We suppose (as motivated by the martingales from Section \ref{sec:stationarity-martingales}) that
$$
t\mapsto \tilde h_t\circ g_t-\chi\log|g'_t|+\beta t=h_t\circ g_t+\alpha\log|g_t|-\chi\log|g'_t|+\beta t
$$
is a local martingale (taking values in the space of harmonic functions in a pointed neighborhood of 0) and that the Markov process $(h_t)$ satisfies for $p \in L^2(\mb U)$ of the form $p = f^*$ for a smooth compactly supported function $f$ on $\mb{D}$,
\begin{equation}
\label{eq:SDE}
d\int_{\mb U} ph_td\lambda=b_t(p)dt+\sigma_t(p)d\beta^p_t
\end{equation}
where $b$ is linear in $p$  and $\sigma^2(p)$ is quadratic in $p$, both being functions of $(h_t,\mu_t)$. Here, $\beta_t^p$ is a standard Brownian motion and we assume that the quadratic variation of the left-hand side of \eqref{eq:SDE} is absolutely continuous with a derivative almost everywhere positive. In the setup of interest which is the LQG metric growth, we have not checked that this holds but we think this is true.

From these assumptions, we would like to find what should be the expression of the generator of $(h_t)$. In the following lemma, we identify under these assumption the drift and quadratic variation at time $0$, in order to motivate Definition \ref{def:generator} below. The calculations involve the Hadamard variation formula and the Loewner equation.

\begin{Lem}
Under the above assumptions, for a test function $f \in C^\infty_c({\mb D})$, the drifts of the local martingales
$$
t\mapsto \int_{D_t}(\tilde h_t\circ g_t-\chi\log|g'_t|+\beta t)\tilde f_0d\lambda, \qquad t\mapsto (\int_{D_t}(\tilde h_t\circ g_t-\chi\log|g'_t|+\beta t)\tilde f_0d\lambda)^2-2\pi  \int_{D_t}\tilde f_0G_{D_t}\tilde f_0
$$
vanish when
\begin{equation}
\label{eq:b-condition}
b_0(f^*) = \int_{\mb D} h_0(D_\mu f)d\lambda - 2\pi\alpha\int_{\mb U}f^*d\mu + \chi\int_{\mb D}f\Re(L'_\mu) d\lambda -\beta \int_{\mb U} f^*d\lambda, \qquad \sigma^2(p)=(2\pi)^2  \int p^2d\mu_0.
\end{equation}
\end{Lem}

\begin{proof}
{\em First moment.} For a test function $f=\tilde f_0\in C^\infty_c({\mb D})$,
$$\int_{D_t}(\tilde h_t\circ g_t-\chi\log|g'_t|+\beta t)\tilde f_0d\lambda=\int_{\mb D} h_t\tilde f_td\lambda+\int_{D_t}(\alpha\log|g_t|-\chi\log|g'_t|+\beta t)\tilde f_0d\lambda$$
for short times, with $\tilde f_t=\tilde f_0\circ g_t^{-1}|(g_t^{-1})'|^2$, we have
\begin{align}
{\frac{d}{dt}}_{|t=0}\log |g_t(z)|&=\Re(\frac{L_\mu(z)}z)=2\pi\int_{\mb U}H(z,w)\mu(dw) \nonumber \\
{\frac{d}{dt}}_{|t=0}\log |g'_t(z)|&=\Re(L'_{\mu_0}(z)) \nonumber \\
{\frac{d}{dt}}_{|t=0}\tilde f_t&=-L_{\mu_0}f-2\Re(L'_{\mu_0})f=:-D_{\mu_0}f 
\label{def:dmu0}
\end{align}
by using for $v : \R \to \C$, $\frac{d}{dt}_{|t=0} \log  |v | =  \Re (\frac{v'(0)}{v(0)})$.

 We want the drift of the local martingale to vanish at time $t=0$, namely
$$
b_0(f^*)-\int_{\mb D} h_0(D_\mu f)d\lambda+2\pi\alpha\int_{\mb U}f^*d\mu-\chi\int_{\mb D}f\Re(L'_\mu) d\lambda +\beta \int_{\mb U} f^*d\lambda=0
$$
where $D_{\mu}f=2f\Re(L'_\mu)+L_\mu f$ and $f^* = H^* f$ for $f \in C_c^{\infty}(\mb{D})$. Taking the derivative in \eqref{eq:Loewner-vector-field} and recalling \eqref{eq:operator-flow}, we get
$$D_{\mu}f(z)=-2\int_{\mb U}\left(\Re\left(\partial_z\left(z\frac{z+w}{z-w}\right)\right)f(z)+
\Re\left(z\frac{z+w}{z-w}\partial_zf(z)\right)\right)\mu(dw)
$$
so that $D_{\mu}f$ admits the following expressions
\begin{equation}
\label{eq:dmu-expressions}
D_{\mu}f=2f\Re(L'_\mu)+L_\mu f = -2\int_{\mb U}\Re\left(\partial_z\left(z\frac{z+w}{z-w}f(z)\right)\right)\mu(dw).
\end{equation}

{\em Second moment.} 
We want the drift of
$$t\mapsto (\int_{D_t}(\tilde h_t\circ g_t-\chi\log|g'_t|+\beta t)\tilde f_0d\lambda)^2-2\pi  \int_{D_t}\tilde f_0G_{D_t}\tilde f_0$$
at $t=0$ to vanish. Using the Hadamard variation formula \eqref{eq:Hadamard-Variation},
$$
{\frac{d}{dt}}_{|t=0}\int_{D_t}f_1G_{D_t}f_2d\lambda  = 2\pi \int_{\mb{D}^2} \int_{\mb{U}}  H(z_1,w) f(z_1) H(z_2,w) f(z_2) \mu_0(dw) d\lambda d\lambda =  2\pi\int_{\mb U}f^*_1f^*_2d\mu_0
$$ 
so we want, with $p=f^* = H^* f$ and with $b = b_0$, $\sigma = \sigma_0$, $\sigma^2(p)=(2\pi)^2  \int p^2d\mu_0$.
\end{proof}

This motivates the following definition of the generator.

\begin{Def}[Generator]
\label{def:generator}
We define an operator ${\mc L}_{\alpha, \chi, \beta}$ on the set of test functionals of the form $F(h)=\psi(\int_{\mb D}f_1 h d\lambda,\dots, \int_{\mb D}f_n h d\lambda)$, where $\psi \in C_c^{\infty}(\mb{R}^n)$ and $f_i \in  C_c^{\infty}(\mb{D} )$,
\begin{equation}
\label{def:L}
{\mc L}_{\alpha, \chi, \beta} F := \sum_i b(p_i)\psi_i+\frac 12\sum_{ij}\sigma(p_i,p_j)\psi_{i,j} 
\end{equation}
with $p_i = f_i^*$,  and
\begin{equation}
\label{eq:drift-diffu}
\sigma(p,q)=  4\pi^2  \int_{\mb{U}} pqd\mu, \qquad b(f^*) =  \int_{\mb D} h(D_\mu f)d\lambda-2\pi\alpha\int_{\mb U}f^*d\mu+\chi\int_{\mb D}f\Re(L'_\mu) d\lambda -\beta \int_{\mb U} f^*d\lambda .
\end{equation}
\end{Def}
Here, for $\psi \in C^{\infty}(\mb{R}^n)$, we denote by $\psi_i$ its partial derivative with respect to its $i$-th coordinate and by $\psi_{i,j}$ its second order partial derivative with respect to the $i$-th and $j$-th coordinates.

In the following section, given the expression of the generator we will be interested in investigating the following additional property: if $\phi$ is a Dirichlet GFF in ${\mb D}$ independent of the rest, we want on $D_t$,
\begin{equation}
\label{eq:invariance-general-parameters}
(\phi+\tilde h_t)\circ g_t-\chi\log|g'_t|+\beta t\stackrel{(d)}{=}\phi_0+\tilde h_0
\end{equation}
from stationarity and strong Markov (this is an identity of $\sigma$-finite measures).

\section{Boundary localization and invariance}

\label{sec-invariance}

In \ref{sec:invariance-general-metric}, we write a formal invariance condition for a general LQG metric or, more generally in the framework considered in the previous section. Then, in \ref{sec:discussion-boundary-measure}, we discuss the driving measure $\mu$ of the Loewner equation. In particular, from that point forward we consider a parameter $\xi >0$ and the GMC $\mu = e^{-\xi h}$. The goal of the rest of the section is to write a more transparent invariance equation associated with the stationarity \eqref{eq:invariance-general-parameters} and to find an explicit invariant boundary measure $\rho(dh)$ associated with \eqref{eq:beta=0}. In \ref{sec:Trace-fields-and-IBP}, we consider a natural family of boundary Gaussian fields and recall their characterizing integration by parts formula. In \ref{sec:boundary-localization-kernelV}, we derive a new expression of the generator $\mc{L}_{\alpha, \chi, \beta}$ in which all terms are boundary localized which we use in \ref{sec:invariance-condition} to prove the invariance of a measure on boundary fields $\rho(dh)$ satisfying several conditions. 

\subsection{Invariance condition for the boundary field}
\label{sec:invariance-general-metric}

To obtain the expression of $\mc{L}_{\alpha, \chi, \beta}$ above, we computed martingales by averaging out the Dirichlet GFF and relied on stochastic calculus. Here, we use instead Gaussian integration by parts to carry the averaging over the Dirichlet GFF.  

\begin{Prop}
\label{Prop:invariance-condition}
We consider a test functional $F(\phi,h)=\psi(\int_{\mb D} f_1(\phi+\tilde h)d\lambda,\dots,\int_{\mb D} f_n(\phi+\tilde h)d\lambda )$ where $f_i\in C^\infty_c({\mb D}\setminus\{0\}),  \psi\in C^\infty_c(\R)$. We also denote  $\tilde\psi(\cdot)=P_\Sigma\psi(\cdot+ \alpha \int f_1 \log |\cdot | d\lambda,\dots)$, $P_\Sigma\psi(x)=\E(\psi(x+\Sigma^{1/2}N))$, where $N$ is a standard Gaussian in $\R^n$ and $\Sigma_{i,j}=-2\pi  \int_{\mb{D}} f_iGf_j d\lambda$  so that $\tilde\psi(\int_{\mb{D}} f_1 h d\lambda,\dots)=P_\Sigma\psi(\int_{\mb{D}} f_1\tilde h d\lambda,\dots)$. Then, the invariance \eqref{eq:framework-invariance-eq} reads for the boundary field
\begin{align}
\label{eq:invariance-cond}
0 = & \int \left( \sum_i\left(\int_{\mb D}(D_{\mu}f_i) hd\lambda+ \alpha \int_{\mb D}(D_{\mu}f_i) \log | \cdot | d\lambda+\chi\int_{\mb D}f_i\Re(L'_\mu)d\lambda-\beta\int_{\mb D}f_id\lambda \right)\tilde\psi_i(\int f_1h d \lambda,\dots, \int f_n h d \lambda) \right. \nonumber \\
& \left. -2 \pi  \sum_{i,j}\left(\int_{\mb D}f_iG(D_{\mu}f_j)\right) \tilde\psi_{ij}(\int_{\mb{D}} f_1h d\lambda,\dots, \int_{\mb{D}} f_nh d \lambda) \right) \rho(dh).
\end{align}
\end{Prop}

\begin{proof}
Recalling the identity \eqref{eq:Shift-equation},
\begin{align*}
\int_{\mb D}f(\phi_t+\tilde h_t)d\lambda &=\int_{D_t}f\circ g_t
(\phi+\tilde h+\chi\log|g'_t|-\beta t)|g'_t|^2  d\lambda 
\end{align*}
so
$$\frac{d}{dt}_{|t=0}\int_{\mb D}f(\phi_t+\tilde h_t)d\lambda=\int_{\mb D}(D_\mu f)(\phi+\tilde h)d\lambda+\chi\int_{\mb D}f\Re(L'_\mu)d\lambda-\beta\int_{\mb D}fd\lambda$$
and  for $F(\phi,h)=\psi(\int_{\mb D} f(\phi+\tilde h)d\lambda)$ where $f\in C^\infty_c({\mb D}\setminus\{0\}), \psi\in C^\infty_c(\R)$,
\begin{equation}
\label{eq:gen-with-dirichlet}
\begin{split}
\mc{G} F(\phi,h) :=  & {\frac{d}{dt}}_{|t=0}(F\circ\theta_t)(\phi,h) \\
  = & \left(\int_{\mb D}(D_\mu f)(\phi+\tilde h)d\lambda+\chi\int_{\mb D}f\Re(L'_\mu)d\lambda
-\beta\int_{\mb D}fd\lambda \right) \psi'(\int_{\mb D} f(\phi+\tilde h) d\lambda)
\end{split}
\end{equation} 
Now, we want to average $\mc{G} F(\phi,h)$ over $\phi$. This requires a Gaussian integration by parts (see, e.g. \eqref{eq:basic-ibp-1} in the appendix), because of the the first term. The averaging over $\phi$ works as follows. 
\begin{align*}
\int \psi(x+\int_{\mb{D}} f\phi d\lambda)\nu(d\phi)&=P_{\sigma^2}\psi(x)\\
\int \left(\int_{\mb{D}} g\phi d\lambda \right)\psi(x+\int_{\mb{D}} f\phi d\lambda)\nu(d\phi)&=-2\pi  \left(\int_{\mb{D}} fGg d\lambda\right)P_{\sigma^2}\psi'(x)
\end{align*}
where $(P_t)$ is the heat kernel on $\R$ and $\sigma^2= \mathrm{Var}(\int_{\mb{D}} f\phi d\lambda)=-2\pi  \int_{\mb{D}} fGf d\lambda$, recalling that $G = \Delta^{-1}$ and $G_D = -2\pi G$. It follows that
\begin{align*}
\int \mc{G} F(\phi,h) \nu(d\phi)= & \int \left(\int_{\mb D}(D_{\mu}f)(\phi+\tilde h)d\lambda+\chi\int_{\mb D}f\Re(L'_\mu)d\lambda-\beta\int_{\mb D}fd\lambda\right)\psi'(\int_{\mb{D}} f(\phi+\tilde h) d\lambda)\nu(d\phi)\\
= & \left(\int_{\mb D}(D_{\mu}f)\tilde hd\lambda+\chi\int_{\mb D}f\Re(L'_\mu)d\lambda-\beta\int_{\mb D}fd\lambda\right)P_{\sigma^2}\psi'(\int_{\mb{D}} f\tilde h d\lambda) \\
& -2\pi  \left(\int_{\mb D}fG (D_{\mu}f) d\lambda \right)P_{\sigma^2}\psi''(\int_{\mb{D}} f\tilde h d\lambda) 
\end{align*}

The generalization to general functional $F(\phi,h)=\psi(\int_{\mb D} f_1(\phi+\tilde h)d\lambda,\dots,\int_{\mb D} f_n(\phi+\tilde h)d\lambda )$ is almost identical. The main differences are that 
$$  \mc{G} F(\phi,h) = \sum_{i}  \left(\int_{\mb D}(D_\mu f_i)(\phi+\tilde h)d\lambda+\chi\int_{\mb D}f_i\Re(L'_\mu)d\lambda
-\beta\int_{\mb D}f_id\lambda \right) \psi_i(\int_{\mb D} f_1(\phi+\tilde h) d \lambda, \dots)$$
and that we use the Gaussian integration by parts \eqref{eq:basic-ibp-2} to obtain
\begin{align*}
\int \psi(x_1+\int_{\mb{D}} f_1\phi d\lambda,\dots)\nu(d\phi)&=P_\Sigma\psi(x_1,\dots)\\
\int \left(\int_{\mb{D}} g\phi\right)\psi(x_1+\int_{\mb{D}} f_1\phi d\lambda,\dots)\nu(d\phi)&=-\sum_i2\pi  \left(\int_{\mb{D}} f_iGg d\lambda \right)P_\Sigma\psi_i(x)
\end{align*}
where $\Sigma_{i,j}=-2\pi  \int_{\mb{D}} f_iGf_j d\lambda$ and $P_\Sigma\psi(x)=\E(\psi(x+\Sigma^{1/2}N))$ where $N$ is a standard Gaussian in $\R^n$. 
\end{proof}

\subsection{Driving measure}
\label{sec:discussion-boundary-measure}

In this section, we provide some arguments explaining why the Gaussian multiplicative chaos measure $e^{- \xi h}$ is a natural measure to appear as the Loewner equation measure.  We consider first the following lemma.

\begin{Lem} 
\label{Lem:driving-cv}
Assume that distances induced by Riemannian metrics $(d^\eps)$ on $\overline{\mb{D}}$ converge to a length metric $d$ on every compact of $\overline{\mb{D}} \backslash \{ 0 \}$ w.r.t. the uniform topology, that for all $\eps > 0$, $d^\eps(0,\mb{U}) = d(0,\mb{U}) = \infty$ and that a local (compact sets of $\overline{\mb{D}} \backslash \{ 0 \}$) uniform (in $\eps$) modulus of continuity between these and the Euclidean distance hold. Furthermore, with $D_t$ the connected component of the complement of the $d$ growth from $\mb{U}$ at time $t$ containing the origin, suppose that $\partial D_t$ is a Jordan curve. Then, the associated driving measures $\mu^{\eps}(dt,dw)$ on $[0,T] \times \mb{U}$ converge to $\mu(dt,dw)$ w.r.t. weak topology. Here, the associated driving measures refer to the driving measures for the filled $d^{\eps}$-neighborhoods of the boundary.
\end{Lem}

\begin{proof}
We denote by $D_t^{\eps}$ the connected component of the complement of the $d_{\eps}$ growth from $\mb{U}$ at time $t$ containing the origin. We use the criterion \cite[Theorem 2.3]{caratheodory-criterion}. It states that the Carath\'eodory convergence follows from establishing that $D_t$ and $D_t^{\eps}$ have arbitrarily good common interior approximations: for each $\delta > 0$, there exists $\eps_0 >0$ and a connected set $K_{\delta}$ containing the origin  such that $K_{\delta} \subset D_{t} \cap \bigcap_{\eps < \eps_0} D_{t}^{\eps}$ and, additionally, $d_L(x, \partial D_t) < \delta$, $d_L(x, \partial D_t^{\eps}) < \delta$ for every $x \in \partial K_{\delta}$, where $d_L$ is the Lebesgue distance on $\mb{D}$.  We choose $K_{\delta}  :=\overline{D_{t+\delta}} \supset D_t$. As mentioned in the introduction, it is known that the boundary of a filled LQG ball is a Jordan curve. Denote by $\Gamma_{\delta}$ the Jordan curve associated with $\partial K_{\delta}$. $\Gamma_{\delta}$ separates $\partial \mb{D}$ and the origin. Take $\eps_0$ small enough so that $\| d^{\eps} - d \|_{\infty} < \delta$ for every $\eps < \eps_0$. Then, $d_{\eps}(\partial \mb{D},x) \geq d(\partial \mb{D},x) -\delta \geq t$. So, altogether we obtain $K_{\delta} \subset D_t^{\eps}$. Finally, the distance bounds follow by the uniform modulus of continuity assumption. Hence  the uniform (in time) Carath\'eodory convergence, i.e., the convergence of the normalized Riemann uniformization maps $(g_t^{\eps})^{-1}$ on compact subsets (by the Carath\'eodory kernel theorem). 

Now, the convergence of $(g_t^{\eps})^{-1}(0)$ to $g_t^{-1}(0)$ implies the convergence of $\int_0^t | \mu_s^\eps | ds$ to $\int_0^t | \mu_s | ds$ hence the tightness of the Borel measures $\mu^\eps$ on $[0,T] \times \mb{U}$. Denote by $\tilde{\mu}$ a subsequential limit. Then, from
$$
\partial_t g_t^\eps = \int - g_t^\eps(z) \frac{g_t^\eps(z)+w}{g_t^\eps(z)-w} \mu_t^\eps(dw)
$$
the convergence implies, in integrated form,
$$
\partial_t g_t  = \int - g_t(z) \frac{g_t(z)+w}{g_t(z)-w} \tilde{\mu}_t(dw) \qquad \text{and} \qquad
\partial_t g_t  = \int - g_t(z) \frac{g_t(z)+w}{g_t(z)-w} \mu_t(dw)
$$
so  $\tilde{\mu}$ and $\mu$ are two driving measures for $(g_t)$ and are therefore  equal. Indeed, this follows from the convergence of $g_t^{\eps}$ to $g_t$ on every compact subset of $D_t = g_t^{-1}(\mb{D})$, which itself follows from the convergence of $(g_t^{\eps})^{-1}$ to $g_t^{-1}$ on every compact subset of $\mb{D}$: fix $y \in K_2 \subset \subset D_t$ and set $x= g_t(y)$, $K_1 = g_t(K_2)$. We use the convergence of $(g_t^\eps)^{-1}$ to $g_t^{-1}$ on $K_1$. For $\delta >0$ and $\eps$ small enough, $(g_t^\eps)^{-1}(x) \in B(g_t^{-1}(x),\delta)=B(y,\delta)$ and
$$
| g_t^\eps (y) - g_t(y) | = | g_t^\eps (y)  - g_t^\eps ((g_t^\eps)^{-1}(x) )  | \leq \| (g_t^\eps)' \|_{L^\infty(B(y,\delta))} | (g_t^\eps)^{-1}(x) - g_t^{-1}(x) |
$$
By a distortion estimate, $|(g_t^\eps)'(w) | \leq d_L(w,\partial_t^\eps)^{-1}$, where we recall that $d_L$ is the Lebesgue distance on $\mb{D}$, from which the uniform convergence follows.  \end{proof}

From now on we will consider the Gaussian multiplicative chaos measure
\begin{equation}
\label{eq:GMC-driving}
\mu_h = e^{-\xi h} = e^{- \xi h (w)} \lambda(dw)
\end{equation}
That this is a good candidate to describe the driving measure for the LQG metric growth can be motivated by the following remarks.

\textbf{Remark.}
\begin{enumerate}
\item
First, $\mu_t$ depends only on the local growth associated with the length metric $e^{\xi (\phi_t + h_t+ \frac{t}{\xi})} ds$ and $\phi_t$ does not carry any information on the boundary $\mb{U}$ (it has zero boundary values, $\bigcap_{r>0} \sigma ( (\phi_t)_{| \mb{D} \backslash e^{-r} \overline{\mb{D}}})$ is trivial) so the local growth only depends on $h_t$ and one expects $\mu_t = G_t(h_t)$ for some function $G_t$.

\item
Consider a smooth metric $e^{\xi \varphi^{\eps}} ds$ with $\varphi^{\eps}$ smooth in $\overline{\mb{D}}$. Then the associated driving measure process, parametrized by distance and denoted by  $(\bar{\mu}_{s}^\eps)_{s>0}$ satisfies $\bar{\mu}_{0}^\eps = e^{-\xi \varphi^\eps} \frac{\lambda}{2\pi}$ (for a smooth metric, at $w \in \mb{U}$, the inward local distance associated with $\delta r$ is $\delta t = e^{\xi \varphi^\eps(w)} \delta r$, so the inward speed at $w$ is $\delta r/\delta t = e^{-\xi \varphi^\eps (w)}$ and the result follows from the remark below Lemma \ref{lem:hadamard}).  By removing the smoothing, namely by assuming a convergence $\varphi^\eps \to \varphi$ and of the approximating metrics, one expects that $\bar{\mu}_0 = e^{-\xi \varphi} \frac{\lambda}{2\pi}$.

\item
Assume $\mu_t = G(h_t)$ for some function $G$. We would like to say that $G(h)$ is a GMC $e^{-\xi h}$. In particular, we would like to prove at least that $G(h+x) = e^{-\xi x} G(h)$. Take $\underline{\phi_\C} = \phi_\C + f(D_{\phi_\C} (0, \cdot))$. Then, recalling that $D_{\underline{\phi_\C}}(0,z) = F( D_{\phi_\C}(0,z))$ where $F(r) = \int_0^r e^{\xi f(s)} ds$, we have $B_t^{\underline{\phi_\C}} = B_{u(t)}^{\phi_\C} $ where
$$
u(t) = \log F^{-1} (e^t)
$$
by considering the $D_{\underline{\phi_\C}}$ ball of radius $e^t$ and the $D_{\phi_{\C}}$ ball of radius $e^{u(t)}$. This gives
$$
\underline{\phi}_t + \underline{\tilde{h}}_t = \phi_{u(t)} + \tilde{h}_{u(t)} + f( D_{\phi_{\C}}(0,\hat{g}_{u(t)}^{-1} (\cdot))) + \frac{u(t)-t}{\xi}
$$
When $z$ in on the boundary $\mb{U}$, the term $f( D_{\phi_{\C}}(0,\hat{g}_{u(t)}^{-1} (z)))$ gives $f(e^{u(t)}) = f(F^{-1}(e^t))$. By using the Loewner's equation, the relation between the uniformizing maps and 
$$
u'(t) = \frac{d}{dt} \log F^{-1}(e^t) = \frac{e^t}{F'(F^{-1}(e^t)) F^{-1}(e^t)} = \frac{e^t e^{-\xi f(F^{-1}(e^t))}}{F^{-1}(e^t)} = e^{-\xi [f(F^{-1}(e^t)) + \frac{u(t)-t}{\xi} ]}
$$
we get 
$$
G(\tilde{h}_{u(t)} + f(F^{-1}(e^t)) + \frac{u(t)-t}{\xi}) = G(\underline{\tilde{h}}_t) = e^{-\xi [f(F^{-1}(e^t)) + \frac{u(t)-t}{\xi} ]} G(\tilde{h}_{u(t)}).
$$
\end{enumerate}

Although the above lemma and remark suggest that  that the Loewner measure of the LQG growth is a multiplicative chaos $e^{- \xi h}$ (and note that $h$ may not be Gaussian), this remains an open problem. In particular, a characterization of this measure by axioms including locality and shift remains open. In the context of QLE, in \cite{MS16}  the authors expect QLE$(\gamma^2,0)$ (which they construct rigorously using SLE for $\gamma = \sqrt{8/3}$) to generate the metric balls of a distance function, the $\gamma$-LQG metric.  At that time the metric was not constructed and this was an indirect approach to construct it. We refer the reader to Section 3.3 of \cite{MS16} for earlier considerations and heuristics with this approach, consistent with the Loewner measure being a multiplicative chaos $e^{- \xi h}$.

As \eqref{eq:GMC-driving} is a  multiplicative chaos, some assumptions on $h$ are needed to make sense of it. In the case of the LQG metric growth, it is not known that the trace of the free field on boundary of LQG balls has an appropriate distribution to be exponentiated in that sense (e.g., absolutely continuous w.r.t. to a log-correlated Gaussian field). 

In the following section, we consider a measure on fields $h$ for which the GMC $e^{- \xi h}$ is well defined. If one could 1) find an invariant distribution, then 2) construct from it the growth process, 3) show that it gives rise to a distance function which satisfies the axioms of the characterization of LQG metric by Gwynne and Miller, then as a consequence the distribution of the field would be obtained.

\subsection{Trace fields and integration by parts}
\label{sec:Trace-fields-and-IBP}

We introduce here the measure $\rho(dh)$ we will be working with in all subsequent sections. The goal of Section \ref{sec:invariance-condition} will be to retrieve an invariance condition \eqref{eq:invariance-cond} for well chosen parameters, when the measure $\mu$ is the GMC measure $e^{-\xi h}$ (see the motivations from Section \ref{sec:discussion-boundary-measure}). The integration by parts formula of the measure (which characterizes it) will be the starting point of the proof in Section \ref{sec:invariance-condition}, we will use it with well-chosen test functionals. However, that cannot be implemented directly and an important intermediate step is to boundary localize the terms appearing in \eqref{def:L} \eqref{eq:drift-diffu} or \eqref{eq:invariance-cond}, which is the content of Section \ref{sec:boundary-localization-kernelV}

\paragraph{Trace fields.} We consider a measure $\rho(dh)$ of the form
\begin{equation}
\rho(dh)=e^{-\VV(h)}{\mc D}h
\end{equation}
where $\VV$ is a smooth potential, in the sense that $D_k\VV$ is defined $\rho$-a.e. for $k$ smooth on ${\mb U}$, and admits the gradient representation
$$D_k\VV(h)=\int_{\mb U}k(D\VV(h))d\lambda$$
where $D\VV(h)$ is a distribution on ${\mb U}$ for $\rho$-a.e. $h$. Note that $\mc{D}h$ alone is never used, this is formally an infinite dimensional Lebesgue measure (translation invariant). The measure $\rho(dh)$ makes sense and falls in the classical setup of Gaussian measures on Banach/Hilbert spaces for some appropriate quadratic form $\mc{V}$ (and non-Gaussian measures can be considered by adding a positive term to such $\mc{V}$'s).

By analogy with LCFT it is natural to consider 
$$ 
\VV(h)  = \frac{1}{4 \pi } \int_{\mb D} | \nabla h |^2 d\lambda  + c \int_{\mb U} h d\lambda + b \int_{\mb{U}} e^{\tau h} d\lambda
$$
where $\tau = \frac{\gamma}{2}$ which is also equal to $2\xi$ in the case of pure gravity. In what follows, we specialize to the case $b = 0$, which we can write as
\begin{equation}
\label{eq:V-exp-DV}
\VV(h) = - \frac{1}{4 \pi } \int_{\mb U} h \partial_n H h d\lambda +  2\pi c \dashint_{\mb U} h d\lambda \quad \text{so} \quad D\VV(h) = - \frac{1}{2 \pi } (\partial_n Hh) + c.
\end{equation}
It is almost the same as the measure \eqref{def:rho0}, rigorously defined in Section \ref{subsec:gff}, except that the one-dimensional Lebesgue measure is reweigthed.

\paragraph{Fr\'echet derivatives and integration by parts.} By recalling the integration by parts formula \eqref{prel:ibp}, we see here that $\rho$ satisfies the following integration by parts formula
$$
\int (D_kF)\rho(dh) = \int F(h)(D_k\VV)(h)\rho(dh)=  \int F(h) \left( -\frac{1}{2\pi } \int_{\mb{U}} k \partial_n H h d\lambda + c \int_{\mb{U}} k d\lambda   \right) \rho(dh)   
$$
The additional term compared to \eqref{prel:ibp} comes from the zero modes. An explicit potential $\mc{V}$ and the fact that $D_f\int g d\mu =-\xi\int (fg) d\mu$ for test functions $f,g$ (see below \eqref{eq:weighting-mass}) will be crucial in our proofs. In particular, this second relation could be satisfied in a non-Gaussian setup where $\mu$ is rather only a multiplicative chaos.

We want to use this formula to retrieve the invariance condition \eqref{eq:invariance-cond}. For a functional  
$$
F(h)=\left(\int_{\mb U}\ell d\mu\right)\varphi(\int_{\mb D}f_1h d\lambda,\cdots, \int_{\mb D}f_n h d\lambda),
$$
where $\ell$ a smooth function on ${\mb U}$, $f_i \in C^{\infty}_c(\mb D \setminus \{ 0 \})$ and $\varphi$ has compact support, we have 
$$
D_kF(h)=\left(-\xi\int_{\mb U}\ell kd\mu\right)\varphi(\int_{\mb D}f_1h d\lambda,\dots, \int_{\mb D}f_n h d\lambda)+
\sum_i\left(\int_{\mb D}f_iHk d\lambda \right)\left(\int_{\mb U}\ell d\mu\right)\varphi_i(\int_{\mb D}f_1h d\lambda,\dots, \int_{\mb D}f_1h d\lambda).
$$ 
So, the integration by parts formula gives 
\begin{equation}
\label{eq:IBP-potential}
\int\left(\left( -\int_{\mb U} k(D\VV)(h) d\lambda \int_{\mb U}\ell d\mu -\xi\int_{\mb U}\ell kd\mu \right)\varphi+\sum_i\left(\int_{\mb D}f_iHk d\lambda \right)\left(\int_{\mb U}\ell d\mu\right)\varphi_i\right)\rho(dh)=0.
\end{equation}
In order to assure integrability, we suppose there is at least one $f_i$ for which $\int_{\mb D} f_i d\lambda \neq 0$. This will be discussed in details in Section \ref{sec:integrability}.

For $\varphi, f_i$'s fixed, this is bilinear in $\ell, k$. Let $R_f:\U^2\rightarrow \R$ be a smooth kernel depending linearly on a function $f\in C^\infty_c({\mb D})$ and admitting a representation
\begin{equation}
\label{def:Rf}
R_f(w,w')=\int k_u(w)\ell_u(w')d \Lambda(u).
\end{equation}
By bilinearity, integrating over $d \Lambda$, we can write
\begin{equation}
\label{eq:IBP-General}
\begin{split}
\int\left(\left(-\int_{{\mb U}^2} (D\VV)(h)(w)R_f(w,w') d\lambda(w) d\mu(w')-\xi\int_{\mb U}R_f(w,w)d\mu(w) \right)\varphi \right.   \\ 
\left.+\sum_j\left(\int_{{\mb U}^2}(H^*f_j)(w)R_f(w,w') d\lambda(w) d\mu(w')\right)\varphi_i\right)\rho(dh)=0.
\end{split}
\end{equation}

\subsection{Boundary localization, kernel $V$ and contractions} 
\label{sec:boundary-localization-kernelV}

In this section, our goal is to boundary localize the terms appearing in \eqref{def:L} and \eqref{eq:drift-diffu} or \eqref{eq:invariance-cond}. The main result corresponds to Proposition \ref{Prop:gen-loc}, which gathers intermediate results from Lemma \ref{lem:boundary-localization} and  involves the kernels $V_p$ constructed in Lemma \ref{lem:kernel-V}. In particular, we will use in the following section the properties of these kernels which are derived in Lemma \ref{lem:prop-kernel}.

To compare the invariance condition \eqref{eq:invariance-cond} with \eqref{eq:IBP-General}, we need the following deterministic lemma. For this  lemma, $\mu$ is not necessarily of the form $e^{-\xi h}$ but is just a Borel measure on $\mb{U}$ with finite total mass. Furthermore, we recall that $*$ is defined is  Section \ref{sec:notation} and denotes the adjoint of the harmonic extension.

\begin{Lem}
\label{lem:boundary-localization}

For $f \in C^\infty_c({\mb D}\setminus\{0\})$, recalling that $L_{\mu}$ is given by \eqref{eq:Loewner-vector-field} and that $D_{\mu} f$ is given by \eqref{eq:dmu-expressions}, we have
\begin{equation}
\label{eq:Boundary-sing}
 \int_{\mb D}(D_{\mu} f) \log |\cdot| d\lambda=-2\pi\int_{\mb U}(H^*f)d\mu
\end{equation}
and
\begin{equation}
\label{eq:Boundary-mu}
\int_{\mb D}f\Re(L'_\mu)d\lambda=2\pi\int_{\mb U}(f^*-\partial_nHf^*)d\mu
\end{equation}
and for $i,j$, setting $p_i = f_i^*$, $p_j = f_j^*$,
\begin{equation}
\label{eq:Boundary-last}
-  \int_{\mb D} f_i G (D_{\mu} f_j) d\lambda -  \int_{\mb D} (D_{\mu} f_i) G f_j  d\lambda = 2 \pi \int_{\mb U} p_i p_j d\mu.
\end{equation}
\end{Lem}
The proof of the lemma is done by proving these formulas in the case of a Dirac measure and using the linearity in $\mu$ and the fact that  finite combination of Dirac measures are dense in $\mc{M}(\mb{U})$ for the weak topology. The formula \eqref{eq:Boundary-last} can also be obtained  by carrying out  the computation of
$${\frac{d}{dt}}_{|t=0}\int_{D_t}f_1 G_{D_t}f_2 d\lambda =  - \int_{\mb D} (D_{\mu}f_1) G f_2 d\lambda - \int_{\mb D}  f_1  G  (D_{\mu}f_2) d\lambda$$
on $\mb D$, i.e., using \eqref{def:dmu0}. The lemma directly implies the following corollary.

\begin{Cor}
The condition \eqref{eq:invariance-cond} in terms of boundary integrals (but for the first term) corresponds to  
\begin{equation}
\label{eq:condition-boundary}
\begin{split}
\int \left(\sum_i\left(\frac 1{2\pi}\int_{\mb D} (D_{\mu} f_i)hd\lambda+(\chi-\alpha) \int_{\mb U} p_i d\mu
-\chi\int_{\mb U}\partial_nH p_i d\mu -\beta\dashint_{\mb U} p_i d\lambda
\right)\tilde\psi_i \right. \\
\left. +\pi  \sum_{i,j}\left(\int_{\mb U}p_i p_jd\mu\right)\tilde\psi_{i,j}\right)\rho(dh) = 0. 
\end{split}
\end{equation}
\end{Cor}

\begin{proof}[Proof of Lemma \ref{lem:boundary-localization}]
We start with  \eqref{eq:Boundary-sing}. We have, recalling \eqref{eq:dmu-expressions} and with $g_{w} = D_{\delta_w} f$,
\begin{align*}
\int_{\mb D}g_w\log|\cdot|d\lambda&=
-2\int_{\mb D}\Re\left(\partial_z\left(z\frac{z+w}{z-w}f(z)\right)\right)\log|z|d\lambda(z)\\
&=2\Re\int_{\mb D}\left(z\frac{z+w}{z-w}f(z)\right)\partial_z\log|z|d\lambda(z)\\
&=\Re\int_{\mb D}\left(\frac{z+w}{z-w}f(z)\right)d\lambda(z) =-2\pi(H^*f)(w)
\end{align*}
so that $\int_{\mb D} (D_{\mu} f) \log | \cdot | d\lambda=-2\pi \int_{\mb U}(H^*f)d\mu$.

Then we prove \eqref{eq:Boundary-mu}. We introduce, for $w \in \mb U$,
$$
\ell_f(w) := \int_{\mb D}f\Re(L'_{\delta_w})d\lambda,
$$
so that $\int_{\mb D}f\Re(L'_\mu)d\lambda=\int_{\mb U}\ell_fd\mu$.
We have, using $z\frac{z+w}{z-w}=z+\frac{2wz}{z-w}=z+2w+\frac{2w^2}{z-w}$,
\begin{align*}
\ell_f(w)&=-\int_{\mb D}f(z)\Re\left(\partial_z\left(z\frac{z+w}{z-w}\right)\right)d\lambda(z)=
-\int_{\mb D}f(z)\Re\left(1-\frac{2w^2}{(z-w)^2}
\right)d\lambda(z).
\end{align*}
With  $z^* = \bar{z}^{-1}$ observe that
$$f^*(w)=\frac 1{2\pi}\int_{\mb D}f(z)\Re\left(\frac{w+z}{w-z}\right)d\lambda(z)=
-\frac 1{2\pi}\int_{\mb D}f(z)\Re\left(\frac{w^*+z^*}{w^*-z^*}\right)d\lambda(z)
$$
so that
$$Hf^*(z')=-\frac 1{2\pi}\int_{\mb D}f(z)\Re\left(\frac{z'+z^*}{z'-z^*}\right)d\lambda(z)$$
and
\begin{equation}
\label{eq:intermediate-qle}
(\partial_nHf^*)(w)=-\frac 1{2\pi}\int_{\mb D}f(z)\Re\left(\frac{2wz^*}{(w-z^*)^2}\right)d\lambda(z)
=-\frac 1{2\pi}\int_{\mb D}f(z)\Re\left(\frac{2wz}{(w-z)^2}\right)\lambda(dz)
\end{equation}
Since
$$1-\frac{2w^2}{(z-w)^2}=1-\frac{2wz}{(z-w)^2}-\frac{2w}{w-z}=-\frac{w+z}{w-z}-\frac{2wz}{(z-w)^2},$$
we have $\ell_f(w)=2\pi (f^*(w)-\partial_nHf^*(w))$.

Finally, we prove \eqref{eq:Boundary-last}.  We consider 
$$T(f_1,f_2,\mu)=\int_{\mb D} (f_1G (D_{\mu} f_2)+f_2G(D_{\mu} f_1)) d\lambda$$
which is linear in $\mu$ and bilinear and symmetric in $(f_1,f_2)$. We have $T(f_1,f_2,\mu)=\int_{\mb U}T_w(f_1,f_2)\mu(dw)$ where
$$T_w(f_1,f_2)=T(f_1,f_2,\delta_w)=-2\left(\int_{\mb D}f_1G\Re\left(\partial_z\left(z\frac{z+w}{z-w}f_2(z)\right)\right)
+\int_{\mb D}f_2G\Re\left(\partial_z\left(z\frac{z+w}{z-w}f_1(z)\right)\right)
\right)$$
This is continuous in $w$ near ${\mb U}$, so we can replace $w$ with $(1+\eps)w$ and take $\eps\searrow 0$. Then
\begin{align*}
T_w(f_1,f_2)&=2\Re\int_{{\mb D}^2}f_1(z_1)\partial_{z_2}G(z_1,z_2)\left(z_2\frac{z_2+w}{z_2-w}f_2(z_2)\right)d\lambda(z_1)d\lambda(z_2)+(1\leftrightarrow 2)\\
&=\frac{\Re}{2\pi}\int_{{\mb D}^2}f_1(z_1)\left(\frac{1}{z_2-z_1}-\frac 1{z_2-z_1^*}\right)\left(z_2\frac{z_2+w}{z_2-w}f_2(z_2)\right)d\lambda(z_1)d\lambda(z_2)+(1\leftrightarrow 2)
\end{align*}
Since $z_2\frac{z_2+w}{z_2-w}=z_2+\frac{2wz_2}{z_2-w}=z_2+2w+\frac{2w^2}{z_2-w}$, we have
$$\frac{1}{z_2-z_1}z_2\frac{z_2+w}{z_2-w}+(1\leftrightarrow 2)=1-\frac{2w^2}{(z_1-w)(z_2-w)}$$
and
\begin{align*}
\frac{1}{z_2-z^*_1}z_2\frac{z_2+w}{z_2-w}+\overline{(1\leftrightarrow 2)}
&=\frac{z_2}{z_2-z^*_1}\frac{z_2+w}{z_2-w}-\frac{z_2}{z_2-z^*_1}\frac{z^*_1+w^*}{z^*_1-w^*}\\
&=-\frac{2wz_2}{(z_2-w)(z_1^*-w)}=\frac{2\bar{z_1}z_2}{(z_2-w)(\bar{z_1}-\bar w)}
\end{align*}
Moreover
\begin{align*}
\left(\frac{z_1+w}{z_1-w}\right)\left(\frac{z_2+w}{z_2-w}\right)&=
\left(1+\frac{2w}{z_1-w}\right)\left(1+\frac{2w}{z_2-w}\right)&=
-1+\frac{z_1+w}{z_1-w}+\frac{z_2+w}{z_2-w}+\frac{4w^2}{(z_1-w)(z_2-w)}\\
\overline{\left(\frac{z_1+w}{z_1-w}\right)}\left(\frac{z_2+w}{z_2-w}\right)&=
\left(-1+\frac{2\bar z_1}{\bar z_1-\bar w}\right)\left(-1+\frac{2z_2}{z_2-w}\right)&=
-1-\frac{\bar z_1+\bar w}{\bar z_1-\bar w}-\frac{z_2+w}{z_2-w}+\frac{4\bar z_1z_2}{(\bar z_1-\bar w)(z_2-w)}\\
\end{align*}
so that
\begin{align*}
\Re\left(\left(\frac{1}{z_2-z_1}-\frac 1{z_2-z_1^*}\right)\left(z_2\frac{z_2+w}{z_2-w}\right)+(1\leftrightarrow 2)\right)
&=-\frac 12\Re\left(\left(\frac{z_1+w}{z_1-w}\right)\left(\frac{z_2+w}{z_2-w}\right)+\overline{\left(\frac{z_1+w}{z_1-w}\right)}\left(\frac{z_2+w}{z_2-w}\right)\right)\\
&=-\Re\left(\frac{z_1+w}{z_1-w}\right)\Re\left(\frac{z_2+w}{z_2-w}\right)
\end{align*}
so that $T_w(f_1,f_2)=-2\pi (H^*f_1)(w)(H^*f_2)(w)$.
\end{proof}

\paragraph{Kernel $V$.}

The last term to reinterpret in \eqref{eq:invariance-cond}  is $ \int_{\mb D} (D_{\mu}f) h d\lambda $. In view of the first term in \eqref{eq:IBP-General},  we wish to write, recalling \eqref{def:Rf},
$$
 \int_{\mb D} (D_{\mu}f) h d\lambda = \int d \Lambda(u) \int_{\mb{U}^2} k_{u}(w) \ell_{u}(w') \partial_n H h(w) d\lambda(w) d\mu(w').
$$
To do so, the idea is the following. By using that $h$ is harmonic we will find $U_f(w,w')$ such that 
$$
\int_{\mb D} (D_{\mu}f) hd\lambda=\int_{\U^2}U_f(w,w')h(w')\mu(dw)d\lambda(w').
$$ Then, we look for a kernel $V$ such that $U_f(w,\cdot)=\partial_n HV_{f^*}(w,\cdot)$, leading to
$$
\int_{\mb D} (D_{\mu}f) h d\lambda =\int_{\U^2}V_{f^*}(w,w') \partial_n Hh(w')\mu(dw)d\lambda(w'),
$$
and we verify that for $f \in C_0^{\infty}(\mb{D})$,  we can write this kernel in the form
$V_{f^{*}}(w,w') = \int_{\mb{U}^2} k_u(w) \ell_u(w') d\Lambda(u).$

This is the content of the following lemma, in which we provide a rather explicit expression of $V$.

\begin{Lem}
\label{lem:kernel-V}
Consider $f \in C_c^{\infty}(\mb{D})$ and write $f^* =: p$. The harmonic extension of $p$ on $\mb{D}$ can be written as $\Re(P)$ with $P$ holomorphic and $\tilde p=\Im(P)$.  We define
\begin{equation}
\label{eq:Vp-expression}
V_p(w,w') :=\Re\left(\frac{w+w'}{w'-w}(P(w')-P(w)\right)=-\Im\left(\frac{w+w'}{w'-w}\right)(\tilde p(w')-\tilde p(w))=V_p(w',w).
\end{equation}
Then, we have
\begin{equation}
\label{eq:singular-V}
\int_{\mb D} (D_{\mu}f) h d\lambda =\int_{\U^2}V_{f^*}(w,w') \partial_n Hh(w')\mu(dw)d\lambda(w').
\end{equation}
Furthermore, there exist $\Lambda(du)$, $k_u$ and $\ell_u$ depending on $f$ such that 
\begin{equation}
\label{eq:kernel-product-form}
V_{f^{*}}(w,w') = \int_{\mb{U}^2} k_u(w) \ell_u(w') d\Lambda(u).
\end{equation}
\end{Lem}

\begin{proof}
Recalling the expression of $D_{\mu}f$ in \eqref{eq:dmu-expressions}, since $h$ is harmonic we can write, by Fubini,
$$\int_{\mb D} (D_{\mu}f)hd\lambda=  \int_{\mb U} (D_{\mu}f)^* h d\lambda  =\int_{\U^2}U_f(w,w')h(w')\mu(dw)d\lambda(w')$$
with
\begin{align*}
U_f(w,w')&=\frac{-2}{2\pi}\int_{\mb D}\Re\left(\partial_z\left(z\frac{z+w}{z-w}f(z)\right)\right)\Re\left(\frac{w'+z}{w'-z}\right)d\lambda(z)\\
&=\frac{\Re}\pi\int_{\mb D}\left(z\frac{z+w}{z-w}f(z)\right)\partial_z\Re\left(\frac{w'+z}{w'-z}\right)d\lambda(z)\\
&=\frac{\Re}{2\pi}\int_{\mb D}\left(z\frac{z+w}{z-w}f(z)\right)\frac{2w'}{(w'-z)^2}d\lambda(z)\\
\end{align*}
Note that $\partial_\theta \Im\left(\frac{\alpha}{z-w'}\right)=\Im\left(\frac{\alpha iw'}{(z-w')^2}\right)=\Re\left(\frac{\alpha w'}{(z-w')^2}\right)$ (where $\partial_\theta$ is the tangential derivative w.r.t. $w'$), so that
\begin{align*}
U_f(w,w')&=\partial_\theta \frac{\Im}{\pi}\int_{\mb D}\left(z\frac{z+w}{z-w}f(z)\right)\frac{1}{z-w'}d\lambda(z)
\end{align*}
Then
\begin{align*}
\frac{z(z+w)}{(z-w)(z-w')}&=1+\frac{2w^2}{(z-w)(w-w')}+\frac{w'(w+w')}{(w'-w)(z-w')}\\
&=1+\frac{z+w}{z-w}\frac{w}{w-w'}-\frac{w}{w-w'}+\frac 12\left(\frac{z+w'}{z-w'}\frac{w+w'}{w'-w}-\frac{w+w'}{w'-w}\right)\\
&=\frac 12+\frac{z+w}{z-w}\frac{w}{w-w'}+\frac 12\frac{z+w'}{z-w'}\frac{w+w'}{w'-w}\\
&=\frac 12\left(1+\frac{z+w}{z-w}+\frac{z+w}{z-w}\frac{w+w'}{w-w'}-\frac{z+w'}{z-w'}\frac{w+w'}{w-w'}\right)
\end{align*}
so that
\begin{align*}
2\Im\left(\frac{z(z+w)}{(z-w)(z-w')}\right)&=\Im\left(\frac{z+w}{z-w}\right)+\Im\left(\frac{w+w'}{w-w'}\right)\Re\left(\frac{z+w}{z-w}-\frac{z+w'}{z-w'}\right)
\end{align*}
and
\begin{align*}
U_f(w,w')&=\partial_\theta \left(\Im\left(\frac{w+w'}{w-w'}\right)(f^*(w')-f^*(w))+\frac 1{2\pi}\int_{\mb D}\Im\left(\frac{z+w}{z-w}\right)f(z)d\lambda(z)\right)\\
&=-\partial_\theta \left(\Im\left(\frac{w+w'}{w'-w}\right)(f^*(w')-f^*(w))\right).
\end{align*} 
Write $f^* = p=\Re(P)$ on ${\mb U}$ with $P$ holomorphic, so that $\tilde p=\Im(P)$. Then for $w$ fixed, on ${\mb U}$
$$\Im\left(\frac{w+w'}{w'-w}\right)(p(w')-p(w))=\Im\left(\frac{w+w'}{w'-w}(P(w')-P(w)\right).$$
Hence
$$U_f(w,\cdot)=\partial_n HV_{f^*}(w,\cdot)$$
where
$$
V_p(w,w')=\Re\left(\frac{w+w'}{w'-w}(P(w')-P(w)\right)=-\Im\left(\frac{w+w'}{w'-w}\right)(\tilde p(w')-\tilde p(w))=V_p(w',w).
$$

Note that, if $f\in C^\infty_c({\mb D})$, $f^*(w)=\sum_{n\in\Z}a_nw^n$ where $(a_n)$ decays exponentially  as $n$ goes to $\pm\infty$ (this can be checked by developing the term $\frac{1}{w-z} = \frac{1}{w}\sum_{k \geq 0} (z/w)^k$ in the Poisson kernel and using that $f$ is compactly supported). It follows that
$$V_{f^*}(w,w')=\sum_{m,n\in\Z}b_{m,n}w^m(w')^n$$
where $(b_{m,n})$ decays exponentially at $|m|+|n|$ goes to $\infty$. In particular, the finite rank kernels
$$V^N_{f^*}(w,w')=\sum_{m,n\in\{-N,\dots,N\}^2}b_{m,n}w^m(w')^n$$
converge to $V$ uniformly in $C^k({\mb U}^2)$ for any $k\geq 0$. 
\end{proof}

\paragraph{Contractions.} We need some contractions of the kernel $V$ to reinterpret the terms appearing in \eqref{eq:IBP-General} in the case of the integration over $d\Lambda$.  By construction
$$\int_{{\mb U}^2}V_{f^*}(w,w')(\partial_nHh)(w')d\lambda(w')d\mu(w)=\int_{{\mb U}^2}U_f(w,w')h(w')d\lambda(w')d\mu(w)=\int_{\mb D} (D_{\mu}f) hd\lambda$$ and we need to compute the following terms as well, with $p, q \in C^{\infty}(\mb{U})$,
$$
\int_{\mb{U}^2} V_p(w,w') d\lambda(w) d\mu(w'), \quad \int_{\mb{U}} V_p(w,w)  d\mu(w) \quad \text{and} \quad \int_{\mb{U}^2} q(w) V_p(w,w') d\lambda(w) d\mu(w').
$$
\begin{Lem}
\label{lem:prop-kernel}
We have the following identities, for $p, q \in C^{\infty}(\mb{U})$,
\begin{align*}
\int_{\mb U}V_p(w,w')d\lambda(w)  & = 2\pi(p(w')-\dashint p) \\
V_p(w,w) & =-2\partial_n Hp(w) \\
\frac 1{2\pi}(pV_q+qV_p)(w') & = pq+\tilde p\tilde q-\dashint p\dashint q  \\
\frac{1}{2\pi} (p V_q - q V_p) & = \wt{(\wt{p} q - p \wt{q})}.
\end{align*}
In the right-hand side of the last equality, the conjugation is applied to $\tilde{p}q - p \tilde{q}$.
\end{Lem}

\begin{proof}
Write $p=\Re(P)$ on ${\mb U}$, $P$ holomorphic. Then
$$\int_{\mb U}V_p(w,w')d\lambda(w)=\int_{\mb U}\Re\left(\frac{w+w'}{w'-w}(P(w')-P(w)\right)d\lambda(w)=2\pi\Re(P(w')-P(0))=2\pi(p(w')-\dashint p)$$
On the diagonal, we have $V_p(w,w)=2\partial_\theta \tilde p(w)=-2\partial_n Hp(w)$. Furthermore,
\begin{align*}
\frac 1{2\pi}(pV_q+qV_p)(w') =&  \frac 1{2\pi}\int_{\mb U}(p(w)V_q(w,w')+q(w)V_p(w,w'))d\lambda(w)\\
= &p(w')(q(w')-\dashint q)+q(w')(p(w')-\dashint p)\\
&+\frac 1{2\pi}\int_{\mb U}((p(w)-p(w')V_q(w,w')+(q(w)-q(w')V_p(w,w'))d\lambda(w)\\
=& p(w')(q(w')-\dashint q)+q(w')(p(w')-\dashint p)\\
&+\frac 1{2\pi}\int_{\mb U}\Im\left(\frac{w+w'}{w'-w}\right)\Im((P(w')-P(w))(Q(w')-Q(w))d\lambda(w)\\
=& p(w')(q(w')-\dashint q)+q(w')(p(w')-\dashint p)-\Re((P(w')-P(0))(Q(w')-Q(0)))
\end{align*}
so that
\begin{align*}
\frac 1{2\pi}(pV_q+qV_p)
=p(q-\dashint q)+q(p-\dashint p)-(p-\dashint p)(q-\dashint q)+\tilde p\tilde q
=pq+\tilde p\tilde q-\dashint p\dashint q
\end{align*}
Lastly,  we have
\begin{equation}
\label{def:operator-A}
A(p,q) := \frac{1}{2\pi} (p V_q - q V_p) = \wt{\wt{p} q - p \wt{q}}
\end{equation}
Indeed, recalling $V_{p}(w,w')  = -\Im\left(\frac{w+w'}{w'-w}\right)(\widetilde{p}(w')-\widetilde{p}(w))$ from  \eqref{eq:Vp-expression}, we have
\begin{align*}
\frac{1}{2\pi} (p V_q - q V_p)(w) & = - \frac{1}{2\pi} \int_{\mb{U}} \Im\left(\frac{w+w'}{w'-w}\right) \left( (\widetilde{q}(w')-\widetilde{q}(w))   p(w') - (\widetilde{p}(w')-\widetilde{p}(w)) q(w') \right) \lambda(dw') \\
& = \frac{1}{2\pi} \int_{\mb{U}} \Im\left(\frac{w'+w}{w'-w}\right) \left( \wt{p}(w') q(w') - p(w') \wt{q}(w') + \wt{q}(w) p(w') - \wt{p}(w) q(w') \right) \lambda(dw') 
\end{align*}
and the result follows  from \eqref{eq:harm-conj-pv}, namely $\tilde p(w)=\frac 1{2\pi}p.v.\int_{\mb U}p(w')\Im\left(\frac{w'+w}{w'-w}\right)d\lambda(w')$.
\end{proof}

We note here that $A(p,q)$ is the term appearing in the expression of the Dirichlet form \eqref{eq:antisymmetric-part}. Clearly, $A(p,q) = - A(q,p)$ and $A(p,q) = A(p - \dashint p, q-\dashint q) + (\dashint p) q - (\dashint q) p$. Also, if $p$ and $q$ are strictly ordered or have the same degree (when considering their decomposition on the Fourier basis), a computation using trigonometry formulas gives
\begin{equation}
A(p,q)  = - \sgn(p,q) (pq + \tilde{p} \tilde{q}) 
\end{equation}
where  $\sgn(p,q) = 1_{\deg(p) > \deg(q)} -1_{\deg(p) < \deg(q)}$.

Now, we give an alternative expression of $\mc{L}_{\alpha, \chi, \beta}$, initially defined in \eqref{def:L} and \eqref{eq:drift-diffu}, with $\mu$ being here the Borel measure coming from the Loewner equation. As a consequence of Lemma \ref{lem:boundary-localization} and Lemma \ref{lem:kernel-V}, we have
\begin{Prop}
\label{Prop:gen-loc}
For a bulk cylindrical test function $F(h)=\psi(\int_{\mb D}f_1 h d\lambda,\dots, \int_{\mb D}f_n h d\lambda)$, 
\begin{equation}
\label{eq:L-boundary-loc}
{\mc L}_{\alpha, \chi, \beta} F = \sum_i b(p_i)\psi_i+\frac 12\sum_{ij}\sigma(p_i,p_j)\psi_{i,j} 
\end{equation}
with $p_i = f_i^*$, $\sigma(p,q)=  4\pi^2  \int pqd\mu$ and
$$
b(f^*) =  \int_{\U^2}V_{f^*}(w,w') \partial_n Hh(w')\mu(dw)d\lambda(w')-2\pi \chi \int_{\mb{U}} \partial_n H p_i d\mu+2\pi(\chi-\alpha)\int_{\mb U}f^*d\mu-\beta \int_{\mb U} f^*d\lambda.
$$
\end{Prop}

\subsection{Vanishing terms and invariance equation}
\label{sec:invariance-condition}

In this section, to prove Theorem \ref{thm-invariance}, we start from the integration by parts formula \eqref{eq:IBP-General} and we look for \eqref{eq:invariance-cond}. This is done by using the framework developed up to now, in particular the results from Section \ref{sec:boundary-localization-kernelV} and well-chosen test functionals  when integrating by parts. Additionally, we will need a proof that some terms appearing along the way cancel with each others (Lemma \ref{Lem:rotational-invariance-v1}). The idea to prove this lemma is to consider the rotational invariance of the field.

 We take $f_1, \dots, f_n \in C^\infty_c({\mb D}\setminus\{0\})$  and $\psi\in C^\infty_c(\R)$  fixed. We suppose there is at least one $f_i$ for which $\int_{\mb D} f_i d\lambda \neq 0$.   Applying \eqref{eq:IBP-General} to pairs $(\varphi,f)=(\tilde{\psi}_i,f_i)$ and summing over $i$ we get
\begin{align*}
\int\left(\sum_i\left(-\int_{{\mb U}^2} (D\VV)(h)(w)R_{f_i}(w,w')d\lambda(w)d\mu(w') -\xi\int_{\mb U}R_{f_i}(w,w)d\mu(w) \right)\tilde\psi_i\right.\\
+\left.\sum_{i,j}\left(\int_{{\mb U}^2}(H^*f_j)(w)R_{f_i}(w,w')d\lambda(w)d\mu(w')\right)\tilde\psi_{i,j}\right)\rho(dh)=0.
\end{align*}
Every term appearing in this section is well defined, i.e., the terms integrated against $\rho$ are integrable. As verifying this distracts from the main arguments, this is postponed and proved in Section \ref{sec:integrability}.

Now, we use the expression \eqref{eq:V-exp-DV} of $\mc{V}$ which gives $D\VV(h) = - \frac{1}{2 \pi } (\partial_n Hh) + c$ and we use the notation $p_i=H^*f_i$. Recalling \eqref{eq:singular-V} and \eqref{eq:kernel-product-form}, we specialize this equation to the kernel $V$ so
\begin{align*}
\int\left( \sum_i \left( \frac{1}{2\pi}\int_{{\mb U}^2}V_{p_i}(w,w')(\partial_nHh)(w)d\lambda(w)d\mu(w') -c \int_{{\mb U}^2}V_{p_i}(w,w')d\lambda(w)d\mu(w')  -\xi\int_{\mb U}V_{p_i}(w,w)d\mu(w)\right)\tilde\psi_i\right.\\
\left.+\sum_{i,j}\left(\int_{{\mb U}^2} p_j(w) V_{p_i}(w,w')d\lambda(w)d\mu(w')\right)\tilde\psi_{i,j}\right)\rho(dh)=0
\end{align*}
and, using contractions from Lemma \ref{lem:prop-kernel}, 
\begin{align}
\int\left( \sum_i\left(\frac{1}{2\pi}\int_{{\mb D}} (D_{\mu}f_i) hd\lambda-2\pi c \int_{{\mb U}}(p_i-\dashint p_i)d\mu+2\xi\int_{\mb U}(\partial_nHp_i)(w)d\mu(w)\right)\tilde\psi_i\right. \nonumber \\
\left.+\pi\sum_{i,j} \tilde\psi_{i,j} \int_{\mb U}(p_ip_j+\tilde p_i\tilde p_j-\dashint p_i\dashint p_j)d\mu\right)\rho(dh)=0. \label{eq:conseq-ibp}
\end{align}
By comparing the last term in \eqref{eq:conseq-ibp} and \eqref{eq:condition-boundary}, we aim to rewrite 
$$ \sum_{i,j} \tilde\psi_{i,j}  \int_{\mb U}(\tilde p_i\tilde p_j-\dashint p_i\dashint p_j)d\mu \rho(dh).$$
Using the integration by parts formula \eqref{eq:IBP-potential} with $(\varphi,\ell,k)=(\tilde\psi_i,\dashint p_i,1)$ gives
\begin{equation}
\int\left(\sum_i\left(-\xi\dashint p_i\int_{\mb U}d\mu  -2\pi c \dashint p_i\int_{\mb U}d\mu\right)\tilde\psi_i+2\pi\sum_{i,j}  \tilde\psi_{i,j} \dashint p_i \dashint p_j\int_{\mb U}d\mu \right)\rho(dh)=0.
\end{equation}

The other term (the one including $\int \tilde{p}_i \tilde{p}_j d\mu$) is more delicate to reinterpret. We need the following lemma which relies on the rotational invariance of $h$ and on an integration by parts. We also provide an alternative proof of this fact in the appendix (Section \ref{sec:GI}).
\begin{Lem} 
\label{Lem:rotational-invariance-v1}
For $f_1, \dots, f_n \in C^\infty_c({\mb D}\setminus\{0\})$ with at least one $f_i$ for which $\int_{\mb D} f_i d\lambda \neq 0$ and $\psi\in C^\infty_c(\R)$, with $p_i = f_i^*$, we have
\begin{align*}
\int\left(\frac{1}{2\pi}\sum_i\left(\int_{\mb U}(\partial_nHp_i)d\mu\right)\tilde\psi_i-\xi\sum_{i,j}\left(\int_{\mb U}\tilde{p_i} \tilde{p_j}d\mu\right)
\tilde\psi_{i,j}
\right)\rho(dh)=0.
\end{align*}
\end{Lem}

\begin{proof}
Recall $F(h)=\left(\int_{\mb U}\ell d\mu\right)\varphi(\int_{\mb D}f_1h d\lambda,\cdots, \int_{\mb D}f_nh d\lambda)$. Let $RF(h)={\frac{d}{dt}}_{|t=0}F(h(e^{it}\cdot))$; then $\int RFd\rho=0$. Hence, 
$$
\int\left(\left(\int_{\mb U}\partial_\theta \ell d\mu\right)\varphi+\sum_{i}\left(\int_{\mb U}\ell d\mu\right)\left(\int_{\mb U}h\partial_\theta f_i^*d\lambda\right)\varphi_i\right)\rho(dh)=0
$$ 
Since 
$$\int_{\mb U}h\partial_\theta f_i^*d\lambda=\int_{\mb U}h\partial_nH\tilde{f^*_i}d\lambda=\int_{\mb U}\tilde{f^*_i}\partial_nHhd\lambda$$
(see the Section \ref{sec:harmonic-conjugate}) integration by parts with $k = \tilde{f^*_i}$, $\varphi = \varphi_i$ yields
\begin{align*}
\int\left(-\left(\int_{\mb U}\frac{\partial_\theta \ell}{2\pi} d\mu\right)\varphi+\sum_{i}\left(-\xi\int_{\mb U}\ell \tilde{f_i^*}d\mu-2\pi c \dashint\tilde{f_i^*}\int\ell d\mu\right)
\varphi_i
+\sum_{i,j}\left(\int_{\mb U}\ell d\mu\right)\left(\int_{\mb U} \tilde{f_i^*}f_j^*d\lambda\right)\varphi_{i,j}
\right)\rho(dh)=0
\end{align*}
By antisymmetry $\int_{\mb U}(\tilde{f_i^*}f_j^*+\tilde{f_j^*}f_i^*)d\lambda=0$ and since $\int_{\mb{U}} \tilde{f^*_i} d\lambda =0$ we are left with
\begin{equation}
\label{eq:lemma-record}
\int\left(\left(\int_{\mb U}\frac{\partial_\theta \ell}{2\pi} d\mu\right)\varphi+\sum_{i}\left(\xi\int_{\mb U}\ell \tilde{f_i^*}d\mu\right)
\varphi_i
\right)\rho(dh)=0
\end{equation}
Applying this to pairs $(\varphi,\ell)=(\tilde\psi_i,\tilde f^*_i)$, we get
$$
\int\left(\sum_i\left(\int_{\mb U}\frac{\partial_\theta \tilde f^*_i}{2\pi} d\mu\right)\tilde\psi_i+\sum_{i,j}\left(\xi\int_{\mb U}\tilde{f^*_i} \tilde{f_j^*}d\mu\right)
\tilde\psi_{i,j}
\right)\rho(dh)=0
$$
or 
$$
\int\left(\frac{1}{2\pi}\sum_i\left(\int_{\mb U}(\partial_nHf_i^*)d\mu\right)\tilde\psi_i-\xi\sum_{i,j}\left(\int_{\mb U}\tilde{f^*_i} \tilde{f_j^*}d\mu\right)
\tilde\psi_{i,j}
\right)\rho(dh)=0
$$
which completes the proof.
\end{proof}

Now, we conclude on the invariance condition satisfied by $\rho(dh)$. Altogether, as a combination of \eqref{eq:conseq-ibp} and Lemma \ref{Lem:rotational-invariance-v1}, we have
\begin{align}
\int \left(\sum_i\left(\frac{1}{2\pi}\int_{{\mb D}} (D_{\mu}f_i) hd\lambda -2\pi c  \int_{\mb U}p_id\mu+(2\xi+\frac 1{2\xi})\int_{\mb U}\partial_nHp_i d\mu+ \frac{2\pi c-\xi}2\dashint p_i\int_{\mb U}d\mu \right)\tilde\psi_i\right. \nonumber\\
\left.  + \pi\sum_{i,j}\left(\int_{\mb U}p_ip_jd\mu\right)\tilde\psi_{i,j}\right)\rho(dh)=0. \label{eq:final-IBP}
\end{align}
By integration by parts on the zero modes, we also note
$$ (2\pi c-\xi)  \sum_i \dashint_{\mb{U}}  p_i d\lambda |\mu| \tilde{\psi}_i d\rho = (2\pi c-\xi) \int \partial_m \tilde{\psi} | \mu| d\rho   =   (2\pi c - \xi)(2\pi c + \xi) \int |\mu| \tilde{\psi} \rho(dh).$$
We compare the combination of \eqref{eq:final-IBP}  and of the above equation with the boundary localized invariance condition \eqref{eq:condition-boundary}, namely
\begin{align*}
\int \left(\sum_i\left(\frac 1{2\pi}\int_{\mb D} (D_{\mu}f_i) hd\lambda+(\chi-\alpha) \int_{\mb U} p_i d\mu
-\chi\int_{\mb U}\partial_nH p_i d\mu -\beta\dashint_{\mb U} p_i d\lambda \right)\tilde\psi_i +\pi  \sum_{i,j}\left(\int_{\mb U}p_i p_jd\mu\right)\tilde\psi_{i,j} \right)\rho(dh) = 0 
\end{align*}
and we conclude that the two equations coincide when
$$
(- 2\pi c) = \chi-\alpha, \quad 
2\xi +(2\xi)^{-1} = -\chi, \quad
\xi^2 = (2\pi c)^2, \quad \beta = 0.
$$
Altogether, this gives
\begin{thm}
\label{thm-invariance}
For $F(\phi,h)=\psi(\int_{\mb D} f_1(\phi+\tilde h)d\lambda,\dots,\int_{\mb D} f_n(\phi+\tilde h)d\lambda )$ where $f_i\in C^\infty_c({\mb D}\setminus\{0\}),  \psi\in C^\infty_c(\R)$, with at least one $f_i$ for which $\int_{\mb D} f_i d\lambda \neq 0$,  the condition \eqref{eq:invariance-cond}
where $\rho(dh) = \exp(-\frac{1}{4 \pi } \int_{\mb D} | \nabla h |^2 d\lambda  - c \int_{\mb U} h d\lambda) \mc{D}h$ reads, with $p_i = f_i^*$,
$$
\int \left(\sum_i\left((\chi-\alpha+2\pi c) \int_{\mb U} p_i d\mu
-(\chi + 2\xi + (2\xi)^{-1}) \int_{\mb U}\partial_nH p_i d\mu -\beta\dashint_{\mb U} p_i d\lambda \right)\tilde\psi_i - \frac{1}{2}((2\pi c)^2- \xi^2) |\mu| \tilde{\psi}  \right)\rho(dh) = 0. 
$$
\end{thm}

\section{Dirichlet forms}

\label{sec:dirichlet-forms}

In this section, we specialize to the pure gravity case. We first recall below the form of the generator in this case. Then, we derive an expression of the Dirichlet form in section \ref{sec:dirichlet-expression}. In section \ref{sec:dirichlet-other-proof}, we give an alternative (formal, non-rigorous but we believe can be made rigorous with additional efforts) proof of this fact. In section \ref{sec:symmetric-part}, we discuss the symmetric part of the Dirichlet form. In particular, we discuss the associated dynamics and we explain how to make sense of a weak solution associated to this process using the formalism of Dirichlet forms.

In the pure gravity case, an invariant measure associated with the boundary field of the process \eqref{eq:beta=0} is given by
\begin{equation}
\label{eq:invariant-measure}
\rho(dh) = e^{\xi m}  \rho_0(dh_0) dm = \exp \left( - \mc{V}(h) \right) \mc{D} h = \exp \left(  - \frac{1}{4\pi} \int h (- \partial_n H h) d\lambda + \frac{\xi}{2\pi} \int h d\lambda  \right) \mc{D} h
\end{equation}
where the field considered is given by $h = h_0 + m = (h - \dashint_{\mb{U}} h d\lambda) + \dashint_{\mb{U}} h d\lambda$. For $F = \varphi( \int h p_1 d\lambda, \dots, \int h p_n d\lambda)$,
\begin{equation}
\label{eq:generator-pure-gravity}
\mc{L} F = \sum_i b(p_i) \partial_i \varphi + \frac{1}{2} \sum_{i,j} \sigma(p_i, p_j) \partial_{i,j} \varphi
\end{equation}
where $\sigma(p,q) = (2\pi)^2 \int_{\mb{U}} pq d\mu$ and
\begin{align*}
b(p)  =& \int_{\mb{D}} h (D_{\mu} f) d\lambda + 2\pi (2\xi + \frac{1}{2\xi}) \int_{\mb{U}} \partial_n H  p d\mu  + 2\pi \xi \int_{\mb{U}} p d\mu  \\
= &  \int_{\mb{U}^2} \partial_n H h (w')  V_{p}(w',w) \lambda(dw') \mu(dw)  + 2\pi (2\xi + \frac{1}{2\xi}) \int_{\mb{U}} \partial_n H  p d\mu  + 2\pi \xi \int_{\mb{U}} p d\mu  
\end{align*}
This is obtained by taking
$$
\xi = \frac{1}{\sqrt{6}} \quad Q = \frac{5}{\sqrt{6}} \quad \chi = -Q \quad \alpha = -2Q- \omega \quad 2 \pi c = - \xi \quad -\omega = \sqrt{\frac{8}{3}} \quad \alpha = - \sqrt{6}, \quad \beta = 0,
$$
and using specifically that $\chi - \alpha = Q + \omega = \xi$ in \eqref{eq:L-boundary-loc}.

 We provide below formal SPDE associated with this generator. In what follows and until the beginning of Section \ref{sec:integrability}, the computations are not rigorous and derived only to provide a formal SPDE,
\begin{equation}
\label{eq:Dynamic-Global-Loc}
\dot{h}_t(w) =  2\pi \left( \partial_n H h  e^{- \xi h} -  \widetilde{\partial_n H h}  \widetilde{e^{- \xi h}}  \right) + 2\pi  Q  \partial_n H (e^{-\xi h(w)}) + 2\pi \xi e^{-\xi h(w)} + 2\pi e^{-\frac{1}{2}\xi h} W.
\end{equation}
Here $W(dw,dt)$ is an $L^2(\lambda)$ space-time white noise on $\mb{U}$. We compute (formally)
\begin{align*}
- \int_{\mb{U}} (D_{\mu} f) h d\lambda = & \iint_{\mb{U}^2} \Im \left( \frac{w'+w}{w'-w} \right) (\tilde{p}(w') - \tilde{p}(w)) \partial_n H h(w') d\mu_h(w) \\
 = & 2\pi \iint_{\mb{U}^2} \frac{1}{2\pi} \Im \left( \frac{w'+w}{w'-w} \right) (\tilde{p}(w') - \tilde{p}(w)) \partial_n H h(w') e^{-\xi h(w)} d\lambda(w) d\lambda(w') \\
 = & 2\pi \int \tilde{p} (w') \partial_n H h(w') \left( \int  \frac{1}{2\pi} \Im \left( \frac{w'+w}{w'-w} \right) e^{-\xi h(w)} d\lambda(w) \right) d\lambda (w')  \\ 
& - 2\pi \int \tilde{p}(w) e^{-\xi h(w)} \left( \int \frac{1}{2\pi} \Im \left( \frac{w'+w}{w'-w} \right) \partial_n H h(w') d\lambda(w') \right) d\lambda(w) \\
=  &  - 2\pi \int \tilde{p} (w') \partial_n H h(w') \widetilde{e^{-\xi h}} (w') d\lambda (w')  - 2\pi \int \tilde{p}(w) e^{-\xi h(w)} \widetilde{\partial_n H h} (w) d\lambda(w) \\
= & 2\pi \int p(w') \widetilde{\left( \partial_n H h \  \widetilde{e^{-\xi h}} \right)} (w') d\lambda(w') + \int p(w) \widetilde{ \left( \widetilde{\partial_n H h} \ e^{-\xi h} \right) }(w) d\lambda(w) \\
= & -2 \pi \int p(w)  \left( \partial_n H h  e^{- \xi h} -  \widetilde{\partial_n H h}  \widetilde{e^{- \xi h}} + 0 \right) d\lambda(w) 
\end{align*}
and altogether, an associated SPDE is
\begin{equation}
\dot{h}_t(w) = 2\pi\left( \partial_n H h + \xi  \right) e^{-\xi h(w)} - 2\pi  \widetilde{\partial_n H h}  \widetilde{e^{- \xi h}}  + 2\pi   Q  \partial_n H (e^{-\xi h(w)}) + 2\pi e^{-\frac{1}{2}\xi h} W.
\end{equation}
It is not clear how to make sense of several of these terms. This can also be written as
\begin{equation}
\dot{h}_t(w) = - 2\pi \wt{\partial_n H h \wt{e^{-\xi h}}} + \xi   e^{-\xi h(w)}  - 2\pi \xi^{-1} (1-\xi  Q  ) \partial_n H (e^{-\xi h(w)}) + 2\pi e^{-\frac{1}{2}\xi h} W 
\end{equation}
where we used the identity $\partial_n H (e^{-\xi h}) = - \partial_\theta (\widetilde{e^{-\xi h}}) = + \xi \widetilde{\partial_\theta h e^{-\xi h}} = \xi \widetilde{(\widetilde{\partial_n H h } e^{-\xi h})}$.

\subsection{Integrability}
\label{sec:integrability}

In this section, we prove the integrability statements that were postponed. Consider a function $G = \psi (\int_{\mb{U}} q_1 h d\lambda, \dots \int_{\mb{U}} q_n h d\lambda)$ with $\psi$ compactly supported and with at least one $j$ such that $\int_{\mb{U}} q_j d\lambda \neq 0$. Above, we used the integrability of terms of the form
$$
\iint_{\mb{D}} (D_{\mu} f) h d\lambda \partial_i \psi \rho(dh), \qquad   \iint_{\mb{U}} q_i d\mu \partial_i \psi \rho(dh), \qquad \iint_{\mb{U}} \partial_n H q_i d\mu \partial_i \psi \rho(dh), \qquad \iint_{\mb{U}} p_i p_j d\mu \partial_{i,j} \psi \rho(dh)
$$
where $f$ is compactly supported in $\mb{D}$. Suppose $q_i \in H^s(\mb{U})$ and $s>3/2$ and set $q_{\infty} := \max_{i \leq m} \| q_i \|_{L^{\infty}(\mb{U})}+ \| \partial_n H q_i \|_{L^{\infty}(\mb{U})}$. Then $q_{\infty} < \infty$ by the Sobolev embedding theorem. Note that the $q_i$'s we consider are always in this Sobolev space. Above, they are in particular always of the form $q = H^* f$ for some $f$ with compact support in $\mb{D} \backslash \{  0 \}$. We use the following lemma to show that all these terms are well-defined.

\begin{Lem}
Suppose $q_i \in H^s(\mb{U})$ and $s>3/2$ with at least one $i$ such that $\int_{\mb{U}} q_i d\lambda \neq 0$. For an arbitrary compactly supported nonnegative function $\varphi_a$, $\delta \in \mb{R}$, we have, for all $\xi \in (-1,1)$,
\begin{equation}
\label{lem:part1}
\int   \varphi_a \left( \int_{\mb{U}} h q_1 d\lambda, \dots, \int_{\mb{U}} h q_m  d\lambda \right)  | \mu | e^{\delta m} d\rho < \infty
\end{equation}
and, with $f$ compactly supported in $\mb{D}$,
\begin{equation}
\label{lem:part2}
\iint_{\mb{D}} \varphi_a \left( \int_{\mb{U}} h q_1 d\lambda, \dots, \int_{\mb{U}} h q_m d\lambda \right)  | (D_{\mu} f) h d\lambda | \rho(dh) < \infty.
\end{equation}
\end{Lem} 

\begin{proof}
Assume $\xi \in (-1,1)$. By assumption, there exists $i$ such that $\int_{\mb{U}} q_i d\lambda \neq 0$. Furthermore, since $\varphi_a$ is bounded, it is sufficient to use the following inequalities
\begin{align*}
\int \ind_{-C \leq \int_{\mb{U}} h q_i d\lambda \leq C } |\mu_h| d\rho & = \int \ind_{-C \leq \int_{\mb{U}}  h_0 q_i d\lambda + 2\pi m \dashint q_i \leq C}  |\mu_{h_0}| d\rho_0 dm \\
& \leq C\int (C+ \left| \int h_0 q_i d\lambda \right|)  |\mu_{h_0}| d\rho_0   < \infty
\end{align*}
by using H\"older inequality with $\E(|\mu_{h_0}|^p)<\infty$ for $p$ close enough to $1$, the other term being a Gaussian variable, it has finite exponential moments. The generalization to $\delta \neq 0$ is straightforward.

Now, we focus on the term containing $\int_{\mb{D}} (D_{\mu} f) h d\lambda$. First, note that $D_{\mu} f$ has zero mean (see \eqref{eq:dmu-expressions}) so decomposing $h = h_0 + m$ gives $\int_{\mb{D}} (D_{\mu} f) h d\lambda = \int_{\mb{D}} (D_{\mu} f) h_0 d\lambda$. Furthermore, we have
\begin{align*}
\int  \left| \Re \left( \partial_{z} (z \frac{z+w}{z-w} f(z)) \right) h(z) \mu(dw) \lambda(dz) \right| \varphi_a(\dots) \rho(dh) & \leq C \int |\mu_{h_0}| \sup_{z \in \mathrm{Sup}(f)} |h_0(z)| \ind_{-C \leq \int_{\mb{U}} h q_i d\lambda \leq C }  d\rho_0 dm \\
& \leq C \int |\mu_{h_0}| \sup_{z \in \mathrm{Sup}(f)} |h_0(z)|  (C+ \left| \int h_0 q_i d\lambda \right|) d\rho_0
\end{align*}
and we use H\"older inequality with $\E(|\mu_{h_0}|^p)<\infty$ for $p$ close enough to $1$. Here, we use that $h_0$ is a smooth Gaussian function on any compact subset of $\mb{D}$, by Fernique theorem it has finite exponential moments. \end{proof}

Below, we will also need the following lemma, associated with the chaos measure $\mu_{\xi}(h) = e^{\xi h}$ for $\xi \in (0,1)$, when we consider the symmetric Dirichlet form.

\begin{Lem}
\label{lem:integrability-sym}
If $F = \varphi ( |\mu_\xi|, \int p_1 d\mu_\xi, \dots, \int p_n d\mu_\xi)$ where the $p_i$'s are smooth and $\varphi$ is nonnegative and has compact support included in $(\delta,\delta^{-1}) \times K$ where $K$ is a compact of $\mb{R}^n$, then $\int F | \mu_\xi| d\rho < \infty$. Furthermore, when $q$ is smooth, $\int F | \langle \partial_n H h, q \rangle_{L^2(\lambda)} | |\mu_\xi | d\rho < \infty$.
\end{Lem}

\begin{proof}
The Lebesgue measure of $m$'s such that $\{ \delta \leq e^{\xi m} | \mu_\xi(h_0) | \leq \delta^{-1} \}$ is bounded and independent of $h_0$. On this event, $e^{\xi m } \leq \delta^{-1} |\mu_{\xi}(h_0)|^{-1}$ so
$\int F | \mu_\xi|  d\rho \leq \int F e^{2 \xi m} |\mu_\xi(h_0)| dm \rho_0(d h_0) \leq C \E ( | \mu_{\xi}(h_0)|^{-1}) < \infty$. For the second assertion, note that $\langle \partial_n H h, q \rangle_{L^2(\lambda)}$ does not depend on $m$. Following the same inequalities, we bound it from above by $C \E( \langle \partial_n H h_0, q \rangle_{L^2(\lambda)} |\mu_{\xi}(h_0)|^{-1})$ and we conclude by using the Cauchy-Schwarz inequality. \end{proof}

\subsection{Expression of the Dirichlet form}
\label{sec:dirichlet-expression}

Here, we derive the expression of the Dirichlet form \eqref{eq:Dirichlet-Form}. This is the content of the following theorem, whose proof relies on the technology developed in Section \ref{sec:section-gen} and Section \ref{sec-invariance} but with slight differences.
\begin{thm}
\label{thm:dirichlet-form}
We suppose that $\gamma = \sqrt{8/3}$, $\xi = \gamma/d_{\gamma}$ and $\rho(dh)$ is the $\sigma$-finite measure given in \eqref{eq:invariant-measure}. Then, for bulk cylindrical test functionals $F = \varphi \left( \int_\mb{D} f_1 h d\lambda, \dots, \int_\mb{D} f_n h d\lambda \right)$ and $G = \psi \left( \int_\mb{D} g_1 h d\lambda, \dots, \int_\mb{D} g_m h d\lambda \right)$,   
\begin{align*}
\mc{E}(F,G):= \int F (-\mc{L} G) d\rho = & 2\pi^2 \int \langle DF, DG \rangle_{L^2(\mu)} d\rho + 2\pi^2  \iint_{\mb{U}} \left( \widetilde{\widetilde{DF} DG} - \widetilde{DF \widetilde{DG}} \right) d\mu d\rho  \\
& + 2\pi^2 \xi  \int \left( \int_{\mb{U}} DF d\lambda \cdot G - \int_{\mb{U}} DG d\lambda \cdot F \right) | \mu | d\rho.  \nonumber
\end{align*}
\end{thm}

First, we state and prove the integration by parts designed for the term $\int_{\mb{D}} (D_{\mu} f) h d\lambda$ but this time with an additional product. We recall that $V_p$ is defined in \eqref{eq:Vp-expression}.

\begin{Lem}[Integration by parts for $h D_{\mu} f$]
\label{Lem:ibp-muh} For $p = f^*$ where $f \in C_c^{\infty}(\mb{D})$ and bulk cylindrical test functionals $F = \varphi \left( \int_\mb{D} f_1 h d\lambda, \dots, \int_\mb{D} f_n h d\lambda \right)$ and $G = \psi \left( \int_\mb{D} g_1 h d\lambda, \dots, \int_\mb{D} g_m h d\lambda \right)$,   
\begin{align*}
0 = & \int \left[ \int_{\mb{D}} (D_{\mu} f) h d\lambda + 2\pi \xi \int_{\mb{U}} (p-\dashint p) d\mu + 2\xi 2\pi \int_{\mb{U}} \partial_n H p(w) d\mu(w) \right] \varphi \psi d\rho \\
& + \int 2\pi \left[ \sum_i \int p_i(w) V_p(w,w') d\lambda(w) d\mu(w') \varphi_i \psi  + \sum_{i'} \int q_{i'}(w) V_p(w,w') d\lambda(w) d\mu(w') \varphi \psi_{i'} \right] d\rho 
\end{align*}
\end{Lem}

\begin{proof}
For
$$
K(h) = \left( \int_{\mb{U}} \ell d\mu \right)  \varphi \left( \int_\mb{D} f_1 h d\lambda, \dots, \int_\mb{D} f_n h d\lambda \right) \psi \left( \int_\mb{D} g_1 h d\lambda, \dots, \int_\mb{D} g_m h d\lambda \right)
$$
we have, with $p_i = f_i^*$ and $q_{i'}=g_{i'}^*$, and $k \in C^{\infty}(\mb{U})$
$$
D_k K(h) = \left( - \xi \int_\mb{U} \ell k d\mu \right) \varphi(\dots) \psi(\dots) + \sum_i \int p_i k d\lambda \left( \int_{\mb{U}} \ell d\mu \right) \varphi_i \psi + \sum_{i'} \int q_{i'} k d\lambda \left( \int_{\mb{U}} \ell d\mu \right) \varphi \psi_{i'}
$$ 
Recall the IBP formula \eqref{eq:IBP-potential} $-\frac{1}{2\pi} \int K \int_{\mb{U}} k (\partial_n H h + \xi) d\lambda  d\rho = \int D_k K d\rho$. Here, this gives
\begin{align*}
0 = & \int \left( \frac{1}{2\pi} \int_{\mb{U}} k (\partial_n H h) d\lambda \left( \int_{\mb{U}} \ell d\mu \right) - \xi \int_{\mb{U}} \ell (k - \dashint_{\mb{U}} k d\lambda) d\mu \right) \varphi \psi d\rho \\
& + \int\left( \sum_i \int_{\mb{U}} p_i k d\lambda \left( \int_{\mb{U}} \ell d\mu \right) \varphi_i \psi  + \sum_{i'} \int_{\mb{U}}q_{i'} k d\lambda \left( \int_{\mb{U}} \ell d\mu \right) \varphi \psi_{i'} \right) d\rho
\end{align*}

For a given $f$ or $p = f^*$, we use the decomposition $V_p(w,w') = \int k_u(w) \ell_u(w') d\Lambda(u)$ and integrate the previous equation over $d\Lambda(u)$,
\begin{align*}
0 = & \int \left[ \frac{1}{2\pi} \int_{\mb{U}^2} V_p(w,w') (\partial_n H h)(w) d\lambda(w) d\mu(w') + \frac{\xi}{2\pi} \int_{\mb{U}^2} V_p(w,w') d\lambda(w) d\mu(w') - \xi \int_{\mb{U}} V_p(w,w) d\mu(w) \right] \varphi \psi d\rho \\ 
& + \int \left[ \sum_i \int_{\mb{U}^2} p_i(w) V_p(w,w') d\lambda(w) d\mu(w') \varphi_i \psi + \sum_{i'} \int_{\mb{U}^2} q_{i'}(w) V_p(w,w') d\lambda(w) d\mu(w') \varphi \psi_{i'} \right] d\rho
\end{align*} 
For $p = f^*$, multiplying by $2\pi$ and using contractions gives
\begin{align*}
0 = & \int \left[ \int_{\mb{D}} (D_{\mu} f) h d\lambda + 2\pi \xi \int_{\mb{U}} (p-\dashint_{\mb{U}} p d\lambda) d\mu + 2\xi 2\pi \int_{\mb{U}} \partial_n H p(w) d\mu(w) \right] \varphi \psi d\rho \\
& + \int 2\pi \left[ \sum_i \int_{\mb{U}^2} p_i(w) V_p(w,w') d\lambda(w) d\mu(w') \varphi_i \psi  + \sum_{i'} \int_{\mb{U}^2} q_{i'}(w) V_p(w,w') d\lambda(w) d\mu(w') \varphi \psi_{i'} \right] d\rho 
\end{align*}
This proves the lemma.
\end{proof}

We need the following lemma which comes from a combination of rotational invariance and of the integration by parts formula.
\begin{Lem}
\label{Lem:rot-invariance}
For smooth $\ell$ and bulk cylindrical test function $F$, we have
$$
\int  \left( \int_{\mb{U}} \partial_\theta \ell d\mu \right) F(h) d\rho + 2\pi \xi \int \int_{\mb{U}} \ell  \widetilde{DF} d\mu d\rho = 0.
$$
\end{Lem}

\begin{proof}
This is equivalent to \eqref{eq:lemma-record}.
\end{proof}

\begin{proof}[Proof of \eqref{eq:Dirichlet-Form}]
By applying Lemma \ref{Lem:ibp-muh} with $(p_i,\varphi_i)$ and summing over $i$ we get
\begin{align}
0 =  & \int \sum_i  \left[ \int_{\mb{D}} (D_{\mu} f_i) h d\lambda + 2\pi \xi \int_{\mb{U}} (p_i - \dashint p_i) d\mu + 2\xi 2\pi \int_{\mb{U}} \partial_n H p_i d\mu \right] \varphi_i \psi d\rho  \\
& + 2\pi \int \sum_{i,j} p_j(w) V_{p_i}(w,w') d\lambda(w) d\mu(w') \varphi_{i,j} \psi d\rho + 2\pi \int \sum_{i,j'} \int q_{j'}(w) V_{p_i}(w,w') d\lambda(w) d\mu(w') \varphi_i \psi_{j'} d\rho  \nonumber
\end{align}
We focus on the fourth term. Using contractions, we get
\begin{equation}
2\pi \sum_{i,j} \int_{\mb{U}^2} p_j(w) V_{p_i}(w,w') d\lambda(w) d\mu(w') \varphi_{i,j} \psi = 2\pi^2 \sum_{i,j} \int_{\mb{U}} \left( p_i p_j + \tilde{p}_i \tilde{p}_j - \dashint p_i \dashint p_j \right) d\mu  \varphi_{i,j} \psi
\end{equation}
We apply Lemma \ref{Lem:rot-invariance} with $(\ell,F) = (\tilde{p}_j,\varphi_j \psi)$, use $\partial_\theta \tilde{p}_j = - \partial_n H p_j$ and sum over $j$'s so that
$$
- \sum_j \left( \int \partial_n H p_j d\mu  \right) \varphi_j \psi d\rho + 2\pi \xi \int \left( \sum_{i,j} \left( \int \tilde{p}_i \tilde{p}_j d\mu \right) \varphi_{i,j} \psi + \sum_{i',j} \left( \int \tilde{p}_j \tilde{q}_{i'} d\mu \right) \varphi_j \psi_{i'} \right) d\rho
$$
Furthermore, an integration by parts on the zero modes gives
$$
2\pi^2 \int \sum_{i,j} \dashint p_i \dashint p_j | \mu| \varphi_{i,j} \psi d\rho + 2\pi^2 \int \sum_{i,j'} \dashint p_i \dashint q_{j'} | \mu | \varphi_i \psi_{j'} d\rho =0 
$$
Combining this with the numbered equations, we get
\begin{align*}
& 2\pi \int \sum_{i,j} \int_{\mb{U}^2} p_j(w) V_{p_i}(w,w') d\lambda(w) d\mu(w') \varphi_{i,j} \psi d\rho =  2\pi^2 \int \sum_{i,j} \int_{\mb{U}} \left( p_i p_j + \tilde{p}_i \tilde{p}_j  - \dashint p_i \dashint p_j \right) d\mu  \varphi_{i,j} \psi d\rho \\
= & 2\pi^2 \int \sum_{i,j} \left( \int_{\mb{U}}  p_i p_j  d\mu \right)   \varphi_{i,j} \psi d\rho + \frac{\pi}{\xi} \int \sum_j \left( \int \partial_n H p_j d\mu  \right) \varphi_j \psi d\rho - 2\pi^2 \int  \sum_{i',j} \left( \int \tilde{p}_j \tilde{q}_{i'} - \dashint p_j \dashint q_{i'} d\mu \right) \varphi_j \psi_{i'}  d\rho 
\end{align*}
So altogether, we get
\begin{align}
\label{eq:IBP-simplification}
0 =  & \int \sum_i  \left[ \int_{\mb{D}} (D_{\mu} f_i) h d\lambda + 2\pi \xi \int_{\mb{U}} (p_i - \dashint p_i) d\mu + 2\pi (2\xi  + \frac{1}{2\xi}) \int_{\mb{U}} \partial_n H p_i d\mu \right] \varphi_i \psi d\rho  \nonumber \\
& + 2\pi \int \sum_{i,j'} \int q_{j'}(w) V_{p_i}(w,w') d\lambda(w) d\mu(w') \varphi_i \psi_{j'} d\rho \\
& + 2\pi^2 \int \sum_{i,j} \left( \int_{\mb{U}}  p_i p_j  d\mu \right)   \varphi_{i,j} \psi d\rho - 2\pi^2 \int  \sum_{i',j} \left( \int \tilde{p}_j \tilde{q}_{i'} - \dashint p_j \dashint q_{i'} d\mu \right) \varphi_j \psi_{i'}  d\rho  \nonumber
\end{align}
namely, recalling $b_i$ from \eqref{eq:generator-pure-gravity},
\begin{align*}
\int \sum_i b_i \varphi_i \psi d\rho = & - 2\pi^2 \int \sum_{i,j} \left( \int_{\mb{U}}  p_i p_j  d\mu \right)   \varphi_{i,j} \psi d\rho \\
& - 2\pi \int \sum_{i,j'} \int q_{j'}(w) V_{p_i}(w,w') d\lambda(w) d\mu(w') \varphi_i \psi_{j'} d\rho \\
& + 2\pi^2 \int  \sum_{i',j} \left( \int \tilde{p}_j \tilde{q}_{i'} - \dashint p_j \dashint q_{i'} d\mu \right) \varphi_j \psi_{i'}  d\rho + 2\pi \xi \sum_i \int \dashint p_i |\mu| \varphi_i \psi d\rho 
\end{align*}
so that, recalling $\mc{L} F$ from \eqref{eq:generator-pure-gravity},
\begin{align*}
- \int G \mc{L} F d\rho   = &  2\pi \int \sum_{i,j'} \int q_{j'}(w) V_{p_i}(w,w') d\lambda(w) d\mu(w') \varphi_i \psi_{j'} d\rho \\
& - 2\pi^2 \int  \sum_{i',j} \left( \int \tilde{p}_j \tilde{q}_{i'} - \dashint p_j \dashint q_{i'} d\mu \right) \varphi_j \psi_{i'}  d\rho - 2\pi \xi \sum_i \int \dashint p_i |\mu| \varphi_i \psi d\rho  
\end{align*}
or, more concisely,
\begin{equation}
\label{eq:concise-id}
\begin{split}
- \int G \mc{L} F d\rho = & 2\pi \int DG(w) V_{DF}(w,w') d\lambda(w) d\mu(w') d\rho \\
&- 2\pi^2 \int \left( \widetilde{DF} \widetilde{DG} -\dashint DF \dashint DG  \right) d\mu d\rho - 2\pi \xi \int \left( \dashint DF \right) G  |\mu| d\rho
\end{split}
\end{equation}
and, using $qV_p + pV_q = 2\pi (pq + \tilde{p} \tilde{q} - \dashint p \dashint q)$ and an integration by parts of the zero modes, we find that the symmetric part $\tilde{\mc{E}}(F,G)$ is given by
\begin{equation}
\label{eq:SymmetricPart}
\frac{1}{2} \int  F (-\mc{L} G ) + G (- \mc{L} F) d\rho  = 2\pi^2 \int \langle DF, DG \rangle_{L^2(\mu)} d \rho
\end{equation}
and that the antisymmetric part $\check{\mc{E}}(F,G)$ is given by
\begin{equation}
\label{eq:AntisymmetricPart}
\begin{split}
\frac{1}{2} \int  F (-\mc{L} G ) - G (- \mc{L} F) d\rho   = &  \pi \int \left( DF(w) V_{DG}(w,w') - DG(w) V_{DF}(w,w') \right) d\lambda(w) d\mu(w') d\rho \\
& + \pi \xi \int \left( (\dashint DF ) G -    (\dashint DG ) F \right)  |\mu| d\rho.
\end{split}
\end{equation}
Finally, using $\frac{1}{2\pi} (p V_q - q V_p) = \wt{(\wt{p} q - p \wt{q})}$,
\begin{align*}
 \int \left( DF(w) V_{DG}(w,w') - DG(w) V_{DF}(w,w') \right) d\lambda(w) d\mu(w') d\rho =   2\pi \int  \int \left( \widetilde{ \widetilde{DF}DG} - \widetilde{ DF\widetilde{DG}}   \right)  d\mu d\rho
\end{align*}
so that
$$
\check{\mc{E}}(F,G) =  2\pi^2 \int  \int \left( \widetilde{ \widetilde{DF}DG}   -  \widetilde{DF\widetilde{DG}}   \right)  d\mu d\rho + \pi\xi \int   \left( \dashint_{\mb{U}} D F  \right) G - \left( \dashint_{\mb{U}} D G  \right) F | \mu|  d \rho 
$$
and this concludes the proof of \eqref{eq:Dirichlet-Form}.
\end{proof}

\paragraph{Remark.} When using the integration by parts developed here but for a potential $\mc{V}_1(h)  =\mc{V}(h)+\mc{R}(h)$ instead of $\mc{V}(h)$, this brings the additional term
$$
\int F \int \frac{2}{2\pi} D \mc{R} V_{DG} - \wt{DR} \wt{DG} d\mu d\rho
$$
Namely, if we denote by $\rho_{c,\mc{R}}$ the associated measure, we have
\begin{align*}
\int  F(h) \mc{L}_{\alpha,\chi,\beta} G(h) \rho_{c, \mc{R}}(dh) = & \int F \left( -2\pi (\chi+2\xi+\frac{1}{2\xi}) \int_{\mb{U}} \partial_n H DG d\mu + 2\pi (\chi-\alpha+2\pi c) \int_{\mb{U}} DG d\mu - \beta \int_{\mb{U}} DG d\lambda \right)  d\rho_{c,\mc{R}} \\
& + 2\pi^2 \int \left(- \frac{2}{2\pi} DF V_{DG} + \wt{DF} \wt{DG}  - \dashint DF \dashint DG \right) d\mu d\rho_{c,\mc{R}} + \pi (\xi-2\pi c) \int F \dashint DG  |\mu| d\rho_{c,\mc{R}}  \\
& + 2\pi^2 \int F \int_{\mb{U}} \frac{2}{2\pi} D \mc{R} V_{DG} - \wt{D \mc{R}} \wt{DG} d\mu d\rho_{c,\mc{R}}.
\end{align*}
Note also that
$$
\frac{2}{2\pi} D\mc{R} V_{DG}  = \frac{D\mc{R} V_{DG}  + DG V_{D \mc{R}} }{2\pi}   + \frac{D\mc{R} V_{DG}  - DG V_{D \mc{R}}}{2\pi} = DG D\mc{R} + \wt{DG} \wt{D \mc{R}}- \dashint DG \dashint D\mc{R} + \frac{D\mc{R} V_{DG}  - DG V_{D \mc{R}}}{2\pi} 
$$
so, recalling $A$ from \eqref{def:operator-A},
$$
\frac{2}{2\pi} D\mc{R} V_{DG}  - \wt{DG} \wt{D \mc{R}} = DG D\mc{R}  - \dashint DG \dashint D\mc{R} + A(D\mc{R}, DG).
$$
With $F =1$, in order to find an invariant measure by using a similar identification as the one of the previous section, we would like to express this extra term with  $DG d\mu$ and $\partial_n H DG d\mu$.  

\subsection{Proof using a formal integration by parts}

\label{sec:dirichlet-other-proof}

The motivation of the proof presented above is to start from the Gaussian integration by parts characterizing the boundary field to obtain the invariance condition. This required boundary localization of the different terms appearing in the generator and the integration by parts associated with the term $h D_{\mu} f$ was nontrivial. Below, we use formal computations to explain in another way how this integration by parts works. A non-trivial observation here is that certain terms can be written in divergence form.

Recall the general integration by parts formula associated with a measure on fields $\rho = e^{-V} \mc{D} h$ where $\mc{D} h$ is formally the Lebesgue measure (which does not exist).
\begin{equation}
\label{eq:divergence-formula}
\int \langle  DF, B \rangle_{L^2(\lambda)} d\rho = - \int F \Div (e^{-V} B) \mc{D} h = - \int F \left[ \Div B - \langle B, DV \rangle_{L^2(\lambda)} \right] d\rho
\end{equation}
The symmetry $V_q(w,w') = V_q(w',w)$ gives the identity $\la p V_q, r \ra = \la r V_q, p \ra$. This and \eqref{eq:divergence-formula} give
\begin{align}
\int \la DF V_{DG}, e^{-\xi h} \ra d\rho & = \int \la DG, e^{-\xi h} V_{DG} \ra d\rho  \nonumber\\
& =  - \int F \left[ \Div (e^{-\xi h} V_{DG}) - \la e^{-\xi h} V_{DG}, D \mc{V} \ra \right] d\rho \nonumber \\
& = - \int F \left[ \Div (e^{-\xi h} V_{DG}) - \la D\mc{V} V_{DG}, e^{-\xi h} \ra \right] d\rho  
\end{align}

The divergence term is given by
\begin{align*}
\Div (e^{-\xi h} V_{DG})  &  = \sum_{j \geq 0, k \geq 0} \partial_{j,k} \psi \int e_k V_{e_j} e^{-\xi h} d\lambda - \xi \sum_{k \geq 0} \int V_{DG}(w,w') e_k(w) e_k(w') e^{-\xi h(w)} \lambda(dw) \lambda(dw') \\
& = \frac{1}{2} \sum_{j \geq 0, k \geq 0} \partial_{j,k} \psi \int  (e_k V_{e_j} + e_j V_{e_k}) e^{-\xi h} d\lambda + 2\xi \int \partial_n H DG d\mu \\
& = \pi \sum_{j \geq 0, k \geq 0} \partial_{j,k} \psi (\int e_j e_k + \wt{e}_j \wt{e}_k - \dashint e_j \dashint e_k) d\mu+ 2\xi \int \partial_n H DG d\mu 
\end{align*}
since $\sum_{k \geq 0} e_k(w) e_k(w') = \delta_{w}(w')$. 

The term with $\wt{e}_j \wt{e}_k$  can be interpreted as the following divergence
\begin{align*}
\Div \left( \wt{\wt{DG} e^{-\xi h}} \right) & =  - \sum_{k \geq 1} D_{e_k}  \int \wt{DG} e^{-\xi h} \wt{e}_k d\lambda \\
& = - \sum_{k  \geq 1,  j \geq 1} \partial_{j,k} \psi \int \wt{e}_j \wt{e}_k e^{-\xi h} d\lambda + \xi \sum_{k \geq 1} \int \wt{DG} e_k \wt{e}_k e^{-\xi h} d\lambda  \\
& = - \sum_{k  \geq 1,  j \geq 1} \partial_{j,k} \psi \int \wt{e}_j \wt{e}_k e^{-\xi h} d\lambda 
\end{align*}
An integration by parts using this divergence gives
$$
\int F \Div \left( \wt{\wt{DG} e^{-\xi h}} \right)  d\rho = - \int \la DF, \wt{\wt{DG} e^{-\xi h}}  \ra d\rho +   \int F \la  \wt{\wt{DG} e^{-\xi h}}  ,  D\mc{V} \ra  d\rho   =  \int \la\wt{DF}, \wt{DG} e^{-\xi h}\ra -  F \la \wt{DG}, \wt{D \mc{V}}  e^{-\xi h}\ra d\rho
$$
so
$$
-  \pi \int F \sum_{j,k} \partial_{j,k}  \psi  \int \wt{e}_j \wt{e}_k d\mu  d\rho = \pi \int F \Div \left( \wt{\wt{DG} e^{-\xi h}} \right) d\rho = \pi \int \la\wt{DF}, \wt{DG} e^{-\xi h}\ra -  F \la \wt{DG}, \wt{D \mc{V}}  e^{-\xi h}\ra d\rho
$$

Altogether, by multiplying by $2\pi$ we get
\begin{align*}
& \int F \left[ - 2\pi^2 \sum_{j,k \geq 0} \partial_{j,k} \psi \int (e_j e_k - \dashint e_j \dashint e_k) d\mu + 2\pi \la D\mc{V} V_{DG}, e^{-\xi h} \ra  - 2\pi 2\xi \int \partial_n H DG d\mu  -2\pi^2 \la \wt{DG}, \wt{D \mc{V}} e^{-\xi h} \ra  \right] d\rho \\
& =  2\pi^2 \frac{2}{2\pi} \int \la DF V_{DG}, e^{-\xi h} \ra d\rho -  2\pi^2 \int \la \wt{DF}, \wt{DG} e^{-\xi h} \ra d\rho
\end{align*}

Now, we specify $D\mc{V} = -\frac{1}{2\pi} (\partial_n H h + \xi) $.  Then,
$$
- 2\pi \la D \mc{V} V_{DG}, e^{-\xi h} \ra = \la \partial_n H h V_{DG}, e^{-\xi h}  \ra + \xi \la  1V_{DG}, e^{-\xi h} \ra = \la \partial_n H h V_{DG}, e^{-\xi h}  \ra + 2\pi \xi  (DG - \dashint DG)
$$

Using $\wt{\partial_n H p} = \partial_\theta p$ and $\partial_\theta \wt{p} = - \partial_n H p$, 
$$
2\pi^2 \la \wt{DG}, \wt{D \mc{V}} e^{-\xi h} \ra = -\pi \la \wt{DG}, \partial_\theta h e^{-\xi h} \ra = \pi \xi^{-1} \la \wt{DG}, \partial_\theta ( e^{-\xi h}) \ra  = - \pi \xi^{-1} \la \partial_\theta \wt{DG}, e^{-\xi h} \ra = \pi \xi^{-1} \la \partial_n H DG , e^{-\xi h} \ra
$$

In this case, since $2\pi^2 \int F \sum_{j,k \geq 0} \partial_{j,k} \psi |\mu| d\rho = 2\pi^2 \int F \partial_m^2 G |\mu| d\rho =  - 2\pi^2 \int \partial_m F \partial_m G  |\mu| d\rho$, the equation becomes
\begin{align*}
& -\int F \left[ 2\pi^2 \sum_{j,k \geq 0} \partial_{j,k} \psi \int e_j e_k d\mu  + \la \partial_n H h V_{DG}, e^{-\xi h}  \ra + 2\pi \left(  (2\xi )^{-1} + 2\xi \right) \int \partial_n H DG d\mu    + 2\pi \xi  \int DG d\mu  \right] d\rho \\
 = & 2\pi^2 \frac{2}{2\pi} \int \la DF V_{DG}, e^{-\xi h} \ra d\rho -  2\pi^2 \int \la \wt{DF}, \wt{DG} e^{-\xi h} \ra d\rho + 2\pi^2 \int \dashint DF \dashint DG |\mu| d\rho  - 2\pi \xi \int F \dashint DG |\mu | d\rho
\end{align*}
which is equivalent to \eqref{eq:concise-id} by recalling \eqref{eq:generator-pure-gravity}. In particular, we readily retrieve $\int F (- \mc{L} G) d\rho = \mc{E}(F,G)$.

\subsection{Symmetric part: formal dynamics and weak solution}

\label{sec:symmetric-part}

In this section, we study the symmetric part of the Dirichlet form, which is given by 
\begin{equation}
\label{eq:original-symmetric}
\tilde{\mc{E}}(F,G) = 2\pi^2   \int \langle DF, DG \rangle_{L^2(\mu)} d\rho 
\end{equation}
where $\mu = e^{-\xi h}$ and $\rho(dh)$ is given in \eqref{eq:invariant-measure}. When $\xi = 0$, the associated dynamics are simple and given by 
\begin{equation}
\label{eq:conc-circles}
\dot{h}_t = \pi \partial_n H h_t  dt + 2\pi W = -\pi (-\Delta_{\mb{U}})^{1/2} h_t + 2\pi W
\end{equation}
where $W$ is a space-time white noise. This is simply an infinite dimensional Ornstein-Uhlenbeck process and it appears naturally in different contexts such as in the fluctuations of Hastings-Levitov planar growth \cite{Silvestri} or with the characteristic polynomial process  associated with the Dyson Brownian motion on the circle \cite{BF22}. Also, it has an important role in the proof of the conformal bootstrap of Liouville theory \cite{Bootstrap}. However, for any $\xi >0$, it is not immediate to make sense of associated dynamics. One can be tempted to use an integration by parts formula to find the symmetric part of the generator. This gives formally, for $G = \varphi( \int_{\mb{U}} h q_1 d\lambda, \dots \int_{\mb{U}} h q_m d\lambda )$,
\begin{equation}
\label{eq:formal-ibp-sym}
\begin{split}
  \int \langle DF, DG \rangle_{L^2(\mu)} d\rho  = &  \int F \left[ - \sum_{j,j'} \partial_{j,j'} \psi \int_{\mb{U}} q_{j'} q_j d\mu - \frac{\xi}{2\pi}  \int_{\mb{U}} DG d\mu  \right] d\rho \\
& -  \frac{1}{2\pi} \int F \int_{\mb{U}} DG ``\partial_n H h d\mu" d\rho 
\end{split}
\end{equation}
so (formally), the symmetric generator is given by
\begin{equation}
\label{eq:formal-sym-gen}
\mc{L} G  = \pi\int_{\mb{U}} DG ``\partial_n H h  d\mu"  + \pi \xi  \int_{\mb{U}} DG d\mu   + 2\pi^2 \sum_{j,j'} \partial_{j,j'} \psi \int_{\mb{U}} q_{j'} q_j d\mu 
\end{equation}
The reason for the ``$\partial_n H h d\mu$"  is that it is not clear how to make sense of this term (contrary to $\partial_\theta h d\mu$ which can be interpreted as a tangential derivative of the boundary measure). Below, we discuss this formal integration by parts and why this term doesn't make readily sense. But first, we instead use a (formal) change of variable in the associated (formal) dynamics to construct a weak solution associated with this symmetric Dirichlet form.

\subsubsection{Change of variable and existence of a diffusion process}

\paragraph{Formal dynamics.} An associated dynamic is given (formally), for a space-time white noise $W$, by
$$
\dot{h}_t(w) = \pi (\partial_n H h_t  + \xi ) e^{- \xi h_t} +2\pi e^{-\frac{1}{2} \xi h_t} W(dw,dt)
$$
Multiplying this SPDE by $\xi e^{\xi h_t}$, we get (formally)
\begin{equation}
\label{eq:symmetric-SPDE}
\frac{d}{dt} e^{\xi h_t}  =  \pi \xi (\partial_n H h_t  + \xi ) +2\pi \xi e^{\frac{1}{2} \xi h_t} W(dw,dt)
\end{equation}
Consider the $\xi$-GMC associated with the field $h$ denoted by $\mu_{\xi} := e^{\xi h}$. The generator associated with this equation and applied to $F = \varphi (\int_{\mb{U}} p_1 d\mu_{\xi}, \dots, \int_{\mb{U}} p_n d\mu_{\xi})$ is given by
\begin{align}
\label{eq:symmetric-generator}
\mc{L}^S F & := \sum_{i} \partial_i \varphi (\dots)   \pi \xi  \langle  \partial_n H h + \xi , p_i \rangle_{L^2(\lambda)}  + 2 \pi^2 \xi^2 \sum_{i,j} \partial_{i,j} \varphi (\dots) \int_{\mb{U}} p_i p_j d\mu_{\xi}  \\
& =  - 2 \pi^2 \xi \sum_{i} \partial_i \varphi (\dots)  D_{p_i} \mc{V}(h) + 2 \pi^2 \xi^2 \sum_{i,j} \partial_{i,j} \varphi (\dots) \int_{\mb{U}} p_i p_j d\mu_{\xi} \nonumber 
\end{align}  
where we used \eqref{eq:V-exp-DV} in the second equality. Now, this expression is well defined as soon as the $p_i$'s are smooth enough. By integration by parts, in particular by using  $D_{p_i} \mc{V}(h)$ above, we have
\begin{align*}
\int G \partial_i \varphi(\dots) D_{p_i} \mc{V}(h) d\rho & =  \int D_{p_i} (G \partial_i \varphi) d\rho \\
&  = \int \left( \xi  \sum_j \partial_j \psi (\dots)   \partial_i \varphi  (\dots)  \int_{\mb{U}} p_i q_j d \mu_{\xi} + \xi \sum_j G \partial_{i,j} \varphi (\dots) \int_{\mb{U}} p_i p_j d\mu_{\xi} \right) d\rho.
\end{align*}
So, by multiplying this equality by $2\pi^2 \xi$, we obtain an integration by parts formula for the symmetric part of the Dirichlet form,
\begin{equation}
\label{eq:ibp-symmetric}
\mc{E}(F,G) := \int G (-\mc{L}^S F) d\rho = 2 \pi^2  \int \langle \mc{D}F, \mc{D}G \rangle_{L^2(\mu_{\xi})} d\rho,
\end{equation}
where $\mc{D}F$ denotes the $L^2(\mu_\xi)$-gradient of $F$.

This corresponds to a (formal) change of variable in the original symmetric Dirichlet form \eqref{eq:original-symmetric}: for the test functional $F = \varphi (\int_{\mb{U}} e^{\xi h} p_1 d\lambda, \dots, \int_{\mb{U}} e^{\xi h} p_n d\lambda)$,  $DF = \sum_i \partial_i \varphi(\dots)  \xi e^{\xi h} p_i$ so, with the $\xi$-GMC $\mu_\xi = e^{\xi h}$, $\langle DF, DG \rangle_{L^2(\mu)} =  \xi^2 \sum_{i,j}  \partial_i \varphi (\dots) \partial_j \psi (\dots) \int_{\mb{U}} p_i q_j d\mu_{\xi} = \langle \mc{D}F, \mc{D}G \rangle_{L^2(\mu_{\xi})}$.

\paragraph{Weak solution of \eqref{eq:symmetric-SPDE}.} We construct, using the formalism of Dirichlet forms, an associated stochastic process and argue by using Fukushima decomposition that the process is a weak solution of \eqref{eq:symmetric-SPDE}. Though this process naturally appears here for $\xi = 1/\sqrt{6}$, we can make sense of it for every $\xi$ in the subcritical regime of GMC measures, namely for $\xi \in (0,1)$. 

Using the push-forward given by the GMC map $h \mapsto e^{\xi h}$, the  $\sigma$-finite measure $\rho(dh)$ on fields induces a $\sigma$-finite measure on the space of Borel measures on  $\mb{U}$,  $X := \mc{M}(\mb{U})$, which we denote by $\m$. With this notation, the Dirichlet form we consider is given by the right-hand side of \eqref{eq:ibp-symmetric}, namely, for $F, G \in \mc{C}_S$ where $\mc{C}_S := \{ F = \varphi (|\mu_\xi|, \int p_1 d\mu_{\xi}, \dots, \int p_n d\mu_{\xi}) \}$, with $\varphi$ with compact support in $(0,\infty) \times \mb{R}^n$,
\begin{equation}
\label{def:symmetric-dirichlet-form}
\mc{E}(F,G) := 2 \pi^2 \int \langle \mc{D}F, \mc{D}G \rangle_{L^2(\mu_{\xi})} d\m
\end{equation}
where $\langle \mc{D}F, \mc{D}G \rangle_{L^2(\mu_{\xi})}  =  \xi^2 \sum_{i,j}  \partial_i \varphi (\dots) \partial_j \psi (\dots) \int_{\mb{U}} p_i q_j d\mu_{\xi}$. The integrability is justified by Lemma \ref{lem:integrability-sym}.

 $\mc{M}(\mb{U})$ is locally compact for the topology of weak convergence (which is metrizable).  We add a cemetery point $\Delta$ to the state space seen as the point at infinity and set $X_{\Delta} = X \cup \{ \Delta \}$.

We need a variant of the inverse mapping from \cite{BSS14} which allows to retrieve the free field from its associated LQG measure in the bulk. To be more precise, we need a one dimensional version of this result for which we retrieve $h$ from $e^{\xi h}$. Such an inverse mapping exists and the proof of \cite{BSS14} can be adapted (see Section \ref{sec:inverse}). We denote it by $I_{\xi}$.

\begin{Prop}
\label{prop:weak-solution}
For every $\xi \in (0,1)$, there exists an $\m$-symmetric diffusion  $( (\mu_{t}^{\xi})_{t\geq 0}, (P_{\mu_0})_{\mu_0 \in \mc{M}(\mb{U})})$ on $\mc{M}(\mb{U})$ such that for any smooth function $p$ and $\m$-every $\mu_0 \in \mc{M}(\mb{U})$, under $P_{\mu_0}$, $\mu_0^{\xi} = \mu_0$ and
\begin{equation}
\label{eq:1d-projections}
d \int_{\mb{U}} p(w) \mu_t^{\xi}(dw) = \pi \xi \int_{\mb{U}} p (\partial_n H h_t  + \xi ) d\lambda dt +2\pi \xi \left( \int_{\mb{U}} p(w)^2 \mu_t^{\xi}(dw) \right)^{1/2} d\beta_t^p
\end{equation}
where, for $t>0$ a.s. $h_t = I_{\xi} (\mu_t^{\xi})$ and $(\beta_t^p)_{t \geq 0}$ is a standard Brownian motion.
\end{Prop}

A diffusion is called $\m$-symmetric if its associated semigroup $P_t$ satisfies $\int f P_t g d\m= \int g P_t f d\m$. If the symmetric process has an initial distribution absolutely continuous w.r.t. the symmetrizing measure $\m$, then for every $t>0$, its distribution is still absolutely continuous w.r.t. $\m$ since
$$
\E_{f d\m} [ g(X_t)  ] = \E_{f d\m} [P_t g ] = \int f P_t g d\m = \int (P_t f) g d\m = \int g \rho_t d\m
$$
so the Radon-Nikodym derivative of $X_t$ under $f d\m$ is given by $\rho_t = P_t f$.

Quasi-everywhere (q.e.) is a notion associated to the Dirichlet form, in particular to its capacity (see \cite[Chapter 2]{FOT11}). A statement is said to hold q.e. if there exists a set $N$ of zero capacity such that the statement is true for $x \notin N$. This introduces exceptional sets finer than $\m$-negligible set: any set of zero capacity is $\m$-negligible.

\paragraph{Remark.} The weak solution $(\mu_t^{\xi}) = (e^{\xi h_t})$ induces a natural growth process. Indeed, we can consider the growth evolution associated with the driving measure $e^{-\xi h_t}$ of the Loewner-Kufarev equation. To consider this evolution, it is sufficient to define the boundary measure $e^{-\xi h_t}$ for Lebesgue every time $t$ in order to obtain the conformal maps $(g_t)$. Here, we retrieve first $h_t$ from $e^{\xi h_t}$ for Lebesgue every time $t$, using $I_{\xi}$, and then define $e^{-\xi h_t}$, again for Lebesgue every time $t$.

\paragraph{Total mass process.} We consider the ``total mass process" associated with $(\mu_t^{\xi}) = (e^{\xi h_t})$, namely
\begin{equation}
X_t := \int_{\mb{U}} e^{\xi h_t} d\lambda := |\mu_t^{\xi}|,
\end{equation}
Using Proposition \ref{prop:weak-solution}, we get, for some standard Brownian motion $(B_t)$,
$$
d X_t = 2\pi^2 \xi^2 dt + 2\pi \xi \sqrt{X_t} dB_t 
$$
With $Y_t := \sigma X_t$ and $2 \pi \xi \sqrt{\sigma} =1$, we have
$$
d Y_t = \sqrt{Y_t} dB_t + \frac{1}{2} dt.
$$
$X_t$ is, up to a multiplicative factor, a $\frac{1}{2}$-dimensional squared Bessel process (see, e.g., Section 8.4.3 in \cite{LeGall-calculus}).

\begin{proof}[Proof of Proposition \ref{prop:weak-solution}]
We first recap the plan to prove this proposition and mention explicitly where the steps are proved. The main parts below consists in proving first the existence of a diffusion and then that the dynamics \eqref{eq:1d-projections}  hold.

First, we need to show that the symmetric form \eqref{def:symmetric-dirichlet-form} is \textit{Markovian} and \textit{closable} (see \cite[Section 1.1]{FOT11} for a precise definition; roughly speaking the Markov property states that if $g$ is a contraction then $\mc{E}(g\circ F, g\circ F) \leq \mc{E}(F,F)$, the form is closed if the set $D(\mc{E}) \subset L^2(\m)$ on which $\mc{E}$ is defined is complete with the metric $\mc{E}_{\alpha}(F,G) = \mc{E}(F,G) + \alpha \int FG d\m$ for some $\alpha >0$). Then, it has a closed extension which we still denote by $\mc{E}$.  Following \cite{FOT11}, a symmetric form which is Markovian and closed is a \textit{Dirichlet form} (this is a definition).  We use the theory of Dirichlet form to construct an associated process. In particular, from \cite[Section 7.2]{FOT11}, any \textit{regular} Dirichlet form admits a Hunt process (Theorem 7.2.1).  This asks for the Dirichlet form to have a ``core" namely a subset
$\mc{C}$ of $\mc{D}(\mc{E}) \cap C(X)$ such that $\mc{C}$ is dense in $\mc{D}(\mc{E})$ with norm $\mc{E}_{\alpha}$ and dense in the space
of continuous functions $C(X)$ with uniform norm. So, we show that $\mc{E}$ is regular. From \cite[Section 4.5]{FOT11}, the \textit{local property} of a regular Dirichlet form (namely, $\mc{E}(F,G) = 0$ whenever $F, G \in \mc{D}(\mc{E})$ have disjoint compact supports) implies the existence of an $\m$-symmetric Hunt process $(\mu_t^{\xi})$ that possesses continuous sample paths with probability one (for an introduction to Hunt processes, see \cite[Appendix A.2]{FOT11}). Such a process is called an $\m$-symmetric diffusion. 

A proof of the Markov property  can be found in the proof of \cite[Proposition 3.9]{DS19}. Closability is proved in  \cite[Lemma 3.8]{DS19}; this is the part that relies on the integration by parts. Regularity is proved in \cite[Proposition 3.9]{DS19}. Locality is proved in \cite[Proposition 3.10]{DS19} (in fact, strong locality is also proved there).


To argue that $(\mu_t^{\xi})$ is a weak solution of  \eqref{eq:symmetric-SPDE}, we need to show that the semi-martingale decomposition of the one-dimensional projection \eqref{eq:1d-projections} holds. For appropriate test functions $F$,  the additive functional \cite[Section 5.2]{FOT11} $F(\mu_t^{\xi}) - F(\mu_0^{\xi})$   admits a unique Fukushima decomposition $F(\mu_t^{\xi}) - F(\mu_0^{\xi}) = M^F_t+N^F_t$ where $M^F$,  $N^F$ are additive functionals, $M^F$ has finite energy and $N^F$ has  zero-energy (see, e.g., \cite[Theorem 5.2.2]{FOT11}). The energy of an additive functional $Y$ is (by definition) given by $\lim_{t \to 0} \frac{1}{2t}  \E_{\m} (Y_t^2)$ whenever the limit exists. We need to identify the quadratic variation of the martingale $M^F$, which we denote by $A_t^F$ and which is a Positive Continuous Additive Functional (PCAF) and the drift $N^F$. In both cases, this relies on the notion of \textit{Revuz measure} since there is a correspondence between these measures and PCAF.  So it is sufficient to show that the Revuz measures of some explicit PCAF coincide with those the quadratic variation and of the drift, which are respectively given in \cite[Theorem 5.2.3.]{FOT11} and \cite[Corollary 5.4.1.]{FOT11}. The two explicit additive functionals are respectively
$$
\tilde{A}_t^F = 2\pi^2 \int_0^t \langle \mc{D}F, \mc{D}F \rangle_{L^2(\mu_s^{\xi})} ds  \qquad \text{and} \qquad \tilde{N}_t = \int_0^t \mc{L} F(\mu_s^{\xi}) ds.
$$
Arguing along the same lines as in \cite{Albeverio-Rockner, DS19}, associated Revuz measures match hence the identification of the terms in the Fukushima decomposition.
\end{proof}

\subsubsection{Formal generator of the symmetric part}

\label{sec-formal-gen}

\paragraph{Formal expression.} We discuss here one obstacle encountered when trying to consider the symmetric part of the generator, without the formal change of variable.  For an orthonormal family  $P = \{ p_1, \dots, p_n \}$ in $L^{2}(\mb{U})$ with the Lebesgue measure, we denote by $\Pi_P$ the orthogonal projection on the span of the $p_i$'s. 

\begin{Lem}
\label{lem:ibp-sym-bon}
Suppose $F = \varphi( \int_{\mb{U}} h p_1 d\lambda, \dots, \int_{\mb{U}} h p_n d\lambda)$ and $G = \psi (\int_{\mb{U}} h q_1 d\lambda, \dots, \int_{\mb{U}} h q_m d\lambda)$ where the $p_i$'s are orthonormal in $L^2(\mb{U})$. Then, we have
\begin{equation}
\label{eq:Goal-Approximation}
\begin{split}
\int \langle DF, DG \rangle_{L^2(\mu)} d\rho  = - & \int F \left[  \sum_{j,j'} \partial_{j,j'} \psi \int_{\mb{U}} \Pi_{P}(	q_{j'}) q_j d\mu + \frac{\xi}{2\pi}  \int_{\mb{U}} \Pi_P(1) DG d\mu    \right. \\
& \left. +  \frac{1}{2\pi} \int_{\mb{U}}  \left[ \xi \E \left( \Pi_P(\partial_n H h) h \right) + \Pi_P(\partial_n H h) \right]  DG d\mu \right] d\rho.
\end{split}
\end{equation}
\end{Lem}
Before proving this lemma, note that by increasing $P$ to an orthonormal basis of $L^2(\mb{U})$ in \eqref{eq:Goal-Approximation}, one hopes to get
\begin{align*}
\int \langle DF, DG \rangle_{L^2(\mu)} d\rho  = -  \int F \left[  \sum_{j,j'} \partial_{j,j'} \psi \int_{\mb{U}} 	q_{j'} q_j d\mu + \frac{\xi}{2\pi}  \int_{\mb{U}}  DG d\mu     +  \frac{1}{2\pi} \int_{\mb{U}}  \left[ \xi \E \left( (\partial_n H h) h \right) + \partial_n H h \right]  DG d\mu \right] d\rho
\end{align*}
but it is not clear that $\int_{\mb{U}}  \left[ \xi  \E ( \partial_n H h(w) h(w) ) + \partial_nH h(w) \right] q d\mu$ makes sense for $\rho$-a.e. (but when $\xi = 0$ in which case this is well-defined when $q$ is smooth). We discuss further this term after the proof of the lemma. Finally, let us point out that an alternative proof of this fact is provided in the appendix (Section \ref{sec:GI}).

\begin{proof}
We set $S_{2}(P) := \sum_i p_i^2$, $\bar{S}_2^1(P) := \sum_i p_i \bar{p}_i$ and  we show first that
\begin{align*}
&\int \langle DF, DG \rangle_{L^2(\mu)} d\rho \\
& = \int \varphi \left[ -\sum_{j,j'} \partial_{j,j'} \psi \int_{\mb{U}} \Pi_{P}(	q_{j'}) q_j d\mu + \xi \int_{\mb{U}} (S_2(P)- \bar{S}_2^1(P)) DG d\mu - \frac{1}{2\pi} \int_{\mb{U}} \Pi_P(\partial_n H h) DG d \mu \right] d\rho
\end{align*}
By rewriting $\partial_i \varphi = D_{p_i} \varphi$ and recalling that
$$
D_k \mc{V} = - \frac{1}{2\pi} \int_{\mb{U}} k (\partial_n H h + \xi) d\lambda = - \frac{1}{2\pi} \int_{\mb{U}} h \partial_n H k d\lambda - \frac{\xi}{2\pi} \int_{\mb{U}} k d\lambda
$$
we have, by integration by parts,
\begin{align*}
\int \partial_i \varphi \partial_j \psi \int_{\mb{U}} p_i q_j d\mu d\rho  = &  \int \varphi \left[ - D_{p_i} \left[ \partial_j \psi \int_{\mb{U}} p_i q_j d\mu \right] + \partial_j \psi \int_{\mb{U}} p_i q_j d\mu D_{p_i }\mc{V} \right] d\rho \\
 = & \int \varphi \left[ - \sum_{j'} \partial_{j,j'}  \psi \int_{\mb{U}} p_i q_{j'} d\lambda \int_{\mb{U}} p_i q_j d\mu + \xi \partial_j \psi  
\int_{\mb{U}} (p_i^2 - p_i \bar{p}_i) q_j d\mu \right. \\
& \left. - \frac{1}{2\pi} \partial_j \psi \int_{\mb{U}} p_i q_j d\mu \int_{\mb{U}} h \partial_n H p_i d\lambda \right] d\rho 
\end{align*}
Since $\Pi_P(f) = \sum_i ( \int_{\mb{U}} f p_i d\lambda ) \ p_i$, summing over $i$'s gives
\begin{align*}
\sum_i \int \partial_i \varphi \partial_j \psi \int_{\mb{U}} p_i q_j d\mu d\rho  = & \int \varphi \left[ - \sum_{j'} \partial_{j,j'}  \psi \int_{\mb{U}} \Pi_P(q_j') q_j d\mu + \xi \partial_j \psi  
\int_{\mb{U}} ( S_2(P) - \bar{S}_2^1(P) )q_j d\mu \right. \\
& \left. - \frac{1}{2\pi} \partial_j \psi \int_{\mb{U}} \Pi_P(\partial_n H h) q_j d\mu \right] d\rho 
\end{align*}
and this first step follows by summing over the $j$'s.

Then, note that $ \E \left[ \Pi_P(\partial_n H h)(x) h(x) \right] = -2\pi S_2(P)(x)$. Indeed, since $\partial_n H h = - \sum_i \sqrt{2\pi \lambda_i} \ e_i X_i$ and $\Pi_P(\partial_n H h) = \sum_j \left( \int_{\mb{U}} p_j \partial_n H h d\lambda \right) p_j$, we have
\begin{align*}
 \E \left[ \Pi_P(\partial_n H h)(x) h(x) \right] & = - \sum_j \sum_i  \sqrt{2\pi \lambda_i} \int_{\mb{U}} p_j e_i d\lambda \ p_j(x) \sqrt{\frac{2\pi}{\lambda_i}} e_i(x) = - 2\pi \sum_j p_j(x)^2 = -2\pi S_2(P)(x).
\end{align*}
Finally, note that $\bar{S}_2^1(P) = \sum_i p_i \bar{p}_i = \frac{1}{2\pi} \sum \langle 1 , p_i \rangle p_i = \frac{1}{2\pi}  \Pi_P(1)$.
\end{proof}

\paragraph{Comparison with the derivative martingale.} This form, namely $\int_{\mb{U}}  \left[ \xi  \E ( \partial_n H h(w) h(w) ) + \partial_nH h(w) \right] q d\mu$ is a sort of renormalization of $(\partial_n H h) e^{-\xi h}$ and is reminiscent of the so called ``derivative martingale'' which is formally given by
\begin{equation}
\label{eq:derivative-mart}
\frac{d}{d\gamma} e^{\gamma h(x) - \frac{\gamma^2}{2} \E(h(x)^2)} = [h(x) - \gamma \E(h(x)^2)]  e^{\gamma h(x) - \frac{\gamma^2}{2}\E(h(x)^2)}.
\end{equation}
The derivative martingale was brought to the context of  LQG in \cite{Critical-lqg} in order to construct a measure for LQG with parameter $\gamma = 2$.  In our case, we have a natural martingale by considering the Fourier basis but higher modes contribute much more. In fact, as pointed out to us by Zeitouni, we can represent it as a derivative martingale. Set $\bar{\xi} := (\xi_1, \dots, \xi_N, \dots)$ and
$M_N(\bar{\xi}) := \prod_{1\leq i \leq 2N} \exp( -\xi_i X_i \frac{e_i}{\sqrt{\lambda_i}} - \frac{\xi_i^2}{2 \lambda_i} e_i^2)$. Then, for every $\bar{\xi}$, this is a martingale and
$$
\sum_i \lambda_i {\frac{d}{d \xi_i }}_{|\xi_i=\xi} M_N(\bar{\xi})  = \left( \partial_n H h_N + \xi \E(h_N \partial_n H h_N) \right) M_{N}(\xi).
$$
Indeed, note that $\lambda_i \frac{d}{d \xi_i} M_N(\bar{\xi}) = (- \sqrt{\lambda_i} X_i e_i - \xi_i e_i^2) M_N$. Then, with $h_N = \sum_{1 \leq i \leq 2N} X_i \frac{e_i}{\sqrt{\lambda_i}}$, we see that $\partial_n H h_N = - \sum_{1 \leq i \leq 2N}  \sqrt{\lambda_i} X_i e_i$ and $-\xi \E(h_N \partial_n H h_N) =  \sum_{1 \leq i \leq 2N} \xi e_i^2 $.

However, an immediate difference with the derivative martingale of LQG is that for $\gamma$ is small enough, this later measure has a finite second moment whereas this is not the case for $``(\partial_n H h)e^{-\xi h}"$. First, the  Gaussian identity \eqref{eq:Gaussian-identity-1} gives
$$
\E(\Pi_P(\partial_n H h)(w) e^{- \xi h(w)- \frac{\xi^2}{2} \E(h(w)^2) } ) = -\xi \E( \Pi_P(\partial_n H h)(w) h(w) ).
$$
Then, the  Gaussian identity \eqref{eq:Gaussian-identity-3} with $W= \Pi_P(\partial_n H h)(w) $, $Z = \Pi_P(\partial_n H h)(z)$, $X =  - \xi h(w)$, $Y= - \xi h(z)$ gives
\begin{align*}
& \E \left( \left( \int_{\mb{U}}  \left[ \xi  \E ( \Pi_P(\partial_n H h) h ) + \Pi_P(\partial_n H h) \right] q d\mu \right)^2 \right) \\
= &  \int_{\mb{U}^2} \Big( \xi^2 \E( \Pi_P(\partial_n H h)(w) h(z) ) \E( \Pi_P(\partial_n H h)(z) h(w) + \E(\Pi_P(\partial_n H h)(w) \Pi_P(\partial_n H h)(z) )   \Big) q(w) q(z)\frac{dwdz}{|w-z|^{2\xi^2}}
\end{align*}
So, increasing $P$ to an orthonormal basis would give
$$
 \int_{\mb{U}^2} \Big( \xi^2 \E( \partial_n H h(w) h(z) ) \E( \partial_n H h(z) h(w) + \E(\partial_n H h(w) \partial_n H h(z) )   \Big) q(w) q(z)\frac{dwdz}{|w-z|^{2\xi^2}}
$$
which is infinite, for every $\xi >0$, e.g. when $q=1$. 

\paragraph{Competing interests.} The authors have no competing interests to declare that are relevant to the content of this article.

\section{Appendix}

\subsection{Gaussian identities}

\label{sec:GI}

\paragraph{Gaussian integration by parts.}

If $X,X_1,\dots,X_n,Y$ are variables in a centered Gaussian space, $\psi\in C^\infty_c$, then by integration by parts, 
\begin{equation}
\label{eq:basic-ibp-1}
\E(\psi'(X)) \Var(X)=\E(X\psi(X)), \quad \E(\psi'(X)) \Cov(X,Y)=\E(\psi(X)Y)
\end{equation}
and more generally
\begin{equation}
\label{eq:basic-ibp-2}
\E(\psi(X_1,\dots,X_n)Y)=\sum_i\E(\psi_i(X_1,\dots,X_n)) \Cov(X_i,Y).
\end{equation}

\paragraph{Cameron-Martin formula.}
Recall first that if $(X,Y)$ is a Gaussian vector and $Y$ is centered, then for a smooth function $F$ with some mild growth condition, $\E(F(X) e^{Y-\E(Y^2)/2}) = \E(F(X+\Cov(X,Y)))$. This is referred to as the Cameron-Martin formula and can be interpreted by a shift of the mean of the distribution of $X$ when the reference measure is tilted by $e^{Y-\E(Y^2)/2}$. We gather here some consequences of this identity. 

Let $(X,Y,W,Z)$ be a Gaussian vector with $(X,Y)$ being centered. By applying the Cameron-Martin formula with shift $Y$, we have
\begin{equation}
\label{eq:Gaussian-identity-1}
\E( W e^{Y- \E(Y^2)/2 })  = \E(W Y).
\end{equation}
With the shift $X+Y$, we have
\begin{equation}
\label{eq:Gaussian-identity-2}
\E(W Z  e^{X-\E(X^2)/2} e^{Y-\E(Y^2)/2}) = \left( \E(WZ) + (\E(WX)+\E(WY))(\E(ZX)+\E(ZY)) \right) e^{\E(XY)}
\end{equation}
and
\begin{equation}
\label{eq:Gaussian-identity-3}
\begin{split}
\E \left( \left( W -\E(W e^{X-\E(X^2)/2}) \right) \left( Z -\E(Z e^{Y-\E(Y^2)/2}) \right)  e^{X-\E(X^2)/2} e^{Y-\E(Y^2)/2} \right)  \\
 = \left( \E(WY) \E(ZX) + \E(WZ) \right) e^{\E(X Y)}.
\end{split}
\end{equation}

\paragraph{Alternative proofs using the Cameron-Martin formula.}
 
Several computations carried out could have been done using the Cameron-Martin formula \eqref{eq:weighting-mass}. Since this way of computing is sometimes useful, but not in the spirit of most of the computations carried above, we present here some applications of this formula where we reinterpret some of our results.

\begin{Lem}
\label{Lem:tilde-DF-shift}
Consider a covariance kernel given on $\mb{U}^2$ by $G_{\partial}(w,z) = -2 \log |w-z|$ and a  bulk test functional $F = \varphi (\int h p_1 d\lambda, \dots, \int h p_n d\lambda)$. If $f(h,w) := \wt{DF}(h,w)$ then $f(h-\xi G_{\partial}(\cdot, w)) = \frac{1}{2\pi \xi} \partial_\theta F (h-\xi G_{\partial}(\cdot,w))$.
\end{Lem}
\begin{proof}
First, note that $\widetilde{\partial_n H p} = \partial_\theta p$: if $p = \sum \alpha_k e_k$, then $\widetilde{\partial_n H p} = - \sum \alpha_k \lambda_k \wt{e}_k = \sum \alpha_k \partial_\theta e_k = \partial_\theta p$. When $f(h,x) = \sum_i \partial_i \varphi( \int h p_1 d\lambda, \dots, \int h p_n d\lambda) \tilde{p}_i(x)$, introduce $P_i$ s.t. $\partial_n H P_i(x) = p_i(x)$ so that 
$$
\int G_{\partial}(\cdot, x) p_i(\cdot) d\lambda  = \int G_{\partial}(\cdot, x) \partial_n H P_i(\cdot) d\lambda = \int \partial_n H G_{\partial}(\cdot, x)  P_i(\cdot) d\lambda = -2\pi P_i(x),
$$
where we used $\partial_n H G_{\partial}(\cdot,x) = - 2\pi \delta_x(\cdot)$ to obtain the last equality. So, in this case, using $\tilde{p}_i(x) = \widetilde{\partial_n H P_i}(x) = \partial_\theta P_i(x) $ we have
\begin{align*}
f(h-\xi G_{\partial}(\cdot,x),x) & = \sum_i \partial_i \varphi( \int h p_1 d\lambda + 2\pi \xi P_1(x), \dots, \int h p_n d\lambda + 2\pi \xi P_n(x)) \tilde{p}_i(x) \\
&  = \frac{1}{2\pi \xi} \partial_\theta \varphi( \int h p_1 d\lambda + 2\pi \xi P_1(x), \dots, \int h p_n d\lambda + 2\pi \xi P_n(x))
\end{align*}
and this completes the proof.
\end{proof}

\textit{Discussion on Lemma \ref{Lem:rotational-invariance-v1}.}  We introduce and have
\begin{equation}
\label{eq:conjugate-dirichlet-form}
\mc{E}_c(F,G) := \int \langle \widetilde{DF} , \widetilde{DG}  \rangle_{L^2(\mu)} d\rho =  \int F  (- \mc{L}_c G ) d \rho
\end{equation}
where 
$$
\mc{L}_c G = \sum_{i,j} \partial_{i,j} \psi \int \tilde{q}_i \tilde{q}_j d\mu - \frac{1}{2\pi \xi} \int \partial_n H DG d\mu.
$$
Then, Lemma \ref{Lem:rotational-invariance-v1} is equivalent to $\int \mc{L}_c F d\rho = 0$ which comes from  \eqref{eq:conjugate-dirichlet-form} since then $\mc{L}_c^* = \mc{L}_c$ so $\int \mc{L}_c F d\rho = \int F \mc{L}^*1 d\rho  = 0$. \eqref{eq:conjugate-dirichlet-form} is indeed a corollary of the Cameron-Martin formula and Lemma \ref{Lem:tilde-DF-shift}:
$$
\int \langle \widetilde{DF} , \widetilde{DG}  \rangle_{L^2(\mu)} d\rho  = \frac{1}{(2\pi \xi)^2} \int \langle \partial_\theta F_x, \partial_\theta G_x \rangle_{L^2(\lambda)} d\rho = - \int \frac{1}{(2\pi \xi)^2} \langle F_x, \partial_\theta^2 G_x \rangle_{L^2(\lambda)} d\rho
$$
and
$$
\partial_\theta G_x = 2\pi \xi \sum_i (\partial_i \psi)_x \tilde{q}_i(x), \quad
\partial_\theta^2 G_x  =  (2\pi \xi)^2 \sum_{i,j} (\partial_{i,j} \psi)_x \tilde{q}_i(x) \tilde{q}_j(x)  + 2\pi \xi \sum_i (\partial_i \psi)_x \partial_\theta \tilde{q}_i(x) 
$$
so, with $\partial_\theta \tilde{q}_i = - \partial_n H q_i$,
$$
\int \langle \widetilde{DF} , \widetilde{DG}  \rangle_{L^2(\mu)} d\rho =  - \int F \left[  \sum_{i,j} \partial_{i,j} \psi \int \tilde{q}_i \tilde{q}_j d\mu - \frac{1}{2\pi \xi} \int \partial_n H DG d\mu  \right] d\rho.
$$

\textit{Discussion on Lemma \ref{lem:ibp-sym-bon}.} For $F = F(h)$, we write $F_x$ for $F(h-\xi G_{\partial}(\cdot,x))$. Assume $(p_1, \dots, p_n)$ is orthonormal in $L^2(\lambda)$. Then, for $F = \varphi (\int_{\mb{U}} h p_1 d\lambda, \dots, \int_{\mb{U}} h p_n d\lambda)$, $G = \psi (\int_{\mb{U}} h q_1 d\lambda, \dots, \int_{\mb{U}} h q_m d\lambda)$
\begin{align*}
\int \langle DF, DG \rangle_{L^2(\mu)} d\rho & = \iint_{\mb{U}}  (DF)_x (DG)_x d\lambda d\rho  = \sum_{i,j} \iint_{\mb{U}} (\partial_i \varphi)_x  (\partial_j \psi)_x p_i(x) q_j(x) d\lambda d\rho \\
& = \sum_{i,j} \iint_{\mb{U}} D_{p_i} \varphi_x (\partial_j \psi)_x p_i(x) q_j(x) d\lambda d\rho \\
& = \sum_{i,j} \iint_{\mb{U}} \varphi_x \left[- D_{p_i}(\partial_j \psi)_x - \frac{1}{2\pi} (\partial_j \psi)_x \int_{\mb{U}} ( \partial_n H h + \xi) p_i d \lambda   \right] p_i(x) q_j(x) d\lambda d\rho \\
& = \iint_{\mb{U}} \varphi_x \left[ - \sum_{j,j'} (\partial_{j,j'} \psi)_x \Pi_P(q_{j'})(x) q_j(x) - \frac{\xi}{2\pi} \Pi_P(1)(x) (DG)_x - \frac{1}{2\pi} \Pi_P(\partial_n H h)(x) (DG)_x  \right] d\lambda d\rho
\end{align*}
where the first equality uses the Cameron-Martin formula, the third one the fact that $P$ is orthonormal and the fourth one an integration by parts.

\subsection{Inverse mapping}
\label{sec:inverse}

We consider for $\xi \in (0,1)$ the GMC $e^{\xi h}$ on $\mb{U}$. This gives a coupling $(h, e^{\xi h})$ where $e^{\xi h}$ is measurable with respect to $h$. We want show the existence of a measurable map $I_{\xi}$ such that $\rho$-a.e., $h = I_{\xi}(e^{\xi h})$. The proof in \cite{BSS14} can be adapted for the log-correlated Gaussian field $h$ on $\mb{U}$, in particular by relying on the detailed study of the field $h$ and its chaos measures in \cite{welding}. We just sketch the main arguments below. First, set 
\begin{equation}
B_{\eps}(x) := \{ e^{i y} : y \in (x-\eps,x+\eps) \} \qquad \text{and} \qquad e^{\xi h^{\eps}(e^{ix})}:= \mu_{\xi}(B_\eps(x))
\end{equation}
so $h^{\eps}(x) = \xi^{-1} \log \mu_{\xi}(B_\eps(x))$. Then,  the following convergence occurs in $L^2$
$$
\int h p d\lambda = \lim_{\eps \to 0} \int (h^{\eps}(w) - \E h^{\eps}(w) ) p(w) d\lambda
$$
To justify it, consider an approximation $h_{\eps}(x)$ at space-scale (e.g., $h_{\eps} := \langle h, \rho_{\eps}^x \rangle$ for a mollification to be specified) and set $f_{\eps}(x) := h^{\eps}(x)-h_{\eps}(x)$. For the convergence to occur, it is sufficient to have the following pointwise estimates: there exists $\alpha > 0$ and $0< \kappa < 1$ such that, uniformly in $\eps \in (0,1/4)$,
\begin{align}
\label{eq:pointwise-var}
& \Var f_{\eps}(x)  \leq C \log \eps^{-1} \\
\label{eq:pointwise-cov}
& | \Cov f_{\eps}(x), f_{\eps}(y) | \leq C \eps^{\alpha} \qquad \text{for } |x-y| > \eps^{\kappa} 
\end{align}
Indeed, by splitting $\mb{U}^2$ in points $\{|w-z| < \eps^{\kappa} \}$ and its complement, using the variance bound in the former case and the covariance one in the latter case, we get $\Var \int p(e^{ix}) f_{\eps}(x) dx \leq C( \eps^{\kappa} \log \eps^{-1} + \eps^{\alpha} )$. The desired convergence follows then from the one of $\int_{\mb{U}} h_\eps p d\lambda$.

We start with \eqref{eq:pointwise-var}. Consider the GMC $\eta = e^{\xi \psi}$ associated with the field $\psi$ with covariance $\E(\psi(x) \psi(y)) = - \log (|x-y| \wedge 1)$ on $\mb{R}$. When fixing an interval $I$ of length $< \pi$, by Lemma 3.6 and equation (53) of \cite{welding}, there exists a coupling of $\mu$ and $\eta$ with a random variable $X$ having Gaussian tails such that for any interval $B\subset I$,  $e^{-X} \eta(B) \leq \mu(B) \leq \eta(B) e^X$. In particular, $\Var \log \mu(B_{\eps}(x)) \leq 2 \E(X^2) + 2 \Var \log \eta (B_{\eps}(x))$ . The pointwise variance estimate \eqref{eq:pointwise-var} follows from the exact scaling relation (54) in \cite{welding} and from the upper bound on the pointwise variance of the mollification of a log-correlated field (which can be found in \cite{Berestycki17}) for $\Var h_{\eps}(x)$.

We sketch here the main ideas to get the covariance bound \eqref{eq:pointwise-cov}. Use a white-noise representation of the field to split it into two parts, a fine field whose restriction on $B_{\eps}(x)$ and $B_{\eps}(y)$ are independent and a coarse field, which is independent. This is useful to obtain the exact decorrelation of measure on these sets. In \cite{welding}, the field $h$ is represented with a white-noise on $\mb{U} \times \mb{R}^+$ (up to an independent additive  constant) where the $y$-axis represents the spatial scale, so a natural way to consider a coarse field and a fine field is to split horizontally the domain of the white-noise ($\{ y < \eps^{\kappa} \}$ and its complement). Doing so, we write $h = h_{(0,\eps^{\kappa})} + h_{(\eps^{\kappa},\infty)}$.
Then, note that the oscillation of the coarse field on a microscopic ball is of order $\osc_{B_{\eps}(x)} h_{(\eps^{\kappa},\infty)} = O(\eps^{1-\kappa})$ so that 
$$
\mu_h (B_{\eps}(x)) = \mu_{h_{(0,\eps^{\kappa})}}(B_{\eps}(x)) e^{\xi h_{(\eps^{\kappa},\infty)}(x)} e^{O(\eps^{1-\kappa})}
$$
and
\begin{align*}
\xi^{-1} \log \mu_h (B_{\eps}(x)) - h_{\eps}(x) & = \xi^{-1} \log \mu_{h_{(0,\eps^{\kappa})}} (B_{\eps}(x)) + h_{(\eps^{\kappa},\infty)}(x) - h_{\eps}(x) + O(\eps^{1-\kappa}) \\
& = \left( \xi^{-1} \log \mu_{h_{(0,\eps^{\kappa})}} (B_{\eps}(x))  - h_{(\eps,\eps^{\kappa})}(x) \right) + \left( h_{(\eps,\infty)}(x) - h_{\eps}(x)  \right) + O(\eps^{1-\kappa}) \\
& = M_x + F_x + R_x
\end{align*}
so that when $|x-y| > \eps^{\kappa}$, $M_x$ and $M_y$ are independent. We need pointwise variance upper bounds for $M_x$ and $F_x$. For $M_x$ this uses the same ideas as for the bound \eqref{eq:pointwise-var}. For $F_x$, note that we have the decomposition in independent terms $h_{\eps}(x) - h_{(\eps,\infty)}(x)  = \langle h_{(0,\eps)}, \rho_{\eps}^x \rangle + \langle h_{(\eps,\infty)}, \rho_{\eps}^x - \delta_x \rangle$. Now, with $\rho_{\eps}^x(y) := \frac{1}{\pi} \sum_{k = 1}^{\eps^{-1}} \cos(k(x-y))$, $h_{(0,\eps)}$ is essentially orthogonal in $L^2$ to $\rho_{\eps}^x(y)$ and  $\delta_x(y)-\rho_{\eps}^x(y)  = \frac{1}{\pi} \sum_{k > \eps^{-1}} \cos(k(x-y))$ to  $h_{(\eps,\infty)}$ and one can get a polynomial upper bound on the variance.

\subsection{Comparison with the QLE generator}
\label{sec:comparison-qle}

The article \cite{MS16} develops a formal SPDE satisfied by the QLE process, using the Brownian motions driving the SLEs. We introduce the notation used by the authors and translate their results in our notation to compare the generator. They work with a boundary probability measure $\nu_t$ instead of $\mu_t$ (a general Borel measure without mass constraint) and with a normalization of harmonic functions such that  $h_t(0) = 0$. They argue that the dynamics of the harmonic functions $(h_t)$ are formally given by 
\begin{equation}
\label{eq:QLE-MS-dynamics}
\dot{h}_t(z) = \int_{\mb{U}} \left( D_t(z,u) + \mc{P}^{\star}(z,u) W(t,u) \right) d\nu_t(u)
\end{equation}
where
\begin{align*}
D_t(z,u) & = - \nabla h_t(z) \cdot \Phi(u,z) + \frac{1}{\sqrt{\kappa}} \mc{P}^{\star}(z,u) + Q (\partial_{\theta} \overline{\mc{P}})(z,u) \\
& =  - \nabla h_t(z) \cdot \Phi(u,z) + \xi \mc{P}^{\star}(z,u) + ( 2\xi+ \frac{1}{2\xi}) (\partial_{\theta} \overline{\mc{P}})(z,u)
\end{align*}
where $\Phi(u,z) = -z \frac{z+u}{z-u}$, for $a,b \in \mb{C} = \mb{R}^2$, $a \cdot b = \Re(\bar{a} b),$ so $\nabla h_t(z) \cdot \Phi (u,z) = 2 \overline{\partial_{z} h_t} \cdot \Phi (u,z) = 2 \Re ( \Phi(u,z) \partial_z h_t )$ and
\begin{equation}
D_t(z,u) =  - 2 \Re ( \Phi(u,z) \partial_z h_t ) + \xi \mc{P}^{\star}(z,u) + ( 2\xi+ \frac{1}{2\xi}) (\partial_{\theta} \overline{\mc{P}})(z,u)
\end{equation}
Above $\mc{P}$ is $2\pi$ times the Poisson kernel on $\mb{D}$, i.e. $\mc{P}(z,w) = 2\pi H(z,w) = \Re \left( \frac{w+z}{w-z} \right)$, $\overline{\mc{P}}$ is $2\pi$ the conjugate Poisson kernel on $\mb{D}$,  i.e.  $\overline{\mc{P}}(z,w) = \Im \left( \frac{w+z}{w-z} \right)$ so $\partial_{\theta} \overline{\mc{P}}(z,w) = \Re \left( \frac{-2zw}{(w-z)^2} \right)$ and $\mc{P}^{\star} = \mc{P}-1$ so $\mc{P}^{\star}(0,w) = 0$ for all $w \in \mb{U}$. Finally, $W$ is a space-time white noise on $\mb{U} \times [0,\infty)$. Furthermore, above we already specified $\xi = \frac{1}{\sqrt{\kappa}} = \frac{1}{\sqrt{6}}$ and $Q = (2\xi)^{-1}+2\xi = \frac{5}{\sqrt{6}}$.

\bigskip

We look at $d \int h_t f d\lambda$ for $f \in C_c^{\infty}(\mb{D})$. By Fubini, the part without white-noise contains three terms: for the first one, we integrate it over $d\nu$,
$$
(i)=- 2 \iint  \Re ( \Phi(u,z) \partial_z h_t ) f(z) dz \nu(du)  = - 2 \iint  h_t(z) \Re (  \partial_z  \left( \Phi(u,z) f(z) \right) dz  \nu(du)  = \int h_t(z) (D_{\nu}f)(z) dz
$$
using Fubini and recalling the expression \eqref{eq:dmu-expressions} of $D_\nu f$ for the last equality.
$$
(ii)=\xi \int \mc{P}^{\star}(z,u) f(z) dz = \xi \int (\mc{P}(z,u)-1) f(z) dz = 2\pi \xi \int H(z,u) f(z) dz - \xi \int f d\lambda = 2\pi \xi ( f^{*}(u) - \dashint f d\lambda )
$$
$$ 
(iii)=Q \int \partial_{\theta} \overline{\mc{P}}(z,u)  f(z) dz = Q \int \Re \left( \frac{-2zw}{(w-z)^2} \right) f(z) dz = 2 \pi Q (\partial_n H f^*)(u)
$$
using \eqref{eq:intermediate-qle} for the last equality.

The part with white-noise becomes
$$
\int \mc{P}^{\star}(z,u) W(t,u) f(z) dz  =  2\pi  (f^{*}(u)   -  \dashint f d\lambda ) W(t,u)
$$

For $F(h) = \int h f d\lambda = \int H h f d\lambda = \int h f^* d\lambda_{\partial}$, $DF = f^*$,
\begin{align*}
d \int h_t f d\lambda = &  \int h_t(z) (D_{\nu_t}f)(z) dz + 2\pi Q \int \partial_n H f^*  d\nu_t + 2 \pi \xi \int (f^*-\dashint f^* d\lambda ) d\nu_t   \\
& +  2\pi \xi \int ( f^{*} - \dashint f^* d\lambda  )W_t d\nu_t  \\
= &   \mc{L} F(h_t) dt + d \mathrm{Mart}(f)
\end{align*}
where $\mc{L}F(h) = \int h (D_{\nu}f) d\lambda + 2\pi Q \int \partial_n H DF d\nu + 2\pi \xi \int (DF - \dashint DF) d\nu $

Now, we compute $\mc{L}F$ for $F = \varphi (\int h f_1 d\lambda, \dots, \int h f_n d\lambda)$. It is an application of It\^o's formula. The quadratic variation is given by
$$
d \langle \mathrm{Mart}(f), \mathrm{Mart}(g) \rangle_t = (2\pi)^2 \int (f^*-\dashint f^*)(g^*-\dashint g^*) d\nu_t
$$
so, writing $p_i = f_i^*$, 
$$
\mc{L}F(h)  = \sum_{i} b(p_i) \partial_i \varphi + \frac{1}{2} \sum_{i,j} \sigma(p_i, p_j) \partial_{i,j} \varphi 
$$
where
$$
b(p) =  \int h(z) (D_{\nu}f)(z) dz + 2\pi Q \int \partial_n H p  d\nu + 2 \pi \xi \int (p-\dashint p d\lambda ) d\nu
$$
and $\sigma(p,q) = (2\pi)^2 \int (p-\dashint p)(q-\dashint q) d\nu$.

\bibliographystyle{abbrv}
\bibliography{biblio}

\end{document}